\documentclass[11pt]{article}
\DeclareMathAlphabet{\mathpzc}{OT1}{pzc}{m}{it}
\usepackage{amssymb,amsmath,amsthm}
\usepackage{extarrows}
\usepackage{enumerate,multicol}
\usepackage[capposition=below]{floatrow}
\usepackage{fullpage}
\usepackage{graphicx}
\usepackage[mathcal]{euscript}
\usepackage{mathrsfs}
\usepackage{upgreek}
\usepackage{MnSymbol}
\usepackage{marginnote}
\usepackage{tikz,tikz-cd}
\usepackage{tikz-3dplot}
\usepackage{comment}
\usepackage{pgfplots}
\usepackage{subfig}
\usepackage{hyperref}
\usepackage{musicography}
\usepackage{breakcites}

\usepackage[backend=biber,bibstyle=authoryear,citestyle=authoryear,
natbib=true,sorting=nyt,sortcites=true,maxnames=10,minnames=10,
hyperref=true,abbreviate=false,date=long,isbn=true,eprint=true,
uniquename=init,giveninits=true
]{biblatex}
\bibliography{Bibliography}

\pgfplotsset{compat=1.15}
\usepgfplotslibrary{fillbetween}
\usetikzlibrary{matrix,arrows}
\newtheorem{thm}{Theorem}[section]

\theoremstyle{definition}
\numberwithin{equation}{section}
\newtheorem{cor}[thm]{Corollary}
\newtheorem{lem}[thm]{Lemma}

\newtheorem{defin}[thm]{Definition}
\newtheorem{exam}[thm]{Example}
\newtheorem{propo}[thm]{Proposition}

\newtheorem{sublem}{Sublemma}

\newcommand{\R}{\mathbb{R}}
\renewcommand{\d}{\ensuremath{\operatorname{d}\!}}
\newcommand{\overbar}[1]{\mkern 1.5mu\overline{\mkern-1.5mu#1\mkern-1.5mu}\mkern 1.5mu}
\newcommand{\localflow}{\mathscr{LF}}
\newcommand{\suchthat}{\;\ifnum\currentgrouptype=16 \middle\fi|\;}
\DeclareMathAlphabet\mathbfcal{OMS}{cmsy}{b}{n}
\tikzset{
  subseteq/.style={
    draw=none,
    edge node={node [sloped, allow upside down, auto=false]{$\subseteq$}}}}
\hypersetup{colorlinks,
            allcolors = blue}

\date{}

\begin{document}

\renewcommand\qedsymbol{$\blacksquare$}
\title{The exponential map for time-varying vector fields}
\author{Yanlei Zhang\footnote{Graduate student, Department of Mathematics and Statistics, Queen’s University, Kingston, ON K7L 3N6, \hspace*{1.5em} Canada. \newline \hspace*{1.8em}Email:\href{mailto:yanlei.zhang@queensu.ca}{yanlei.zhang@queensu.ca}}}
\maketitle

\begin{abstract}
The exponential map that characterises the flows of vector fields is the key in understanding the basic structural attributes of control systems in geometric control theory. However, this map does not exists due to the lack of completeness of flows for general vector fields.

An appropriate substitute is devised for the exponential map, not by trying to force flows to be globally defined by any compactness assumptions on the manifold, but by categorical development of spaces of vector fields and flows, thus allowing for systematic localisation of such spaces. That is to say, we give a presheaf construction of the exponential map for vector fields with measurable time-dependence and continuous parameter-dependence in the category of general topological spaces. Moreover, all manners of regularity in state are considered, from the minimal locally Lipschitz dependence to holomorphic and real analytic dependence. Using geometric descriptions of suitable topologies for vector fields and for local diffeomorphisms, the homeomorphism of the exponential map is derived by a uniform treatment for all cases of regularities. Finally, a new sort of continuous dependence is proved, that of the fixed time local flow on the parameter which plays an important role in the establishment of the homeomorphism of the exponential map.

\vspace{5pt}

\vspace{5pt}
\end{abstract}
\newpage
\begingroup
  \hypersetup{hidelinks}
  \tableofcontents
\endgroup  

\section{Introduction}
In geometric control theory, one can study linear and nonlinear control theory from the point of view of applications, or from a more fundamental point of view where the system structure is a key element. It is well understood that the language of systems such as I am interested in should be founded in the study of differential geometry and vector fields on manifolds \citep{MR2062547,MR1410988,MR1978379,MR1425878,MR2099139}.

Understanding the basic structural attributes of control systems requires understanding how the system is connected to the trajectories of the system. For a control system 
$$\xi'(t)=F(t,\xi(t),\mu(t))$$
with a control $\mu:\mathbb{T}\rightarrow \mathcal{C}$ and trajectory $\xi:\mathbb{T}\rightarrow M$, the controllability, reachability, and stabilisability properties depend solely on the understanding a family of flows
$$(t,t_0,x_0)\mapsto\Phi^F(t,t_0,x_0,\mu),\hspace{10pt} \mu\in\mathscr{U}$$
for some class of controls $\mathscr{U}$. This is typically thought of as the image of the family of time-varying vector fields $(t,x)\mapsto F_\mu(t,x):=F(t,\xi(t),\mu(t))$ under some “exponential map” 
\begin{eqnarray*}
    \text{exp}:\{\text{vector fields}\}&\rightarrow& \{\text{local flows}\}\\
	  F&\mapsto& \Phi^F.
\end{eqnarray*}
The idea of the flow of a vector field seems so well understood that it barely merits any systematic explication. However, the accepted casual manner of this presentation has a deficiency, this being that there is no way to generally define the exponential map as a mapping from the Lie algebra of vector fields to the group of diffeomorphisms. Quite apart from any technical difficulties that arise from working with infinite-dimensional manifolds, the incompleteness of general vector fields causes any na\"ive definition to fail. 

One might overcome this by working with compact manifolds or by working with only completevector fields. Particularly, the assumption of completeness is one that is very often made in passing ``for the sake of convenience." For manifolds that are compact, the problem of completeness can be overcome, and the desired exponential map, in fact, exists \citep{MR0271983}. For noncompact manifolds, one can work with vector fields with compact support. These compactness assumptions (for compact manifolds) or impositions (for compact support) are not satisfactory. For example, the compactness assumption fails for linear ordinary differential equations, and any theory not including these can hardly be said to be general. As another instance of the lacking of these compactness constructions, note that a real analytic vector field on a noncompact manifold can never have compact support. From our perspective, any theory not including analytic vector fields is not satisfactory.

Another route to some sort of exponential map in the time-varying case involves coming up with some sort of series representation for time-varying flows. The so-called Volterra series is an adaptation of the exponential series for time-varying vector fields, and rigorous versions of this work date back to \citep{Agra_ev_1979}. A nice recent summary of this work can be found in the book of \citep{MR2062547}. The “inversion” of the Volterra series leads to the notion of a Baker–Campbell–Hausdorff formula in the time-varying case. An example of this can be found in the work of \citep{MR886816}. These sorts of considerations have given rise to an area dedicated to using methods from the theory of free algebras. For a recent outline of these methods, we refer to \citep{Kawski}. While these techniques have proved to have significant value in geometric control theory, the problem of defining the exponential map for general vector fields still remains open. 

The drawbacks of these approaches on this subject comes in various forms. One such drawback concerns regularity of the vector fields, and definedness and convergence of the series. The series from this theory are comprised of differential operators that arise from iterating first-order differential operators, i.e., vector fields. In degenerate cases where some sort of nilpotency can be assumed, only finite iterations are necessary. However, one cannot expect this to be the case generically, and so a general theory in this
framework must allow for infinite iterations of first-order differential operators, i.e., at least infinite differentiability. Thus the theory simply does not apply to lower degrees of regularity. Moreover, even in the infinitely differentiable setting, the series do not converge in any meaningful way. This is not surprising since there are some aspects of Taylor series in these series representations, and so one expects, and it indeed the case, that real analyticity is required for convergence. Again, lesser regularity is simply not
represented by these series methods (and it is certainly not claimed to be represented in the literature on the subject). 

Another limitation of the series representations of flows is that they are only made for a single time-varying vector field, whereas any sort of “exponential map” should give us some representation of a flow given any time-varying vector field. The problem here comes in two flavours. First, the completeness problem mentioned above appears in the series representations as well; the domain in space and time simply cannot be uniform over all vector fields. Second, if one is working with real analytic vector fields and asking for convergence of series, the region of convergence will depend on the specific vector field \citep[Example 6.24]{Jafarpour2014}. In any case, the lack of a fixed domain for flows and convergence creates a problem for series representation methods as a means for defining any general sort of exponential map.

This, then, leaves open the question of how one satisfactorily addresses the problem of defining the exponential map in any general way.

\subsection{Contribution of the paper}
The general question about the existence of the exponential map is addressed in this paper by considering, not vector fields and diffeomorphisms, but presheaves of vector fields and groupoids of local diffeomorphisms of all possible regularities, which allow for systematic localisation of the components of what will become the exponential map, i.e.,
$$\text{exp}:\{\text{presheaf of vector fields}\}\rightarrow\{\text{groupoid of local diffeomorphisms}\}.$$

Moreover, we present a methodology for working with vector fields with measurable time-dependence, continuous parameter-dependence with parameters in an arbitrary topological space, and for working with the resulting flows of such vector fields. The presheaf point of view provides a theory that integrates the fact that flows are only locally defined, even for vector fields that are globally defined. That is to say, we deal with the lack of completeness not by trying to force flows to be globally defined by some sort of assumption, but by making vector fields themselves locally defined, thus putting them on the same local footing as their flows. The categorical framework for talking about time-dependent vector fields and local flows herein allows one to infer the existence of a presheaf in the category of topological spaces, which plays an important role in the study of controllability of a control system.  

Another attribute of the framework is that time-varying vector fields are considered with a variety of degrees of regularity with respect to state, namely, Lipschitz, finitely differentiable, smooth, real analytic, and holomorphic. In doing this, use is made of locally convex topologies for the spaces of vector fields with these degrees of regularity. The classes of time-varying and parameter-dependent vector fields are characterised by their continuity, measurability, and integrability with respect to these locally convex topologies. Within this framework, very general are provided results concerning existence, uniqueness, and regular dependence of flows on initial and final time, initial state, and parameters. These results include, and drastically extend, known results for properties of flows.

Moreover, the homeomorphism of this map is established upon the suitable topologies for sets of vector fields and flows using geometric decompositions of various jet bundles by various of connections. This framework is interesting in that it allows an elegant and uniform treatment of vector fields across various regularity classes.

\subsection{An outline of the paper}\label{sec:1.2}
Roughly speaking, in the next two sections we develop the classes of vector fields and flows we use, and attributes of these. In the final section, we prove properties of flows of vector fields, and define the exponential map and its properties.

In more details, Section \ref{sec:2} overviews the locally convex topologies for the space of sections of a vector bundle for various regularity classes presented in \citep{Jafarpour2014}, including the newly arrived topology in the real analytic case. Most importantly, we use the these locally convex topologies to describe classes of time-dependent and parameter-dependent sections. The approach we take strictly extends the usual approach to parameter-dependence in the theory of ordinary differential equations, and allows, for example, vector fields that depend on a parameter in a general topological space.

In Section \ref{sec:3}, we carefully and geometrically establish the basic results concerning existence and uniqueness of integral curves, and of the regular dependence of flows on initial time, final time, initial state, and parameter. We make a comment that the new type of continuity result for the “parameter to local flow” mapping provides a geometric toolbox for dealing with the homeomorphism of the exponential map which will be established in Section \ref{sec:4}.

In Section \ref{sec:4}, we carefully establish the presheaf exponential map in the category of topological spaces. We start by developing a categorical framework for talking about time-dependent vector fields and local flows, which allow one to infer the existence of a presheaf, here in the category of topological spaces. The presheaves we construct are put together by requiring that, in any product neighbourhood of a point in time, state, and parameter, the theory should agree with the “standard” theory of Section \ref{sec:3.1} and Section \ref{sec:3.2}. Since the collection of product neighbourhoods are a basis for the open sets in the product, standard presheaf theory constructions then give a presheaf whose local sections over products agree with the prescribed ones. Essentially by taking inverse limits in these appropriate categories, we show that the theory of local flows is very elegantly represented by the existence of an “exponential mapping” from the presheaf of vector fields to the presheaf of flows. Moreover, under the appropriate topologies we developed for time-dependent vector fields and the topology for local flows herein,  this exponential map can be shown to be an homeomorphism onto its image by the universal property of the inverse limit. 
\subsection{Background and notation}
We shall give a brief outline of the notations we use in the main body of the paper. We shall mainly give definitions, establish the bare minimum of facts we require, and refer the reader to the references for details.

\vspace{10pt}
\noindent\textbf{Manifolds, vector bundles, and jet bundles.} 
We shall assume all manifolds to be e Hausdorff, second countable, and connected. We shall will work with manifolds and vector bundles coming from the different categories: smooth (i.e., infinitely differentiable), real analytic, and holomorphic (i.e., complex analytic). We shall use ``class $C^r$" to denote these three cases, i.e., $r\in\{\infty,\omega,\text{hol}\}$ for smooth, real analytic, and holomorphic, respectively. When $r = \text{hol}$, we shall frequently ask that $M$ be a Stein manifold; this means that there is a proper embedding of $M$ in $\mathbb{C}^N$ for a suitable $N \in \mathbb{Z}_{>0}$ \citep{MR71085}. A typical $C^r$-vector bundle we will denote by $\pi : E \to M$. The dual bundle we denote by $E^*$. The tangent bundle we denote by $\pi_{TM} : TM \to M$ and the cotangent bundle by $\pi_{T^\ast M} : T^\ast M \to M$.

While manifolds and vector bundles are smooth, real analytic, or holomorphic, we will work with other sorts of geometric objects, e.g., functions, mappings, sections, that have various sorts of regularity. Let us introduce the terminology we shall use. Let $m\in\mathbb{Z}_{\geq 0}$ and let $m'\in \{0,\text{lip}\}$. We will work with objects with regularity $\nu\in \{m+m',\infty,\omega,\text{hol}\}$. Thus $\nu=m$ means ``$m$-times continuously differentiable," $\nu=\infty$ means ``smooth", $\nu=\omega$ means ``real analytic", and $\nu=\text{hol}$ means ``holomorphic." Given $\nu\in \{m+m',\infty,\omega,\text{hol}\}$, we shall often say ``let $r\in\{\infty,\omega,\text{hol}\}$, as required." This has obvious meaning that $r=\text{hol}$ when $\nu=\text{hol}$, $r=\omega$ when $\nu=\omega$, and $r=\infty$ otherwise. We shall also use the terminology ``let $\mathbb{F}\in\{\R,\mathbb{C}\}$, as appropriate." This means that $\mathbb{F}=\mathbb{C}$ when $r=\text{hol}$ and $\mathbb{F}=\R$ otherwise.

If $m \in \mathbb{Z}_{\geq 0}$ and $m' \in \{0, \text{lip}\}$, and if $\nu \in \{m+m', \infty, \omega, \text{hol}\}$, then the $C^\nu$-sections of $E$ are denoted by $\Gamma^\nu(E)$. By $C^\nu(M)$ we denote the set of $C^\nu$-functions on $M$, noting that these are $\mathbb{F}$-valued, i.e., $\mathbb{C}$-valued when $\nu = \text{hol}$. If $M$ and $N$ are $C^r$-manifolds, $C^\nu(M; N)$ denotes the set of $C^\nu$-mappings from $M$ to $N$. If $f \in C^{\nu+1}(M)$ and $X \in \Gamma^\nu(TM)$, we denote by $\mathscr{L}_Xf$ or $Xf$ the Lie derivative of $f$ with respect to $X$. If $\Phi \in C^1(M; N)$, $T\Phi: TM \to TN$ denotes the derivative of $\Phi$.

For a $C^r$-vector bundle $\pi : E \to M$, $r \in \{\infty, \omega,\text{hol}\}$, we denote by $\pi_m : J_mE \to M$ the vector bundle of $m$-jets of sections of $E$; see \citep[\S 12.17]{MR1202431} and \citep{MR989588}. For $C^r$-manifolds $M$ and $N$, we denote by $\rho^m_0: J^m(M; N) \to M \times N$ the bundle of $m$-jets of mappings from $M$ to $N$. We also have a fibre bundle 
$$\rho_m \triangleq \text{pr}_1\circ \rho^m_0: J^m(M; N) \to M,$$
where $\text{pr}_1$ is the projection onto the first component. We signify the $m$-jet of a section, function, or mapping by use of the prefix $j_m$, i.e., $j_m\xi$, $j_mf$, or $j_m\Phi$. The set of jets of sections at $x$ we denote by $J^m_x E$ and the set of jets of mappings at $(x, y) \in M \times N$ we denote by $J^m(M; N)_{(x,y)}$.  As a special case, $J^m(M; \R)$ denotes the bundle of $m$-jets of functions. We denote by $T^{*m}_x M = J^m(M; \R)_{(x,0)}$ the jets of functions with value 0 at $x$, and $T^{*m} M = \cup_{x\in M}T^{*m}_x M$. The space $T^{*m}_x M$ has the structure of a $\R$-algebra specified by requiring that
$$\mathbf{m}^r_x\ni  f \mapsto j_mf(x) \in T^*m_x M$$
be a $\R$-algebra homomorphism, with $\mathbf{m}^r_x \subseteq C^r(M)$ being the ideal of functions vanishing at $x$. We then note \citep[Proposition 12.9]{MR1202431}  that $J^m(M; N)_{(x,y)}$ is identified with the set of $\R$-algebra homomorphisms from $T^{*m}_y N$ to $T^{*m}_x M$ according to
$$j_m\Phi(x)(j_m g(y)) = j_m(\Phi^*g)(x)$$
for $\Phi$ a smooth mapping defined in some neighbourhood of $x$ and satisfying $\Phi(x) = y$. We note that $\Gamma^\nu(J^mE)$ can be thought of in the usual way since $\pi_m : J^mE \to M$ is a $C^r$-vector bundle. However, $J^m(M; N)$ is not, generally, a vector bundle; nonetheless, we shall denote by $\Gamma^\nu(J^m(M; N))$ the set of $C^\nu$-sections of the bundle $\rho_m :J^m(M; N) \to M$.

\vspace{10pt}
\noindent\textbf{Metrics and connections.} We shall make use of Riemannian and fibre metrics for the same reason of convenience. Many definitions we make use a specific choice for such metrics, although none of the results depend on these choices. For $r \in \{\infty, \omega\}$, we let $\pi : E \to M$ be a $C^r$-vector bundle. We denote by $\mathbb{G}_M$ a $C^r$-Riemannian metric on $M$ and by $\mathbb{G}_\pi$ a $C^r$-metric for the fibres of $E$. We make a note that the existence of these in the real analytic case is verified by \citep[Lemma 2.4]{Jafarpour2014}. The metrics $\mathbb{G}_M$ and $\mathbb{G}_\pi$ then induce metrics in all tensor products of $TM$ and
$E$ and their duals. For simplicity, we just denote any such metric by $\mathbb{G}_{M,\pi}$. 

We will frequently make use of the distance function on M associated with a Riemannian metric $\mathbb{G}$. In order to have constructions involving G make sense—in terms of not depending on the choice of Riemannian metric—we should verify that such constructions do not depend on the choice of this metric. Of course, this is not true for all manner of general assertions. However, Lemma \ref{lem:5.2} captures what we need. This lemma will not surprise most readers, but we could not find a proof of this anywhere. 

For convenience we shall make use of connections in representing certain objects that do not actually require a connection for their description. For $r \in \{\infty, \omega\}$, we let $\pi : E \to M$ be a $C^r$-vector bundle. We let $\nabla^M$ be a $C^r$-affine connection on $M$ and we let $\nabla^\pi$ denote a $C^r$-linear connection in the vector bundle. The existence of these on the real analytic case is proved by \citep[Lemma 2.4]{Jafarpour2014}. Almost always we will not require $\nabla^M$ to be the Levi-Civita connection for the Riemannian metric $\mathbb{G}_M$, nor do we typically require there to be any metric relationship between $\nabla^\pi$ and $\mathbb{G}_\pi$. However, in our constructions for the Lipschitz topology, it is sometimes convenient to assume that $\nabla^M$ is the Levi-Civita connection for $\mathbb{G}_M$ and that
$\nabla^\pi$ is $\mathbb{G}_\pi$-orthogonal, i.e., parallel transport consists of inner product preserving mappings. Thus, a safety-minded reader may wish to make these assumptions in all cases.
\section{Space of sections}\label{sec:2}
We present a methodology for working with vector fields with measurable time-dependence and for working with the resulting flows of such vector fields. We begin in this section by characterizing time-varying vector fields on manifolds using locally convex topologies generated by a family of seminorms. This presentation of time-varying vector fields agrees with, and extends, the standard treatments. It closely follow from \citep{Jafarpour2014}.
\subsection{Measurable and integrable funtions with values in locally convex topological vector spaces}
The topological characterisation relies on notions of measurability, integrability, and boundedness in the locally convex spaces $\Gamma^\nu(E)$. For an arbitrary locally convex space $V$, let us review some definitions.
\begin{enumerate}
    \item[(1)]{
    Let $(\mathcal{M},\mathscr{A})$ be a measurable space. A function $\Psi:\mathcal{M}\rightarrow V$ is ``measurable" if $\Phi^{-1}(\mathcal{B})\in \mathscr{A}$ for every Borel set $\mathcal{B}\subseteq V$.
    }
    \item[(2)]{
    It is possible to describe a notion of integral, called the ``Bochner integral", for a function $\gamma:\mathbb{T}\rightarrow V$ that closely resembles the usual construction of the Lebesgue integral. A curve $\gamma:\mathbb{T}\rightarrow V$ is ``Bochner integrable" if its Bochner integral exists and is ``locally Bochner integrable" if the Bochner integral of $\gamma|\mathbb{T'}$ exists for any compact subinterval $\mathbb{T'}\subseteq\mathbb{T}$.
    }
    \item[(3)]{
    Finally, a subset $\mathcal{B}\subseteq V$ is bounded if $p|\mathcal{B}$ is bounded for any continuous seminorm $p$ on $V$. A curve $\gamma:\mathbb{T}\rightarrow V$ is ``essentially von Neumann bounded" if there exists a bounded set $\mathcal{B}$ such that 
    $$\lambda(\{t\in \mathbb{T}\;|\;\gamma(t)\notin\mathcal{B} \})=0,$$
    and is ``locally essentially von Neumann bounded" if $\gamma|\mathbb{T'}$ is essentially von Neumann bounded for every compact subinterval $\mathbb{T'}\subseteq\mathbb{T}$.
	}
\end{enumerate}
\subsection{Topologies on space of sections}
In this section we will provide explicit seminorms that define the various topologies we use for local sections, corresponding to regularity classes $\nu\in \{m+m',\infty,\omega,\text{hol}\}$. We shall not use much space to describe the nature of these topologies, but give a general sketch and refer the interested readers to  \citep{Jafarpour2014} for more details.

\subsubsection{Locally Lipschitz sections of vector bundles}
As we are interested in ordinary differential equations with well-defined flows, we must, according to the usual theory, consider locally Lipschitz sections of vector bundles. In particular, we will find it essential to topologise the space of locally Lipschitz sections of $\pi : E \to M$. To define the seminorms for this topology, we make use of a ``local least Lipschitz constant."

We let $\xi : M \to E$ be such that $\xi(x) \in E_x$ for every $x \in M$. For a piecewise differentiable curve $\gamma : [0, T] \to M$, we denote by $\tau_{\gamma,t} : E_{\gamma(0)} \to E_{\gamma(t)}$ the isomorphism of parallel translation along $\gamma$ for each $t \in [0, T]$. We then define, for $K \subseteq M$ compact,
$$l_K(\xi)=\sup\left\{\frac{||\tau^{-1}_{\gamma,1}(\xi\circ\gamma(1))-\xi\circ\gamma(0)||_{\mathbb{G}_\pi}}{\ell_{\mathbb{G}_M}(\gamma)}\ \suchthat\ \gamma:[0,1]\to M,\  \gamma(0), \gamma(1) \in K,\  \gamma(0) \neq \gamma(1)\right\},$$
which is the \textbf{K-sectional dilatation of} $\xi$. Here $\ell_{\mathbb{G}_M}$ is the length function on piecewise differentiable curves. We also define
\begin{eqnarray*}
  \text{dil}\ \xi: M &\to& \R_{\geq 0}\\
  x&\mapsto&\inf\{l_{\text{cl}}(\mathcal{U})(\xi)\ \suchthat\  \mathcal{U} \text{ is a relatively compact neighbourhood of } x\},
\end{eqnarray*}
which is the local sectional dilatation, respectively, of $\xi$. Note that, while the values taken by $\text{dil}\xi$ will depend on the choice of a Riemannian metric $\mathbb{G}$, the property $\text{dil}\  \xi(x) < \infty$ for $x \in M$ is independent of $\mathbb{G}$, whence $\xi\in \Gamma^{\text{lip}}(E)$ \citep[Lemma 3.10]{Jafarpour2014}.

The following characterisations of the local sectional dilatation are useful.
\begin{lem}[Local sectional dilatation using derivatives]
For a $C^\infty$-vector bundle $\pi : E \to M$ and for $\xi \in \Gamma^{\text{lip}}(E)$, we have
\begin{eqnarray*}
      \text{dil}\xi(x)=\inf\{\sup\{\|\nabla^{\pi_m}_{v_y}\xi\|_{\mathbb{G}_{M,\pi}}\ |\ y\in \text{cl}(\mathcal{U}),\ \|v_y\|_{\mathbb{G}_M}=1, \ \xi \text{ differentiable at }y\}|\\ \mathcal{U} \text{ is a relatively compact neighbourhood of } x\}.
\end{eqnarray*}
\begin{proof}
\citep[Lemma 3.12]{Jafarpour2014}.
\end{proof}
\end{lem}
\begin{lem}[Local sectional dilatation and sectional dilatation]\label{lem:lsdasd}
Let $\pi : E \rightarrow M$ be a $C^\infty$-vector bundle. Then, for each $x_0 \in M$, there exists a relatively compact neighbourhood $\mathcal{U}$ of $x_0$ such that
$$l_{\text{cl}(\mathcal{U})}(\xi) = \sup\{\text{dil}\ \xi(x)\ \suchthat\  x \in \text{cl}(\mathcal{U})\},\  \xi \in \Gamma^{\text{lip}}(E).$$
\end{lem}
\begin{proof}
We let $\mathcal{U}$ be a geodesically convex neighbourhood of $x_0$ so that
$$\text{dil}\ \xi(x)=\sup\{\|\nabla^{\pi_m}_{v_y}\xi\|_{\mathbb{G}_{M,\pi}}\ |\ y\in \text{cl}(\mathcal{U}),\ \|v_y\|_{\mathbb{G}_M}=1, \ \xi \text{ differentiable at }y\}.$$
Thus $l_{\text{cl}}(\mathcal{U})(\xi)$ is an upper bound for
$$\{\text{dil}\ \xi(x)\ \suchthat\ x\in\text{cl}(\mathcal{U})\}.$$
Next, let $\epsilon \in \R_{>0}$. Let $x \in \mathcal{U}$ and $v_x \in T_xM$ be such that (1) $\xi$ is differentiable at $x$,
(2) $||v_x||_{\mathbb{G}_M} = 1$, and (3) $l_{\text{cl}}(\mathcal{U})(\xi) - ||\nabla^\pi_{v_x}\xi ||_{\mathbb{G}_{M,\pi}} <\frac{\epsilon}{2}$. Then let $\mathcal{V}$ be a geodesically convex neighbourhood of $x$ such that $\text{cl}(\mathcal{V}) \subseteq \mathcal{U}$ and such that
$$\sup\{||\nabla^\pi_{v_y}\xi||_{\mathbb{G}_M,\pi}\ \suchthat\  y \in \text{cl}(\mathcal{V}),\  ||v_y||_{\mathbb{G}_M} = 1,\ \xi\text{ differentiable at } y\} - \text{dil}\ \xi(x) <\frac{\epsilon}{2}.$$
We have 
\begin{eqnarray*}
  l_{\text{cl}(\mathcal{U})}(\xi) -\frac{\epsilon}{2}< \sup\{\|\nabla^{\pi_m}_{v_y}\xi\|_{\mathbb{G}_{M,\pi}}\ |\ y\in \text{cl}(\mathcal{V}),\ \|v_y\|_{\mathbb{G}_M}=1, \ \xi \text{ differentiable at }y\} \leq l_{\text{cl}(\mathcal{U})}(\xi).
\end{eqnarray*}
Therefore,
\begin{eqnarray*}
  l_{\text{cl}(\mathcal{U})}(\xi) -\epsilon&=& l_{\text{cl}(\mathcal{U})}(\xi) -\frac{\epsilon}{2}-\frac{\epsilon}{2}\\
  &\leq& \sup\{\|\nabla^{\pi_m}_{v_y}\xi\|_{\mathbb{G}_{M,\pi}}\ |\ y\in \text{cl}(\mathcal{V}),\ \|v_y\|_{\mathbb{G}_M}=1, \ \xi \text{ differentiable at }y\} \\
  &&\ \ +\text{dil}\ \xi-  \sup\{\|\nabla^{\pi_m}_{v_y}\xi\|_{\mathbb{G}_{M,\pi}}\ |\ y\in \text{cl}(\mathcal{V}),\ \|v_y\|_{\mathbb{G}_M}=1, \ \xi \text{ differentiable at }y\}\\
  &=& \text{dil}\ \xi(x).
\end{eqnarray*}
This shows that $l_\text{cl}(\mathcal{U})(\xi)$ is the least upper bound for
$$\{\text{dil}\ \xi(x)\ \suchthat\ x\in\text{cl}(\mathcal{U})\},$$
as required. 
\end{proof}

From this result, we can deduce the continuity of the local sectional dilatation.
\begin{lem}[Local sectional dilatation is continuous]\label{lem:lsdic}
Let $\pi : E \to M$ be a $C^\infty$-vector bundle. If $\xi \in \Gamma^{\text{lip}}(E)$, then $\text{dil} \xi \in C^0(M)$.
\end{lem}
\begin{proof}
Let $x_0 \in M$ and let $\epsilon \in \R_{>0}$. By definition of $\text{dil} \xi$, there exists a relatively compact neighbourhood $\mathcal{U}$ of $x_0$ such that
$$l_{\text{cl}\mathcal{U})}(\xi) - \text{dil} \xi(x_0) < \epsilon.$$
If $x \in \mathcal{U}$, then, since $\mathcal{U}$ is a relatively compact neighbourhood of $x$, 
$$\text{dil} \xi(x) \leq l_{\text{cl}(\mathcal{U})} \implies \text{dil}\xi(x) - \text{dil} \xi(x_0) \leq l_{\text{cl}(\mathcal{U})} -  \text{dil} \xi(x_0) < \epsilon.$$
Now suppose that $\text{dil} \xi$ is discontinuous at $x_0$. Then there exists $\epsilon \in \R_{>0}$ such that, for every relatively compact neighbourhood $\mathcal{U}$ of $x_0$,
$$-\epsilon \geq \text{dil} \xi(x) - \text{dil} \xi(x_0) \geq \epsilon, x \in \mathcal{U}.$$
This is in contradiction with our conclusion from the preceding paragraph.
\end{proof}
\subsubsection{Fibre norms for jet bundles}\label{Section:3.1}
Fibre norms for jet bundles of a vector bundle play an important role in our unified treatment of various classes of regularities. Our discussion begins with general constructions for the fibres of jet bundles. Let $r\in \{\infty,\omega,\text{hol}\}$ and let $M$ be a $C^r$-manifold.  Let $\pi:E\rightarrow M$ be a $C^r$-vector bundle with $\pi_m:J^mE\rightarrow M$ its $m$th jet bundle. We shall suppose that we have a $C^r$-affine connection $\nabla^M$ on $M$ and a $C^r$-vector bundle connection $\nabla^\pi$ in $E$. By additionally supposing that we have a $C^r$-Riemannian metric $\mathbb{G}_M$ on M and a $C^r$-fibre metric $\mathbb{G}_\pi$ on $E$, we shall give a $C^r$-fibre norm on $J^mE$.

Denote $T^m(T^*M)$ the $m$-fold tensor product of $T^*M$ and $S^m(T^*M)$ the symmetric tensor bundle. The connection $\nabla^M$ induces a covariant derivative for tensor fields $A\in \Gamma^1(T^k_l(TM))$ on $M$, $k,l\in\mathbb{Z}_{\geq 0}$. This covariant derivative we denote by $\nabla^M$,  dropping the particular $k$ and $l$.  Similarly, the connection $\nabla^\pi$ induces a covariant derivative for sections $B\in \Gamma^1(T^k_l(E))$  of the tensor bundles associated with $E$, $k,l\in\mathbb{Z}_{\geq 0}$. This covariant derivative we denote by $\nabla^\pi$,  dropping the particular $k$ and $l$.
We will also consider differentiation of sections of $T^{k_1}_{l_1}(TM) \otimes T^{k_2}_{l_2}(E)$, and we denote the covariant derivative by $\nabla^{M,\pi}$. Note that
$$\nabla^{M,\pi,m}\xi\triangleq\underbrace{\nabla^{M,\pi}\cdots(\nabla^{M,\pi}}_{m-1\ \text{times}}(\nabla^\pi\xi))\in\Gamma^\infty(T^m(T^*M)\otimes E).$$
For $\xi\in \Gamma^\infty(E)$ and $m\in\mathbb{Z}_{\geq 0}$, we define
$$D^m_{\nabla^M,\nabla^\pi}(\xi)=\text{Sym}_m\otimes\text{id}_E(\nabla^{M,\pi,m}\xi)\in\Gamma^\infty(S^m(T^*M)\otimes E).$$
where $\text{Sym}_m:T^m(T^*M)\rightarrow S^m(T^*M)$ by
$$\text{Sym}_m(v_1\otimes\cdots\otimes v_m)=\frac{1}{m!}\sum_{\sigma\in\mathfrak{S}_m}v_{\sigma(1)}\otimes\cdots\otimes v_{\sigma(m)}.$$
We take the convention that $D^0_{\nabla^M,\nabla^\pi}(\xi)=\xi$. We then have a map
\begin{eqnarray*}
      S^m_{\nabla^M,\nabla^{M,\pi}}:J^mE &\rightarrow&  \bigoplus\limits_{j=0}^{m}(S^j(T^*M)\otimes E) \\
	         j_m\xi(x)&\mapsto& (\xi(x),D^1_{\nabla^M,\nabla^\pi}(\xi)(x),...,D^m_{\nabla^M,\nabla^\pi}(\xi)(x))
\end{eqnarray*}
that can be verified to be an  isomorphism of vector bundles \citep[Lemma 2.1]{Jafarpour2013}. Note that inner products on the components of a tensor products induce an inner product on the tensor product in a
natural way \citep[Lemma 2.2]{Jafarpour2013}. Then we have a fibre metric in all tensor spaces associated with $TM$ and $E$ and their tensor products. We shall denote by $G_{M,\pi}$ any of these various fibre metrics. In particular, we have a fibre metric $\mathbb{G}_{M,\pi}$ on $T^j(T^*M)\otimes E$ for each $j\in \mathbb{Z}_{\geq 0}$. This thus gives us a fibre metric $\mathbb{G}_{M,\pi,m}$ on $J^mE$ defined by
\begin{equation}\label{eq:fibremetric}
    \mathbb{G}_{M,\pi,m}(j_m\xi(x),j_m\eta(x))=\sum_{j=0}^{m}\mathbb{G}_{M,\pi}\left(\frac{1}{j!}D^j_{\nabla^M,\nabla^\pi}(\xi)(x),\frac{1}{j!}D^j_{\nabla^M,\nabla^\pi}(\eta)(x)\right).
\end{equation}

Associated to this inner product on fibres is the norm on fibres, which we denote by $\|\cdot\|_{\mathbb{G}_{M,\pi,m}}$. We shall use these fibre norms continually in our descriptions of various topologies in the next few sections.
\subsubsection{Seminorms for spaces of finitely differentiable sections}
In this section we give a seminorm for sections of regularity $\nu=m\in \mathbb{Z}_{\geq 0}$. We again take $\pi:E\rightarrow M$ to be smooth vector bundle. For the space $\Gamma^m(E)$ of $m$-times continuously differentiable sections, we define seminorms $p^m_K,\  K\subseteq M$ compact, for $\Gamma^m(E)$ by
$$p^m_K(\xi)=\sup\{\|j_m\xi(x)\|_{G_{M,\pi,m}}\ |\ x\in K\}.$$
We call the locally convex topology on $\Gamma^m(E)$ defined by the family of seminorms $p^m_K$ where $K\subseteq M$ compact, the $\mathbf{C^m}$\textbf{-topology}, and it is complete, Hausdorff, separable, and metrizable \citep[§3.4]{Jafarpour2013}.
\subsubsection{Seminorms for spaces of Lipschitz sections}
In this section we again work with a smooth vector bundle $\pi:E\rightarrow M$. Different from defining the fibre metrics from the last section, for the Lipschitz topologies the affine connection $\nabla^M$ is required to be the Levi-Civita connection for the Riemannian metric $G_M$ and the linear connection $\nabla^\pi$is required to be $G_\pi$-orthogonal. Because we have the decomposition
$$J^mE \simeq \bigoplus\limits_{j=0}^{m}(S^j(T^*M)\otimes E),$$
it follows that the vector bundle $J^mE$ has a $C^r$-connection $\nabla^{\pi_m}$, defined by
$$\nabla^{\pi_m}_X j_m\xi=(S^m_{\nabla^M,\nabla^\pi})^{-1}(\nabla^\pi_X \xi, \nabla^{M,\pi}_X D^1_{\nabla^M,\nabla^\pi}(\xi),...,\nabla^{M,\pi}_X D^m_{\nabla^M,\nabla^\pi}(\xi)).$$

By Rademacher’s Theorem \citep[Theorem 3.1.6]{MR0257325}, if a section $\xi$ is of class $C^{m+\text{lip}}$, then its $(m+1)$-th derivative exists almost everywhere. Then by \citep[Lemma 3.12]{Jafarpour2013}, we define
\begin{eqnarray*}
      \text{dil}\  j_m\xi(x)=\inf\{\sup\{\|\nabla^{\pi_m}_{v_y}j_m\xi\|_{\mathbb{G}_{M,\pi}}\ |\ y\in \text{cl}(\mathcal{U}),\ \|v_y\|_{\mathbb{G}}=1, \ j_m\xi \text{ differentiable at }y\}|\\ \mathcal{U} \text{ is a relatively compact neighbourhood of } x\}.
\end{eqnarray*}
which is the \textbf{local sectional dilatation} of $\xi$. Let $K\subseteq M$ be compact and define 
$$\lambda^m_K(\xi)=\sup\{\text{dil}j_m\xi(x)\ |\ x\in K\}$$
for $\xi\in \Gamma^{m+\text{lip}}(E)$. We then can define a seminorm on $\xi\in \Gamma^{m+\text{lip}}(E)$ by 
$$p^{m+\text{lip}}_K(\xi)=\max\{\lambda^m_K(\xi), p^m_K(\xi)\}.$$
We call the locally convex topology on $\Gamma^{m+\text{lip}}(E)$ defined by the family of seminorms $p^{m+\text{lip}}_K$, $K\subseteq M$ compact, the $\mathbf{C^{m+\textbf{lip}}}$\textbf{-topology}, and it is complete, Hausdorff, separable, and metrizable \citep[§3.5]{Jafarpour2014}.
\subsubsection{Seminorms for spaces of smooth sections}
Let $\pi:E\rightarrow M$ be a smooth vector bundle.  Using the fibre norms from the preceding section, it is a straightforward matter to define appropriate seminorms that define the locally convex topology for $\Gamma^\infty(E)$. For $K\subseteq M$ compact and for $m\in\mathbb{Z}_{\geq 0}$, define a seminorm $p^\infty_{K,m}$ on $\Gamma^\infty(E)$ by
$$p^\infty_{K,m}(\xi)=\sup\{\|j_m\xi(x)\|_{\mathbb{G}_{M,\pi,m}}\ |\ x\in K\}.$$
We call the locally convex topology on $\Gamma^\infty(E)$ defined by the family of seminorms $p^\infty_{K,m}$, $K\subseteq M$ compact, $m\in\mathbb{Z}_{\geq 0}$, the $\mathbf{C^\infty}$\textbf{-topology}, and it is complete, Hausdorff, separable, and metrizable \citep[§3.2]{Jafarpour2014}.
\subsubsection{Seminorms for spaces of holomorphic sections}
For the topology of the holomorphic sections, we consider an holomorphic vector bundle $\pi:E\rightarrow M$ and denote by $\Gamma^{\text{hol}}(E)$ the space of holomorphic sections of $E$. Let $G_\pi$ be an Hermitian metric on the vector
bundle and denote by $\|\cdot\|_{G_\pi}$ the associated fibre norm. For $K\subseteq M$ compact, denote by $p^{\text{hol}}_K$ the seminorm on $\Gamma^{\text{hol}}(E)$ defined by
$$p^{\text{hol}}_K(\xi)=\sup\{\|\xi(z)\|_{\mathbb{G}_\pi}\ |\ z\in K\}.$$
The family of seminorms $p^{\text{hol}}_K$ where $K\subseteq M$ compact, defines a locally convex topology for $\Gamma^{\text{hol}}(E)$ we call the $\mathbf{C}^{\textbf{hol}}$\textbf{-topology}, and it is complete, Hausdorff, separable, and metrizable \citep[§4.2]{Jafarpour2014}.
\subsubsection{Seminorms for spaces of real analytic sections}\label{sec:sfras}
The topology one considers for real analytic sections does not have the same attributes as smooth, finitely differentiable, Lipschitz, and holomorphic cases. There is a history to the characterisation of real analytic topologies, and we refer to [Jafarpour and Lewis 2014, §5.2] for four equivalent characterisations of the real analytic topology for the space of real analytic sections of a vector bundle. Here we will give the most elementary of these definitions to state, although it is probably not the most practical definition.

In this section we let $\pi:E\rightarrow M$ be a real analytic vector bundle and let $\Gamma^\omega(E)$ be the space of real analytic sections. Here we need all of the data used to define the seminorms in the finitely differentiable and smooth cases to topologise $\Gamma^\omega(E)$, only now we need this data to be real analytic. We refer to \citep[Lemma 2.4]{Jafarpour2014} for the existence of this data. Therefore, we can define real analytic fibre metrics $\mathbb{G}_{M,\pi,m}$ on the jet bundles $J^mE$ as in Section \ref{Section:3.1}. To define seminorms for $\Gamma^\omega(E)$, we let $c_0(\mathbb{Z}_{\geq 0};\R_{>0})$ denote the space of sequences in $\R_{> 0}$, indexed by $\mathbb{Z}_{\geq 0}$, and converging to zero. We shall denote a typical element of
$c_0(\mathbb{Z}_{\geq 0};\R_{>0})$ by $\boldsymbol{a}=(a_j)_{j\in \mathbb{Z}_{\geq 0}}$. Now for $K\subseteq M$ compact, and $\boldsymbol{a}\in c_0(\mathbb{Z}_{\geq 0};\R_{>0})$, we define a seminorm $p^\omega_{K,\boldsymbol{a}}$ on $\Gamma^\omega(E)$ by 

$$p^\omega_{K,\boldsymbol{a}}(\xi)=\sup\{a_0a_1...a_m \|j_m\xi(x)\|_{\mathbb{G}_{M,\pi,m}}\ |\ x\in K,\ m\in \mathbb{Z}_{\geq 0}\}.$$

The family of seminorms $p^\omega_{K,\boldsymbol{a}}$, $K\subseteq M$ compact, $\boldsymbol{a}\in c_0(\mathbb{Z}_{\geq 0};\R_{>0})$, defines a locally convex topology on $\Gamma^\omega(E)$ that we call the $\mathbf{C^\omega}$\textbf{-topology}, and it is complete, Hausdorff, separable, but not metrizable \citep[§5.3]{Jafarpour2014}.
\subsubsection{Summary and notation}
The preceding developments have been made for spaces of sections of a general vector bundle. We will primarily (but not solely) be interested in spaces of vector fields and functions. For vector fields, the seminorms above can be defined for an affine connection $\nabla^M$ on $M$, as this also serves as a vector bundle connection for $\pi_{TM}:TM\rightarrow M$. For vector fields, the decomposition of the jet bundle $J^m TM$ is 
\begin{eqnarray*}
      S^m_{\nabla^M}:J^mTM &\rightarrow&  \oplus_{j=0}^{m}(S^j(T^*M)\otimes TM) \\
	         j_mX(x)&\mapsto& (X(x),\text{Sym}_1\otimes\text{id}_{TM}(\nabla^M X)(x),...,\text{Sym}_m\otimes\text{id}_{TM}(\nabla^{M,m} X)(x)).
\end{eqnarray*}
where 
$$\nabla^{M,k} X\triangleq \underbrace{\nabla^M\cdots \nabla^M}_{k \text{ times}} X$$
For the Lipschitz topologies, one needs for $\nabla^M$ to be a metric connection, and so it may as well be the Levi-Civita connection associated with a Riemannian metric $\mathbb{G}$.  With the Riemannian metric $\mathbb{G}$, this decomposition gives the various topologies for spaces of vector fields. Functions are particular instances of sections of a vector bundle, as we can identity a function with a section of the trivial line bundle $\mathbb{F}_M=M\times \mathbb{F}$. We denote by $C^\nu(M)$ the set of functions having the regularity $\nu$ coming from one of the classes of regularity we consider. Note that this bundle has a canonical flat connection $\nabla^\pi$ which, when translated to functions, amounts to the requirement that $\nabla^\pi_X f=\mathscr{L}_Xf$. For $f\in C^\nu(M)$ when $\nu\geq m$, let us denote 
$$\nabla^{M,0}f=f,\  \nabla^{M,1}f=df, \text{ and }\nabla^{M,m}f=\nabla^{M,m-1}df.$$ 
Then the decomposition of the jet bundle of $\mathbb{F}_M$ looks like 
\begin{eqnarray*}
      S^m_{\nabla^M}:J^m(M;\R) &\rightarrow&  \R\oplus(\oplus_{j=0}^{m}S^j(T^*M)) \\
	         j_mf(x)&\mapsto& ((\nabla^{M,0}f)(x),(\nabla^{M,1}f)(x),...,\text{Sym}_m(\nabla^{M,m} f)(x)).
\end{eqnarray*}
This decomposition can be used to define the locally convex topologies for the various regularity classes of functions. Thus, if we suppose that we have a Riemannian metric $G_M$ and affine connection $\nabla^M$ on $M$, there is induced a natural fibre metric $G_m$ on $J^m(M;\R)$ for each $m\in\mathbb{Z}_{\geq 0}$ by 
$$G_{M,m}(j_mf(x),j_mg(x))=\sum_{j=0}^{m}G_{M}\left(\frac{1}{j!}\text{Sym}_j(\nabla^{M,j}f)(x),\frac{1}{j!}\text{Sym}_j(\nabla^{M,j}g)(x)\right),$$
and the associated norm we denote by $\|\cdot\|_{G_{M,m}}$.

In the real case, the degrees of regularity are ordered according to
\begin{equation}\label{eq:3.1}
    C^0\supset C^{\text{lip}}\supset C^1\supset\cdots \supset C^m\supset C^{m+1}\supset \cdots\supset C^\infty\supset C^\omega,
\end{equation}
and in the complex case the ordering is the same, of course, but with an extra $C^{\text{hol}}$ on the right.

With all the topologies for our various cases of regularities, we will, for $K\subseteq M$ be compact, for $k\in\mathbb{Z}_{\geq 0}$, and for $\boldsymbol{a}\in c_0(\mathbb{Z}_{\geq 0};\R_{>0})$, denote 
\begin{equation}\label{eq:3.2}
p^\nu_K=\left\{\begin{array}{lcl} p^m_K, \ \ \ \ \ \ \ \nu=m\\  p^{m+\text{lip}}_K, \ \ \ \nu=m+\text{lip},\\  p^\infty_{K,m}, \ \ \ \ \nu=\infty, \\  p^\omega_{K,\boldsymbol{a}}, \ \ \ \ \ \nu=\omega, \\  p^{\text{hol}}_K, \ \ \ \ \ \ \nu=\text{hol}\end{array}\right.
\end{equation}
The convenience and brevity more than make up for the slight loss of preciseness in this approach.

We comment that these seminorms make it clear that we have an ordering of the regularity classes as
$$m_1 < m_1 + \text{lip} < \cdots < m_2 < m_2 + \text{lip} < \cdots < \infty < \omega < \text{hol}$$
from least regular (coarser topology) to more regular (finer topology), and where $m_1 < m_2$. There is also an obvious “arithmetic” of degrees of regularity that we will use without feeling the need to explain it.
\subsection{Time- and parameter-dependent sections and functions}\label{sec:2.3}
In this section we introduce the classes of vector fields, depending on both time and parameter, that we work with. In our presentation, we shall make use of measurable and integrable functions with values in a locally convex topological vector space. This is classical in the case of Banach spaces, but is not as fleshed out in the general case.
\subsubsection{time-dependent sections}
We carefully introduce in this section the class of timedependent vector fields we consider, and which were quickly introduced in Section 1.2. First of all, since all topological vector spaces we consider are Suslin spaces, all standard notions of measurability coincide \citep[Theorem 1]{MR385067}. Thus, for example, one can take as one’s notion of measurability the naive one that preimages of Borel sets are measurable. The notion of integrability we use is ``integrability by seminorm," and seems to originate in \citep{MR0442176}. We refer to \citep{lewis2021geometric} for details and further references. For complete Suslin spaces, such as we are working with, integrability by seminorm amounts to the requirement that the application of any continuous seminorm to the vector-valued functions yields a function in the usual scalar $L^1$-space. 

With these technicalities stowed away for safety and convenience, we now make our definitions. Let $m\in \mathbb{Z}_{\geq 0}$, let $m'\in\{0,\text{lip}\}$, let $\nu\in\{m+m',\infty,\omega,\text{hol}\}$, and let $r\in \{\infty,\omega,\text{hol}\}$, as required. Let $\pi : E \rightarrow M$ be a $C^r$-vector bundle and let $\mathbb{T} \subseteq \R$ be an interval. We say that $\xi : \mathbb{T} \rightarrow \Gamma^\nu(E)$ is locally integrally bounded if, for any continuous seminorm $p$ for $\Gamma^\nu(E)$ and for $\mathbb{S} \subseteq \mathbb{T}$ compact, $p \circ (\xi|\mathbb{S}) \in L^1(\mathbb{T}; \R)$. We give a locally convex topology for the set of locally integrally bounded sections by the seminorms
$$p^\nu_{K,\mathbb{S}}(\xi) = \int_{\mathbb{S}}p^\nu_{K}(\xi(t))\d t,\hspace{10pt} K \subseteq M, \mathbb{S}\subseteq \mathbb{T} \ \text{compact},$$
where $p^\nu_{K}$ is the seminorm defined by (\ref{eq:3.2}). Thus,
colloquially, the set of locally integrally bounded mappings is just $L^1_{\text{loc}}(\mathbb{T}; \Gamma^\nu(E))$. We abbreviate this space by $\Gamma^\nu_{\text{LI}}(\mathbb{T}; E)$. Explicit characterisations of these spaces are given by \citep{Jafarpour2014}.
\begin{defin}[Classes of time-varying sections]
 For a vector bundle $\pi:E\rightarrow M$ of class $C^r$ and an interval $\mathbb{T}\subseteq\R$, let $\xi:\mathbb{T}\times M\rightarrow E$ satisfy $\xi(t,x)\in E_x$ for each $(t,x)\in\mathbb{T}\times M$. Denote by $\xi_t\;(t\in \mathbb{T})$, the map $x\mapsto \xi(t,x)$ and suppose that $\xi_t\in\Gamma^\nu(E)$ for every $t\in\mathbb{T}$. Then $\xi$ is:
\begin{enumerate}
    \item[(i)]{
    a \emph{Carath\'{e}odory section of class $C^\nu$} if the curve $\mathbb{T}\ni t\mapsto\xi_t\in \Gamma^\nu(E)$ is measurable;
    }
    \item[(ii)]{
    \emph{(locally) integrally $C^\nu$-bounded} if the curve $\mathbb{T}\ni t\mapsto\xi_t\in\Gamma^\nu(E)$ is (locally) Bochner integrable;
    }
    \item[(iii)]{
    \emph{(locally) essentially $C^\nu$-bounded} if the curve $\mathbb{T}\ni t\mapsto\xi_t\in\Gamma^\nu(E)$ is (locally) essentially von Neumann bounded.
	}
\end{enumerate}
We denote:
\begin{enumerate}
    \item[(iv)]{
    the set of Carath\'{e}odory sections of class $C^\nu$ by $\Gamma^\nu_{\text{CF}}(\mathbb{T};E)$;
    }
    \item[(v)]{
    the set of (locally) integrally $C^\nu$-bounded sections by ($\Gamma^\nu_{\text{LI}}(\mathbb{T};E)$) $\Gamma^\nu_{\text{I}}(\mathbb{T};E)$;
    }
    \item[(vi)]{
    the set of (locally) essentially $C^\nu$-bounded sections by ($\Gamma^\nu_{\text{LB}}(\mathbb{T};E)$) $\Gamma^\nu_{\text{B}}(\mathbb{T};E)$;
	}
\end{enumerate}
\end{defin}
We shall also find some alternative, more expansive, notation profitable, notation that is the “standard” adaptation from the usual notation for functions. For a time-domain $\mathbb{T}$ we denote
\begin{enumerate}
    \item $\text{CF}(\mathbb{T};\Gamma^{\nu}(E))$: Carath\'{e}odory sections of class $C^\nu$;
    \item $\text{L}^{1}(\mathbb{T};\Gamma^{\nu}(E))$: integrally $C^\nu$-bounded sections;
    \item $\text{L}^{1}_{\text{loc}}(\mathbb{T};\Gamma^{\nu}(E))$: integrally $C^\nu$-bounded sections;
    \item $\text{L}^{\infty}(\mathbb{T};\Gamma^{\nu}(E))$: essentially $C^\nu$-bounded sections;
    \item $\text{L}^{\infty}_{\text{loc}}(\mathbb{T};\Gamma^{\nu}(E))$: locally essentially $C^\nu$-bounded sections.
\end{enumerate}
The spaces of integrable or essentially bounded sections have natural topologies, which we now describe.
\begin{enumerate}
    \item The space $\text{L}^{1}(\mathbb{T};\Gamma^{\nu}(E))$ has a locally convex topology defined by seminorms $p^{\nu}_{K,1}$, $K\subseteq M$ compact, given by
    $$p^{\nu}_{K,1}(\xi)=\int_\mathbb{T}p^{\nu}_{K}\circ \xi_t \d t,$$
    where $p^{\nu}_{K}$ is a seminorm for the topology on $\Gamma^{\nu}(E)$. Here we make the abuse of notation by possibly omitting extra bits of notation from the seminorm $p^{\nu}_{K}$, cf. (\ref{eq:3.2}). We recall that $$\text{L}^{1}(\mathbb{T};\Gamma^{\nu}(E))\simeq \text{L}^{1}(\mathbb{T};\R)\widehat{\otimes}_\pi \Gamma^{\nu}(E),$$
    where $\widehat{\otimes}_\pi$ denotes the completed projective tensor product \citep[Corrolary 15.7.2]{MR632257}.
    \item The space $\text{L}^{1}_{\text{loc}}(\mathbb{T};\Gamma^{\nu}(E))$ has a locally convex topology defined by seminorms $p^{\nu}_{K,\mathbb{S},1}$, $K\subseteq M$ compact, $\mathbb{S}\subseteq\mathbb{T}$ compact, given by 
    \begin{equation}\label{eq:2.3}
        p^{\nu}_{K,\mathbb{S},1}(\xi)=\int_\mathbb{S}p^{\nu}_{K}\circ \xi_t \d t,
    \end{equation}
    where seminorms $p^{\nu}_{K}$ are as above. We note that $\text{L}^{1}_{\text{loc}}(\mathbb{T};\Gamma^{\nu}(E))$ is the inverse limit of the inverse system $\text{L}^{1}_{\text{loc}}(\mathbb{S};\Gamma^{\nu}(E))_{\mathscr{K}_\mathbb{T}}$, where $\mathscr{K}_\mathbb{T}$ is the directed set of compact subintervals of $\mathbb{T}$ ordered by $\mathbb{S}_1\preceq \mathbb{S}_2$ if $\mathbb{S}_1\subseteq\mathbb{S}_2$. 
    \item The space $\text{L}^{\infty}(\mathbb{T};\Gamma^{\nu}(E))$ has a locally convex topology defined by seminorms $p^\nu_{K,\infty}$, $K\subseteq M$ compact, given by 
    $$p^\nu_{K,\infty}(\xi)=\text{ess sup}\{p^\nu_K\circ\xi_t\;|\;t\in\mathbb{T}\},$$
    where seminorms $p^{\nu}_{K}$ are as above. We have 
    $$\text{L}^{\infty}(\mathbb{T};\Gamma^{\nu}(E))\simeq \text{L}^{\infty}(\mathbb{T};\R)\widehat{\otimes}_\pi \Gamma^{\nu}(E).$$
    \item The space $\text{L}^{\infty}_{\text{loc}}(\mathbb{T};\Gamma^{\nu}(E))$ has a locally convex topology defined by seminorms $p^{\nu}_{K,\mathbb{S},\infty}$ $K\subseteq M$ compact, $\mathbb{S}\subseteq\mathbb{T}$ compact, given by 
    $$p^{\nu}_{K,\mathbb{S},\infty}(\xi)=\text{ess sup}\{p^\nu_K\circ\xi_t\;|\;t\in\mathbb{S}\},$$
    where seminorms $p^{\nu}_{K}$ are as above. We note that $\text{L}^{\infty}_{\text{loc}}(\mathbb{T};\Gamma^{\nu}(E))$ is the inverse limit of the inverse system $\text{L}^{\infty}_{\text{loc}}(\mathbb{S};\Gamma^{\nu}(E))_{\mathscr{K}_\mathbb{T}}$. 
\end{enumerate}
Our method of working with vector fields and their flows is to use general globally defined functions to replace local coordinates. As such, functions assume an important role in our presentation. Bearing in mind that functions are sections of the trivial line bundle, the above general definitions for sections of vector bundles apply specifically to functions, and yield the spaces $C^\nu_{\text{LI}}(\mathbb{T}; M)$ of time-dependent functions $f:\mathbb{T} \rightarrow C^\nu(M)$.

The following lemma indicates how we will convert sections and vector fields into functions.
\begin{lem}[Time-dependent functions from time-dependent sections and vector fields]\label{lem:2.3}
Let $m\in \mathbb{Z}_{\geq 0}$, let $m'\in\{0,\text{lip}\}$, let $\nu\in\{m+m',\infty,\omega,\text{hol}\}$, and let $r\in \{\infty,\omega,\text{hol}\}$, as required. Let $\beta : B \rightarrow M$ be a $C^r$-affine bundle modelled on the $C^r$-
vector bundle $\pi : E \rightarrow M$, and let $\mathbb{T} \subseteq \R$ be an interval. When $\nu = \text{hol}$, assume that $M$ is a Stein manifold. Then $X\in\Gamma^\nu_{\text{LI}}(\mathbb{T};TM)$ if and only if $X_tf\in C^0_{\text{LI}}(\mathbb{T};B)$ for any $f\in C^r(M)$. 
\end{lem}
\begin{proof}
We note that the seminorms on $\Gamma^\nu(TM)$ provided by (\ref{eq:3.2}) is the initial topology associated with the mappings $X\mapsto\mathscr{L}_Xf$, $f\in C^r(M)$ \citep[Section 3]{Jafarpour2014}. We observe that if $\pi_E : E \rightarrow M$ and $\pi_F : F \rightarrow N$ are $C^r$-vector bundles, $\nu_1, \nu_2 \in \{m + m', \infty, \omega, \text{hol}\}$ are two regularity classes, and if $\phi: \Gamma^{\nu_1} (E) \rightarrow  \Gamma^{\nu_1} (F)$ is a continuous linear map, then
$$\phi \circ \xi\in \Gamma^{\nu_1}_{\text{LI}}(\mathbb{T};F),\hspace{10pt} \xi\in \Gamma^{\nu_2}_{\text{LI}}(\mathbb{T};E),$$
since continuous linear maps preserve continuity of seminorms. Then the desired conclusion follows.
\end{proof}
The following result gives a more familiar characterisation of locally integrally bounded functions of class $C^{\text{lip}}$, bearing in mind our policy of introducing a Riemannian metric whenever it is convenient.
\begin{lem}[Property of time-varying locally Lipschitz functions]\label{lem:2.4}
Let $M$ be a smooth manifold, let $\mathbb{T}$ be a time-domain, and let $f \in C^{\text{lip}}_{\text{LI}}(\mathbb{T}; M)$. If $K \subseteq M$ is compact, then there exists $l\in L^1_{\text{loc}}(\mathbb{T}; \R_{\geq 0})$ such that
$$|f(t, x_1) - f(t, x_2)| \leq l(t)d_\mathbb{G}(x_1, x_2), \hspace{10pt}t \in \mathbb{T},\  x_1, x_2 \in K.$$
\end{lem}
\begin{proof}
Since functions are to be thought of as sections of the trivial line bundle $\R_M=M\times \R$, and since we use the flat connection on this bundle, we have, for any compact set $K \subseteq M$ and for $g \in C^{\text{lip}}(M)$,
\begin{align*}
    l_K(g)&= \sup\left\{\frac{|g\circ \gamma(1) - g \circ \gamma(0)|}{l_\mathbb{G}(\gamma)}\ \suchthat \ \gamma:[0, 1] \rightarrow M, \gamma(0),\  \gamma(1) \in K,\  \gamma(0) \neq\gamma(1)\right\}\\
    &=\sup\left\{\frac{|g(x_1) - g(x_2)|}{d_\mathbb{G}(x_1, x_2)}\ \suchthat \ x_1, x_2 \in K,\  x_1 \neq x_2\right\}.
\end{align*}
Let $K\subseteq M$ be compact. Let $x \in K$ and let $\mathcal{U}_x$ be a neighbourhood of $x$ such that, by Lemma 2.9, for $g \in C^{\text{lip}}(M)$, we have $\lambda^0_{\text{cl}(\mathcal{U}_x)}(g) =l_{\text{cl}(\mathcal{U}_x)}(g)$. Since $f\in C^{\text{lip}}_{\text{LI}}(\mathbb{T};M)$, there exists $l_x\in L^1_{\text{loc}}(\mathbb{T};\R_{\geq 0})$ such that 
$$\text{dil} f(t, y) \leq l_x(t), \hspace{10pt} (t, y) \in \mathbb{T} \times \text{cl}(\mathcal{U}_x).$$
Thus, for $x_1, x_2 \in \mathcal{U}_x$, we have
$$|f(t, x_1) - f(t, x_2)| \leq l_x(t)d_\mathbb{G}(x_1, x_2), \hspace{10pt} t \in \mathbb{T}.$$
By compactness of $K$, there exist $x_1,..., x_m \in K$ such that $K \subseteq \cup_{j=1}^m \mathcal{U}_{x_j}$. By the Lebesgue Number Lemma \citep[Theorem 1.6.11]{MR1835418},  there exists $r \in \R_{>0}$ with the property that, if $x_1, x_2 \in K$ satisfy $d_\mathbb{G}(x_1, x_2) < r$, then there exists $j \in \{1,..., m\}$ such that $x_1,x_2\in\mathcal{U}_{x_j}$. Since $C^{\text{lip}}_{\text{LI}}(\mathbb{T};M)\subseteq C^{0}_{\text{LI}}(\mathbb{T};M)$, there exists $\kappa\in L^1_{\text{loc}}(\mathbb{T};\R_{\geq 0})$ such that $|f(t, x)| \leq \kappa(t)$ for $(t, x) \in \mathbb{T} \times K$. Let
$$l(t) = \max \left\{l_{x_1}(t), . . . , l_{x_m}(t), \frac{2\kappa(t)}{r},\hspace{10pt} t \in \mathbb{T}\right\},$$
noting that $l\in L^1_{\text{loc}}(\mathbb{T};\R_{\geq 0})$.  Let $x_1, x2 \in K$. If $d_\mathbb{G}(x_1, x_2) < r$, then let $j \in \{1, . . . , m\}$ be such that $x_1, x_2 \in \mathcal{U}_{x_j}$, and then we have
$$|f(t, x_1) - f(t, x_2)| \leq l_{x_j}(t)d_\mathbb{G}(x_1, x_2)\leq l(t)d_\mathbb{G}(x_1, x_2), \hspace{10pt} t \in \mathbb{T}.$$
If $d_\mathbb{G}(x_1,x_2)\geq r$, then 
$$|f(t, x_1) - f(t, x_2)| \leq |f(t, x_1)| + |f(t, x_2)| \leq 2\kappa(t) \leq
\frac{2\kappa(t)}{r}d_\mathbb{G}(x_1, x_2) \leq l(t)d_\mathbb{G}(x_1, x_2),$$
which gives the result. 
\end{proof}
\subsubsection{Integrable sections along a curve.}\label{sec:2.3.2}
We shall require the notion of a section of an affine bundle over a curve. Thus we let $r \in \{\infty, \omega, \text{hol}\}$ and let $\beta: B \rightarrow M$ be a $C^r$-affine bundle over a $C^r$-vector bundle $\pi : E \rightarrow M$. Let $\mathbb{T} \subseteq \R$ be an interval. We denote by $L^1_{\text{loc}}(\mathbb{T}; B)$ the mappings $\Gamma: \mathbb{T} \rightarrow B$ for which $F \circ \Gamma \in L^1_{\text{loc}}(\mathbb{T};F)$ for every $F \in \text{Aff}^r(B)$. Note that, if $\Gamma \in L^1_{\text{loc}}(\mathbb{T}; B)$, then there is a mapping $\gamma : \mathbb{T} \rightarrow M$ specified by the requirement that the diagram
\[\begin{tikzcd}
\mathbb{T} \arrow{dr}{\gamma} \arrow{r}{\Gamma}
& B \arrow{d}{\beta} \\
& M
\end{tikzcd}\]
commute. We can think of $\Gamma$ as being a section of $B$ over $\gamma$. We topologise $L^1_{\text{loc}}(\mathbb{T}; B)$ by
giving it the initial topology associated with the mappings
\begin{eqnarray*}
  \alpha_F:L^1_{\text{loc}}(\mathbb{T}; B)&\rightarrow& L^1(\mathbb{S};\mathbb{F})\\
  \Gamma&\mapsto& F\circ\Gamma|\mathbb{S},
\end{eqnarray*}
where $F\in\text{Aff}^r(B)$ and $\mathbb{S}\subseteq\mathbb{T}$ a compact subinterval. 

Associated to these spaces of integrable sections over a curve, we have a few constructions and technical results whose importance will be made apparent at various points during the subsequent presentation. We consider the space $C^0 (\mathbb{T}; M)$ with the topology (indeed,
uniformity) defined by the family of semimetrics 
\begin{equation}\label{eq:gam12}
    d_{\mathbb{S},M}(\gamma_1, \gamma_2) = \sup\{d_\mathbb{G}(\gamma_1(t), \gamma_2(t))\ \suchthat\  t \in \mathbb{S}\}, \hspace{10pt} \mathbb{S} \subseteq \mathbb{T} \text{ a compact interval.}
\end{equation}
We consider this in the following context. We consider $C^r$-vector bundles $\pi_E : E \rightarrow M$ and $\pi_F : F \rightarrow N$. We abbreviate
$$E^* \otimes F = \text{pr}^*_1 E^* \otimes \text{pr}^*_2 F,$$
where $\text{pr}_1: M \times N \rightarrow M$ and $\text{pr}_2: M \times N \rightarrow N$ are the projections. We regard $E^* \otimes F$ as a
vector bundle over $M \times N$. The fibre over $(x, y)$ we regard as
$$E^*_x \otimes F_y \simeq \text{Hom}_\mathbb{F}(E_x; F_y).$$
Note that the total spaces $E$ and $F$ of these vector bundles, and so also $E^* \otimes F$, inherit a Riemannian metric from a Riemannian metrics on their base spaces and fibre metrics \citep[\S 4.1]{lewis2020geometric}. Thus $E^*\otimes F$ possess the associated distance function, and we shall make use of this to define, as in (\ref{eq:gam12}), a topology on the space $C^0(\mathbb{T}; E^*\otimes F)$. If $\Gamma \in C^0(\mathbb{T}; E^*\otimes F)$, then we have induced mappings
$$\gamma_M \in C^0(\mathbb{T}; M),\hspace{10pt} \gamma_N \in C^0(\mathbb{T}; N)$$
obtained by first projecting to $M \times N$, and then projecting onto the components of the product. Note that $\Gamma(t) \in \text{Hom}_\mathbb{F}(E_{\pi_F\circ\Gamma(t)}; E_{\pi_F\circ\Gamma(t)}),\  t \in \mathbb{T}$.

If $\xi : \mathbb{T} \times M \rightarrow E$ satisfies $\xi(t, x) \in E_x$ and if $\Gamma: \mathbb{T} \rightarrow E^* \otimes F$, then we can define
\begin{eqnarray*}
  \xi_\Gamma:\mathbb{T}&\to&F\\
  t&\mapsto&\Gamma(t)(\xi(t,\gamma_M(t))).
\end{eqnarray*}
We call $\xi_\Gamma$ the composite section associated with $\xi$ and $\Gamma$. The following lemma shows that this mapping is integrable, under suitable hypotheses on $\xi$ and $\Gamma$.
\begin{lem}[Integrability of composite section]\label{lem:2.5}
Let $\pi_E : E \to M$ and $\pi_F : F \to M$ be $C^\infty$-vector bundles and let $\mathbb{T} \subseteq \R$ be an interval. If $\xi \in \Gamma^0_{\text{LI}}(\mathbb{T}; E)$ and if $\Gamma \in C^0(\mathbb{T}; E^*\otimes E)$, then $\xi_\Gamma \in L^1_{\text{loc}}(\mathbb{T}; F)$.
\end{lem}
\begin{proof}
Let $G\in\text{Aff}^\infty(F)$. We first show that $t\mapsto G\circ\xi_\Gamma$ is measurable on $\mathbb{T}$. Note that 
$$t\mapsto G\circ\Gamma(s)(\xi(t,\gamma_M(s)))$$
is measurable for each $s \in \mathbb{T}$ and that
\begin{equation}\label{eq:2.5}
    s\mapsto G\circ\Gamma(s)(\xi(t,\gamma_M(s)))
\end{equation}
is continuous for each $t \in \mathbb{T}$. Let $[a, b] \subseteq \mathbb{T}$ be compact, let $k \in \mathbb{Z}_{>0}$, and denote
$$t_{k,j} = a + \frac{j-1}{k}(b - a), \hspace{10pt} j \in \{1, . . . , k + 1\}.$$
Also denote 
$$\mathbb{T}_{k,j} = [t_{k,j} , t_{k,j+1}), \hspace{10pt} j \in \{1, . . . , k - 1\},$$
and $\mathbb{T}_{k,k} = [t_{k,k}, t_{k,k+1}]$. Then define $g_k : \mathbb{T} \rightarrow \R$ by
$$g_k(t)=\sum_{j=1}^{k}G\circ\Gamma(t_{k,j})(\xi(t,\gamma_M(t_{k,j})))\chi_{t_{k,j}}.$$
Note that $g_k$ is measurable, being a sum of products of measurable functions \citep[Proposition 2.1.7]{MR3098996}. By continuity of (\ref{eq:2.5}) for each $t \in \mathbb{T}$, we have
$$\lim_{k\to\infty}g_k(t) = G\circ\Gamma(t)(\xi(t,\gamma_M(t))),\hspace{10pt} t \in [a, b],$$
showing that $t\mapsto G\circ\Gamma(t)(\xi(t,\gamma_M(t)))$ is measurable on $[a,b]$, as pointwise limits of measurable functions are measurable \citep[Proposition 2.1.5]{MR3098996}. Since the compact interval $[a, b] \subseteq \mathbb{T}$ is arbitrary, we conclude that $t \mapsto  G\circ\Gamma(t)(\xi(t,\gamma_M(t)))$ is measurable on $\mathbb{T}$.

Let $\mathbb{S} \subseteq \mathbb{T}$ be compact and let $K \subseteq M$ be a compact set for which $\gamma_M(\mathbb{S}) \subseteq K$. Since $\xi\in \Gamma^0_{\text{LI}}(\mathbb{T};M)$ and since $\Gamma$ is continuous, there exists $h\in L^1(\mathbb{S};\R_{\geq 0})$ be such that 
$$|G\circ \Gamma(t)(\xi(t,x))|\leq h(t) \hspace{20pt} (t,x)\in\mathbb{S}\times K,$$
In particular, this shows that $t \mapsto G\circ\Gamma(t)(\xi(t,\gamma_M(t)))$ is integrable on $\mathbb{S}$ and so locally integrable on $\mathbb{T}$.
\end{proof}
We also have the mapping
\begin{eqnarray*}
  \Psi_{\mathbb{T},E,\xi}:C^0(\mathbb{T}; E^* \otimes F) &\to& L^1_{\text{loc}}(\mathbb{T}; F)\\
  \Gamma&\mapsto& \xi_\Gamma,
\end{eqnarray*}
which is well-defined by Lemma \ref{lem:2.5}. The following lemma gives the continuity of this mapping.
\begin{lem}[Continuity of curve to composite section map]\label{lem:2.6}
Let $\pi_E : E \to M$ and $\pi_F : F\to M$ be $C^\infty$-vector bundles, let $\mathbb{T} \subseteq \R$ be an interval, and let $\xi \in \Gamma^0_{\text{LI}}(\mathbb{T}; E)$. Then $\Psi_{\mathbb{T},E,\xi}$ is continuous.
\end{lem}
\begin{proof}
Let $G \in \text{Aff}^\infty(F)$ and let $\mathbb{S} \subseteq \mathbb{T}$ be a compact interval. Let $\Gamma_j \in C^0(\mathbb{T}; E^* \otimes F)$,
$j \in \mathbb{Z}_{>0}$, be a sequence of curves converging to $\Gamma \in C^0(\mathbb{T}; E^* \otimes F)$. Since $\Gamma(\mathbb{S})$ is compact
and $E^* \otimes F$ is locally compact, we can find a precompact neighbourhood $\mathcal{W}$ of $\Gamma(\mathbb{S})$. Then, for $N \in \mathbb{Z}_{>0}$ sufficiently large, we have $\Gamma_j (\mathbb{S}) \subseteq \mathcal{W}$ by uniform convergence. Therefore, we can find a compact set $L \subseteq E^* \otimes F$ such that $\Gamma^j (\mathbb{S}) \subseteq L$, $j \in \mathbb{Z}_{>0}$, and $\Gamma(\mathbb{S}) \subseteq L$. Let $g \in L^1(\mathbb{S}; \R_{\geq0})$ be such that
$$|G\circ\Gamma(t)(\xi(t,x))|\leq h(t) \hspace{10pt} (t,x)\in\mathbb{S}\times K,$$
this since $\xi \in \Gamma^0_{\text{LI}}(\mathbb{T}; M)$ and since $\Gamma$ is continuous. Then, for fixed $t \in \mathbb{S}$, continuity of
$x \mapsto G \circ \Gamma(t)(\xi(t, x))$ ensures that
$$\lim_{j\to\infty}G\circ\Gamma(t)(\xi(t,\gamma_{M,j}(t)))= G\circ\Gamma(t)(\xi(t,\gamma_M(t))).$$
We also have
$$|G\circ\Gamma(t)(\xi(t,\gamma_{M,j}(t)))|\leq g(t),\hspace{10pt} t\in\mathbb{S}.$$
Therefore, by the Dominated Convergence Theorem \citep[Theorem 2.4.5]{MR3098996}
$$\lim_{j\to\infty}\int_{\mathbb{S}}G\circ\Gamma(t)(\xi(t,\gamma_{M,j}(t)))\d t=\int_{\mathbb{S}}G\circ\Gamma(t)(\xi(t,\gamma_M(t)))\d t,$$
which gives the desired continuity.
\end{proof}
\subsubsection{Time- and parameter-dependent sections.}
Now we turn our attention to vector fields depending on both parameter and time, as we outlined in Section \ref{sec:1.2}. 
\begin{defin}
Let $m\in \mathbb{Z}_{\geq 0}$, let $m'\in\{0,\text{lip}\}$, let $\nu\in\{m+m',\infty,\omega,\text{hol}\}$, and let $r\in \{\infty,\omega,\text{hol}\}$, as required. Let $\pi:E\rightarrow M$ be a vector bundle of class $C^r$, let $\mathbb{T}\subseteq\R$  be a time interval, and let $\mathcal{P}$ be a topological space. Consider a map $\xi:\mathbb{T}\times M\times \mathcal{P}\rightarrow E$ with the property that $\xi(t,x,p)\in E_x$ for each $(t,x,p)\in\mathbb{T}\times M\times \mathcal{P}$. Denote by $\xi^p:\mathbb{T}\times M\rightarrow E$ the map $\xi^p(t,x)=\xi(t,x,p)$. Then $\xi$ is a:
\begin{enumerate}[(i)]
    \item \textbf{\textit{separately continuous parameter-dependent, (locally) integrally bounded section of class}} $\mathbf{C^\nu}$ if $(t,x)\to \xi(t,x,p)$ is a (locally) integrally bounded section of class of class $C^\nu$ for each $p\in\mathcal{P}$ and if $p\to \xi(t,x,p)$ is continuous for each $(t,x)\in\mathbb{T}\times M$;
    \item \textbf{\textit{parameter-dependent, (locally) integrally bounded section of class}} $\mathbf{C^\nu}$ if it is a separately continuous parameter-dependent, (locally) integrally bounded section of class $\mathbf{C^\nu}$ and if the map $\mathcal{P}\ni p\to \xi^p\in\Gamma^\nu_{\text{I}}(\mathbb{T};E)\ (\Gamma^\nu_{\text{LI}}(\mathbb{T};E))$ is continuous;
    \item \textbf{\textit{separately continuous parameter-dependent, (locally) essentially bounded section of class}} $\mathbf{C^\nu}$ if $(t,x)\to \xi(t,x,p)$ is a (locally) essentially bounded section of class of class $C^\nu$ for each $p\in\mathcal{P}$ and if $p\to \xi(t,x,p)$ is continuous for each $(t,x)\in\mathbb{T}\times M$;
    \item \textbf{\textit{parameter-dependent, (locally) essentially bounded section of class}} $\mathbf{C^\nu}$ if it is a separately continuous parameter-dependent, (locally) essentially bounded section of class $\mathbf{C^\nu}$ and if the map $\mathcal{P}\ni p\to \xi^p\in\Gamma^\nu_{\text{B}}(\mathbb{T};E)\ (\Gamma^\nu_{\text{LB}}(\mathbb{T};E))$ is continuous.
\end{enumerate}
We denote:
\begin{enumerate}
    \item[(v)]{
    the set of parameter-dependent, (locally) integrally bounded section of class $C^\nu$ by $\Gamma^\nu_{\text{PLI}}(\mathbb{T};E;\mathcal{P})$;
    }
    \item[(vi)]{
    the set of parameter-dependent, (locally) essentially bounded section of class $C^\nu$ by ($\Gamma^\nu_{\text{PLB}}(\mathbb{T};E;\mathcal{P})$) $\Gamma^\nu_{\text{PB}}(\mathbb{T};E;\mathcal{P})$.
	}
\end{enumerate}
\end{defin}
As with solely parameter-dependent and solely time-varying sections, we have more fulsome notation for time-varying, parameter-dependent sections that is sometimes useful. To see this, we first observe that the spaces
$$\text{L}^{1}(\mathbb{T};\Gamma^{\nu}(E)),\hspace{8pt}\text{L}^{1}_{\text{loc}}(\mathbb{T};\Gamma^{\nu}(E)),\hspace{8pt}\text{L}^{\infty}(\mathbb{T};\Gamma^{\nu}(E)),\hspace{8pt}\text{L}^{\infty}_{\text{loc}}(\mathbb{T};\Gamma^{\nu}(E))$$
have topologies, as described above. We then have the spaces 
\begin{enumerate}
    \item $C^0(\mathcal{P};\text{L}^{1}(\mathbb{T};\Gamma^{\nu}(E)))$,
    \item $C^0(\text{L}^{1}_{\text{loc}}(\mathbb{T};\Gamma^{\nu}(E)))$,
    \item $C^0(\text{L}^{\infty}(\mathbb{T};\Gamma^{\nu}(E)))$, and
    \item $C^0(\text{L}^{\infty}_{\text{loc}}(\mathbb{T};\Gamma^{\nu}(E)))$.
\end{enumerate}
This notation is sufficiently obvious that it does not warrant explanation. As in the solely time-varying case, we shall not always state result in all four cases of ``integrable," ``locally integrable," ``essentially bounded," and ``locally essentially bounded," although all of the obvious results hold, and we shall use them.

To give a slightly explicit characterisation of membership in $\Gamma^\nu_{\text{PLI}}(\mathbb{T};E;\mathcal{P})$, we note that the conditions for such membership on $\xi$ are, just by definition: for each $p_0 \in \mathcal{P}$, for each compact $K \subseteq M$ and $\mathbb{S} \subseteq \mathbb{T}$, and for each $\epsilon \in \R_{>0}$, there exists a neighbourhood $\mathcal{O} \subseteq \mathcal{P}$ of $p_0$ such that
\begin{equation}\label{eq:2.4}
    \int_{\mathbb{S}}p^\nu_K(\xi^p_t-\xi^{p_0}_t)\d t<\epsilon,\hspace{10pt}p\in\mathcal{O},
\end{equation}
where seminorms $p^{\nu}_{K}$ are as above. We can, moreover, topologise these spaces. As in the case of parameter-dependent sections, there are at least two ways to introduce such a topology.
\begin{enumerate}
    \item The \textbf{\textit{topology of pointwise convergence}} on $C^0(\mathcal{P};\text{L}^{1}(\mathbb{T};\Gamma^{\nu}(E)))$ is the locally convex topology defined by the family of seminorms $p^{\nu}_{K,1,p}$, $K\subseteq M$ compact, given by
    $$p^{\nu}_{K,1,p}(\xi)=p^{\nu}_{K,1}\circ \xi^p,$$
    where $p^{\nu}_{K,1}$ is a seminorm for the topology on $\text{L}^{1}(\mathbb{T};\Gamma^{\nu}(E))$. Here we make the abuse of notation by possibly omitting extra bits of notation from the seminorm $p^{\nu}_{K,1}$, cf. (\ref{eq:3.2}).
    \item The \textbf{\textit{topology of pointwise convergence}} on $C^0(\text{L}^{1}_{\text{loc}}(\mathbb{T};\Gamma^{\nu}(E)))$ is the locally convex topology defined by the family of seminorms $p^{\nu}_{K,\mathbb{S},1,p}$, $p\in\mathcal{P}$, $K\subseteq M$ compact and $\mathbb{S}\subseteq\mathbb{T}$ compact, given by
    $$p^{\nu}_{K,\mathbb{S},1,p}(\xi)=p^{\nu}_{K,\mathbb{S},1}\circ \xi^p,$$
    where $p^{\nu}_{K,\mathbb{S},1}$ is a seminorm for the topology on $\text{L}^{1}_{\text{loc}}(\mathbb{T};\Gamma^{\nu}(E))$. Here we make the abuse of notation by possibly omitting extra bits of notation from the seminorm $p^{\nu}_{K,\mathbb{S},1}$, cf. (\ref{eq:3.2}).
    \item The \textbf{\textit{compact open topology}} on $C^0(\mathcal{P};\text{L}^{1}(\mathbb{T};\Gamma^{\nu}(E)))$ is the locally convex topology defined by the family of seminorms $p^{\nu}_{K,1,L}$, $L\subseteq\mathcal{P}$ compact and $K\subseteq M$ compact, given by
    $$p^{\nu}_{K,1,L}(\xi)=\sup\{ p^{\nu}_{K,1}\circ \xi^p\ |\ p\in L\},$$
    where $p^{\nu}_{K,1}$ is a seminorm for the topology on $\text{L}^{1}(\mathbb{T};\Gamma^{\nu}(E))$. If $\mathcal{P}$ is locally compact, we can make use of \citep[Corollary 16.6.3]{MR632257}, along with the discussion contained in \citep[\S 3.6.E]{MR632257}, to conclude that 
    $$C^0(\mathcal{P};\text{L}^{1}(\mathbb{T};\Gamma^{\nu}(E)))\simeq C^0(\mathcal{P})\check{\otimes}_\varepsilon \text{L}^{1}(\mathbb{T};\Gamma^{\nu}(E)),$$
    where $\check{\otimes}_\varepsilon$ is the injective tensor.
    \item The \textbf{\textit{compact open topology}} on $C^0(\mathcal{P};\text{L}^{1}_{\text{loc}}(\mathbb{T};\Gamma^{\nu}(E)))$ is the locally convex topology defined by the family of seminorms $p^{\nu}_{K,\mathbb{S}, 1,L}$, $L\subseteq\mathcal{P}$ compact, $K\subseteq M$ compact, and $\mathbb{S}\subseteq\mathbb{T}$ compact given by
    $$p^{\nu}_{K,\mathbb{S},1,L}(\xi)=\sup\{ p^{\nu}_{K,\mathbb{S},1}\circ \xi^p\ |\ p\in L\},$$
    where $p^{\nu}_{K,\mathbb{S},1}$ is a seminorm for the topology on $\text{L}^{1}_{\text{loc}}(\mathbb{T};\Gamma^{\nu}(E))$. If $\mathcal{P}$ is locally compact, we can make use of \citep[Corollary 16.6.3]{MR632257}, along with the discussion contained in \citep[\S 3.6.E]{MR632257}, to conclude that 
    $$C^0(\mathcal{P};\text{L}^{1}_{\text{loc}}(\mathbb{T};\Gamma^{\nu}(E)))\simeq C^0(\mathcal{P})\check{\otimes}_\varepsilon \text{L}^{1}_{\text{loc}}(\mathbb{T};\Gamma^{\nu}(E)),$$
    where $\check{\otimes}_\varepsilon$ is the injective tensor.
\end{enumerate}
The following result then characterises time-and parameter-dependent Lipschitz functions, just as in the time-dependent case.
\begin{lem}[Property of time- and parameter-dependent locally Lipschitz functions]\label{lem:5.3}
Let $M$ be a smooth manifold, let $\mathbb{T}$ be a time-domain, let $\mathcal{P}$ be a topological space, and let $f \in C^{\text{lip}}_{\text{PLI}}(\mathbb{T}; M; \mathcal{P})$. If $K \subseteq M$ is compact, if $\mathbb{S} \subseteq \mathbb{T}$ is a compact interval, and if $p_0 \in \mathcal{P}$, then there exists $C \in \R_{>0}$ and a neighbourhood $\mathcal{O}$ of $p_0$ such that
$$\int_{\mathbb{S}}|f(t,x_1,p)-f(t,x_2,p)|\d t\leq Cd_\mathbb{G}(x_1,x_2), \hspace{10pt} x_1,x_2\in K,\ p\in\mathcal{O}.$$
\end{lem}
\begin{proof}
Let $K \subseteq M$ and $\mathbb{S} \subseteq \mathbb{T}$ be compact, and let $p_0 \in \mathcal{P}$. Let $x \in K$ and, as in the proof of Lemma \ref{lem:2.4}, let $\mathcal{U}_x$ be a neighbourhood of $x$ and let $\ell_x \in L^1(\mathbb{S};\R_{\geq0})$ be such that
$$\text{dil}\  f(t, y, p_0) \leq \ell_x(t)d_\mathbb{G}(x_1, x_2),\hspace{10pt} (t, y) \in \mathbb{S} \times \text{cl}(\mathcal{U}_x).$$
According to (\ref{eq:2.4}), there exists a neighbourhood $\mathcal{O}_x$ of $p_0$ such that
\begin{equation*}
    \int_{\mathbb{S}} \text{dil}\ (f^p-f^{p_0})(t,y)\d t<1,\hspace{10pt}(t,y,p)\in\mathbb{S}\times\mathcal{U}_x\mathcal{O}_x.
\end{equation*}
Therefore, by the triangle inequality,
$$\int_{\mathbb{S}} \text{dil}\ f^p(t,y)\d t\leq\int_{\mathbb{S}} \text{dil}\ (f^p-f^{p_0})(t,y)\d t+\int_{\mathbb{S}} \text{dil}\ f^{p_0}(t,y)\d t<\underbrace{1+\int_{\mathbb{S}}\ell_x(t)\d t}_{C_x}$$
for all $(t, y, p) \in \mathbb{S} \times \mathcal{U}_x \times \mathcal{O}_x$. Thus, by Lemma \ref{lem:lsdasd}, there exists $C_x \in \R_{>0}$ such that
$$\int_{\mathbb{S}}\frac{f(t,x_1,p)-f(t,x_2,p)}{d_\mathbb{G}(x_1,x_2)}\d t\leq \int_{\mathbb{S}}\lambda^0_{\text{cl}(\mathcal{U}_x)}(f^p_t)\d t \leq C_x$$
for $x_1, x_2 \in \mathcal{U}_x$ distinct and for $p \in \mathcal{O}_x$.
By compactness of $K$, there exists $x_1, . . . , x_m \in K$ such that $K\subseteq \cup_{j=1}^{m}\mathcal{U}_{x_j}$.  By the Lebesgue Number Lemma \citep[Theorem 1.6.11]{MR1835418}, there exists $r \in \R_{>0}$ with the property that, if $x_1, x_2 \in K$ satisfies $d_\mathbb{G}(x_1, x_2) < r$, then there exists $j \in \{1, . . . , m\}$ such that $x_1, x_2 \in \mathcal{U}_{x_j}$. Since $C^0_{\text{LI}}(\mathbb{T}; M; \mathcal{P}) \subseteq C^{\text{lip}}_{\text{LI}}(\mathbb{T}; M; \mathcal{P})$, by (\ref{eq:2.4}) there exists a neighbourhood $\mathcal{O}'$ of $p_0$ such that
\begin{equation*}
    \int_{\mathbb{S}} |f(t,x,p)|\d t<1,\hspace{10pt}(t,y,p)\in\mathbb{S}\times K\times \mathcal{O'}.
\end{equation*}
Let
$$C=\max\left\{C_{x_1},...,C_{x_m},\frac{r}{2}\right\}$$
and let $\mathcal{O} = \mathcal{O'} \cap (\cap^m_{j=1}\mathcal{O}_{x_j})$. Let $x_1, x_2 \in K$ and $p \in \mathcal{O}$. If $d_{\mathbb{G}}(x_1, x_2) < r$, then let $j \in \{1, . . . , m\}$ be such that $x_1, x_2 \in \mathcal{U}_{x_j}$, and then we have
$$\int_{\mathbb{S}}|f(t, x_1, p) - f(t, x_2, p)| \d t \leq C_{x_j}d_\mathbb{G}(x_1, x_2) \leq Cd_\mathbb{G}(x_1, x_2).$$
If $d_\mathbb{G}(x_1, x_2) \geq r$, then
\begin{eqnarray*}
  \int_{\mathbb{S}}|f(t, x_1, p) - f(t, x_2, p)| \d t &\leq& \int_{\mathbb{S}}|f(t, x_1, p)| \d t + \int_{\mathbb{S}}|f(t, x_2, p)| \d t\\
  &<& 2 \leq \frac{2}{r}d_\mathbb{G}(x_1,x_2)\leq Cd_\mathbb{G}(x_1,x_2),
\end{eqnarray*}
which gives the result.
\end{proof}
\section{Vector fields and flows}\label{sec:3}
For the class of time- and parameter-dependent vector fields introduced in the previous section, we give a geometric characterisation of their integral curves and flows, and prove the more standard results on the manner in which the flow depends on its arguments. While the arguments used in the proofs of bear an unsurprising similar to the standard proofs, we highlight two important points of departure:
\begin{enumerate}
    \item in the time-dependent setting, we are using a class of vector fields that is new, so the proofs necessarily reflect this by being different than standard proofs;
    \item as throughout the chapter, we eschew the use of coordinates in favour of globally defined functions, as this is an important ingredient in our approach.
    \item give geometric proofs for standard results, globally expressed, concerning continuous (and more regular, when appropriate) dependence of terminal state on initial state, initial and final time.
    \item prove a new type of continuity result for the “parameter to local flow” mapping. 
\end{enumerate}
\subsection{Integral curves for vector fields}\label{sec:3.1}
In this section, we will give geometric definitions and characterisations of integral curves and flows. We begin by defining and characterising integral curves in our global framework.
\begin{defin}[Integral curve]
Let $m\in \mathbb{Z}_{\geq 0}$, let $m'\in\{0,\text{lip}\}$, let $\nu\in\{m+m',\infty,\omega,\text{hol}\}$, and let $r\in \{\infty,\omega,\text{hol}\}$, as required. Let $M$ be a $C^r$-manifold, let $\mathbb{T}\subseteq \R$ be an interval, let $X\in \Gamma^\nu_{\text{LI}}(\mathbb{T};TM)$. An \textbf{integral curve} for $X$ is a locally absolutely continuous curve $\xi:\mathbb{T'}\rightarrow M$ such that 
\begin{enumerate}[(i)]
    \item $\mathbb{T'}\subseteq \mathbb{T}$ and
    \item $\xi'(t)=X(t,\xi(t))$ for almost every $t\in\mathbb{T'}$.
\end{enumerate}
An integral curve $\xi:\mathbb{T'}\rightarrow M$ is \textbf{\textit{maximal}} if, given any other integral curve $\eta:\mathbb{T''}\rightarrow M$ for which $\eta(t)=\xi(t)$ for some $t\in\mathbb{T'}\cap\mathbb{T''}$, we have $\mathbb{T}''\subseteq\mathbb{T'}$.
\end{defin}
The following result, while admittedly simple, characterises integral curves in a way that will be of essential use to our approach.
\begin{lem}[``Weak" characterisation of integral curves]\label{lem:3.2}
Let $m\in \mathbb{Z}_{\geq 0}$, let $m'\in\{0,\text{lip}\}$, let $\nu\in\{m+m',\infty,\omega,\text{hol}\}$, and let $r\in \{\infty,\omega,\text{hol}\}$, as required. Let $M$ be a $C^r$ manifold, Stein when $\nu=\text{hol}$, let $\mathbb{T}\subseteq \R$ be an interval, and let $X\in \Gamma^\nu_{\text{LI}}(\mathbb{T};TM)$. For a curve $\xi:\mathbb{T'}\rightarrow M$, the following statements are equivalent:
\begin{enumerate}[(i)]
    \item\label{lem:3.2-1} $\xi $ is a integral curve for $X$;
    \item\label{lem:3.2-2} for any $t_0\in\mathbb{T'}$ and any $f\in C^\infty(M)$,
    $$f\circ \xi(t)=f\circ\xi(t_0)+\int^{t}_{t_0}Xf(s,\xi(s))\d s;$$
    \item\label{lem:3.2-3} for any $t_0\in\mathbb{T'}$ and any $f\in C^r(M)$,
    $$f\circ \xi(t)=f\circ\xi(t_0)+\int^{t}_{t_0}Xf(s,\xi(s))\d s.$$
\end{enumerate}
\begin{proof}
(\ref{lem:3.2-1})$\implies$ (\ref{lem:3.2-2}). Let $t_0\in\mathbb{T'}$ and $f\in C^{\infty}(M)$. Then we have 
\begin{align*}
    f\circ\xi(t_0)+\int^{t}_{t_0}Xf(s,\xi(s))\d s &= f\circ\xi(t_0)+\int^{t}_{t_0}\langle df(\xi(s));\xi'(s)\rangle\d s\\
    &= f\circ\xi(t_0)+\int^{t}_{t_0}\frac{\d}{\d s}f\circ\xi(s)\d s=f\circ\xi(t),
\end{align*}
as claimed. 

(\ref{lem:3.2-2})$\implies$ (\ref{lem:3.2-3}). This follows since real analytic and holomorphic functions are smooth.

(\ref{lem:3.2-3})$\implies$ (\ref{lem:3.2-1}). Let $t_0\in \mathbb{T'}$ and let $\chi^1,...,\chi^n\in C^{r}(M)$
be such that 
$$(\d\chi^1(\xi(t_0)),...,\d\chi^1(\xi(t_0)))$$
is a basis for $T^*_{\xi(t_0)}M$. The existence of such functions follows from the existence of globally defined $C^r$-coordinate functions about any point in M for smooth, real analytic, and Stein manifolds. One then argues, just as in the proof of the first implication above, that
$$\chi^j\circ \xi(t)=\chi^j\circ\xi(t_0)+\int^{t}_{t_0}X\chi^j(s,\xi(s))\d s$$
for each $j\in\{1,...,n\}$. This gives 
$$(\chi^j\circ\xi)'(t)=X\xi^j(t,\xi(t)),\hspace{2cm}j\in\{1,...,n\}, \ \text{a.e. } t\in\mathbb{T'}.$$
The linear independence of $\d\chi^1,...,\d\chi^n$ in a neighbourhood of $\xi(t_0)$ gives $\xi'(t)=X(t,\xi(t))$ for almost every $t$ in some neighbourhood of $t_0$. As this holds for every $t_0\in\mathbb{T'}$, we conclude that $\xi$ is an integral curve for $X$. 
\end{proof}
\end{lem}
The next lemma is a sort of adaptation of the preceding lemma for curves that are not necessarily integral curves of vector fields. In the statement of the result, we make use of the notation
$$\mathbb{T}_{t_0,\alpha}=\mathbb{T}\cap [t_0-\alpha,t_0+\alpha].$$
\begin{lem}[Curves determined by functions]\label{lem:3.3}
Let $M$ be a manifold of class $C^\infty$, let $\mathbb{T} \subseteq \R$ be an interval, and let $\eta^j \in C^0(\mathbb{T}; \R)$. Let $x_0 \in M$ and suppose that $\chi^1,. . ., \chi^n \in C^\infty(M)$ are such that $(\d\chi^1(x_0), . . . , \d\chi^n(x_0))$ is a basis for $T^*_{x_0}M$. If $\eta^j(t_0) = \chi^j(x_0)$, then the following statements hold:
\begin{enumerate}[(i)]
    \item\label{lem:3.3-1} there exists $\alpha \in \R_{>0}$ and $\gamma \in C^0(\mathbb{T}_{t_0,\alpha}; M)$ such that
    $$\chi^j\circ \gamma(t) = \eta^j(t), \hspace{10pt} j \in \{1, . . . , n\},\  t \in \mathbb{T}_{t_0,\alpha};$$
    \item with $\chi^1, . . . , \chi^n$ and $\gamma$ as in part (\ref{lem:3.3-1}), $\alpha$ can be chosen so that the curve $\gamma$ is unique in that any two such curves agree on the intersection of their domains.
\end{enumerate}
\end{lem}
\begin{proof}
Consider the map 
\begin{eqnarray*}
    f:M &\rightarrow&  \R^n\\
	  x&\mapsto& (\chi^1(x),...,\chi^n(x)).
\end{eqnarray*}
By the Inverse Function Theorem, $f$ is a diffeomorphism from a neighbourhood $\mathcal{U}$ of $x_0$ onto a neighbourhood $\mathcal{V}$ of $f(x_0)$. Let $\alpha \in \R_{>0}$ be sufficiently small that
$$(\eta^1(t),...,\eta^n(t))\in\mathcal{V},\hspace{10pt}j\in\{1,...,n\},\ t\in\mathbb{T}_{t_0,\alpha}.$$
Thus there is a unique continuous curve $\gamma : \mathbb{T}_{t_0,\alpha} \rightarrow \mathcal{U}$ satisfying $\chi^j \circ \gamma(t) = \eta^j(t)$ for
$j \in \{1, . . . , n\}$ and $t \in \mathbb{T}_{t_0,\alpha}$.
\end{proof}
The next step should be that of existence and uniqueness of integral curves for time-dependent vector fields (or time- and parameter-dependent vector fields with the parameter fixed). This, however, requires no special measures since the condition that $\Gamma^{\text{lip}}_{\text{LI}}(\mathbb{T};TM)$ returns exactly the condition required for existence and uniqueness of local integral curves, and, hence, also maximal integral curves, i.e., those defined on the largest possible time interval. Thus we shall not give an independent proof here. However, we will give a  a precise formulation for this in the presence of parameters, with our particular form of parameter-dependence, and we shall take for granted the usual existence and uniqueness results for maximal integral curves. 
\subsection{Flows for vector fields}\label{sec:3.2}
With the notion of an integral at hand, we can define flows of time-varying vector fields. 
\begin{defin}[Domain of a vector field, flow of a vector field]
Let $m\in \mathbb{Z}_{\geq 0}$, let $m'\in\{0,\text{lip}\}$, let $\nu\in\{m+m',\infty,\omega,\text{hol}\}$, and let $r\in \{\infty,\omega,\text{hol}\}$, as required. Let $M$ be a $C^r$ manifold, let $\mathbb{T}\subseteq \R$ be a time-domain, let $\mathcal{P}$ be a topological space, and let $X\in \Gamma^\nu_{\text{PLI}}(\mathbb{T};TM;\mathcal{P})$. 
 \begin{enumerate}[(i)]
    \item For $(t_0,x_0,p_0)\in \mathbb{T}\times M\times\mathcal{P}$, denote
     \begin{eqnarray*}
      J_X(t_0,x_0,p_0)=\bigcup\{J\subseteq \mathbb{T}\;|\; J \text{ is an interval and there exists an integral curve }\\
       \xi:J\rightarrow M\text{ for $X^{p_0}$ satisfying }\xi(t_0)=x_0\}.
     \end{eqnarray*}
     The interval $j_X(t_0, x_0, p_0)$ is the interval of existence for $X^{p_0}$ for the initial condition $(t_0,x_0)$.
    \item For $(t_1,x_0,p_0)\in \mathbb{T}\times M\times\mathcal{P}$, denote
    $$I_X(t_1,x_0)=\{t_0\in \mathbb{T}\;|\;t_1\in J_X(t_0,x_0,p_0)\}.$$
    \item For For $(t_1,x_0,p_0)\in \mathbb{T}\times M\times\mathcal{P}$, denote 
    $$D_X(t_1,t_0,p_0)=\{x\in M\;|\;t_1\in J_X(t_0,x,p_0)\}.$$
	\item Denote 
    $$D_X=\{(t_0,t_0,x_0,p_0)\in \mathbb{T}\times\mathbb{T}\times M\times\mathcal{P}\;|\;t_1\in J_X(t_0,x_0,p_0)\}.$$
	\item The \textbf{\textit{flow}} for $X$ is the mapping
      \begin{eqnarray*}
      \Phi^X:D_X &\rightarrow&  M\\
	         (t_1,t_0,x_0,p_0)&\mapsto& \xi(t_1),
      \end{eqnarray*}
    where $\xi$ is the integral curve for $X^{p_0}$ satisfying $\xi(t_0)=x_0$.
\end{enumerate} 
\end{defin}
We will work with time-varying vector fields both with and without parameter dependence. When we work with vector fields that are time-dependent but not parameter-dependent, we will simply omit the argument corresponding to the parameter without further mention. We shall also use the notation
$$\Phi^p_{t_1,t_0}:D_X(t_1,t_0,p)\rightarrow M$$
when convenient. 

The flow has the following elementary properties that follow from the definitions and by the uniqueness of integral curves:
\begin{enumerate}[(i)]
    \item for each $(t_0, x_0)\in \mathbb{T} \times M$, $(t_0, t_0, x_0) \in D_X$ and $\Phi^X(t_0, t_0, x_0) = x_0$;
    \item if $(t_2, t_1, x) \in D_X$, then $(t_3, t_2, \Phi^X(t_2, t_1, x)) \in D_X$ if and only if $(t_3, t_1, x) \in D_X$ and, if this holds, then
    $$\Phi^X(t_3, t_1, x) = \Phi^X(t_3, t_2, \Phi^X(t_2, t_1, x));$$
    \item if $(t_2, t_1, x) \in D_X$, then $(t_1, t_2, \Phi^X(t_2, t_1, x)) \in D_X$ and $\Phi^X(t_1, t_2, \Phi^X(t_2, t_1, x)) = x$.
\end{enumerate}
We can now state a “standard” theorem in our nonstandard framework. In the statement and proof of the result, we will find the notation
$$
|t_0,t_1|=
\begin{cases}
[t_0,t_1], \ \ \ t_1\ge t_0,\\ [t_1,t_0], \ \ \ t_1\le t_0.
\end{cases}
$$
useful, for $t_0,t_1\in\mathbb{T}$. We also denote $\text{LocFlow}^\nu(\mathbb{S'};\mathbb{S};\mathcal{U})$ the set of local flows defined on $\mathbb{S'}\times\mathbb{S}\times \mathcal{U}\subseteq\mathbb{T}\times\mathbb{T}\times M$, and
$$\mathcal{V}^\nu_{\mathbb{S'}\times\mathbb{S}\times\mathcal{U}}=\{X\in\Gamma^\nu_{\text{LI}}(\mathbb{T};TM) \ |\ \mathbb{S'}\times \mathbb{S}\times \mathcal{U}\subseteq D_X,\;\;\mathbb{S}\subseteq\mathbb{S'}\},$$
the space of time-verying sections whose flows are defined on $\mathbb{S'}\times \mathbb{S}\times \mathcal{U}$.

\begin{thm}[Continuous dependence of flows]\label{thm:3.4}
Let $M$ be a $C^\infty$-manifold, let $\mathbb{T}\subseteq \R$ be an interval, let $\mathcal{P}$ be a topological space, and let $X\in \Gamma^{\text{lip}}_{\text{PLI}}(\mathbb{T};TM;\mathcal{P})$.  Then the following statements hold:
\begin{enumerate}[(i)]
    \item\label{thm:3.4-1} for $(t_0, x_0, p_0) \in \mathbb{T} \times M \times \mathcal{P}$, $J_X(t_0, x_0, p_0)$ is a relatively open subinterval of $\mathbb{T}$;
    \item\label{thm:3.4-2} for $(t_0, x_0, p_0) \in \mathbb{T} \times M \times \mathcal{P}$, the curve
    \begin{eqnarray*}
       \gamma_{(t_0,x_0,p_0)}: J_X(t_0, x_0, p_0) &\rightarrow& M\\
       t&\mapsto & \Phi^X(t, t_0, x_0, p_0)
    \end{eqnarray*}
    is well-defined and locally absolutely continuous;
    \item\label{thm:3.4-3} for $(t_1, x_0, p_0) \in \mathbb{T} \times M \times \mathcal{P}$, $I_X(t_1, x_0, p_0)$ is a relatively open subinterval of $\mathbb{T}$;
    \item\label{thm:3.4-4} for $(t_1, x_0, p_0) \in\mathbb{T} \times M \times \mathcal{P}$, the curve
    \begin{eqnarray*}
       \iota_{(t_1,x_0,p_0)}: I_X(t_1, x_0, p_0) &\rightarrow& M\\
       t&\mapsto & \Phi^X(t_1, t, x_0, p_0)
    \end{eqnarray*}
    is well-defined and locally absolutely continuous;
    \item\label{thm:3.4-5} for $t_1, t_0 \in \mathbb{T}$ and $p_0 \in \mathcal{P}$, $D_X(t_1, t_0, p_0)$ is open in M;
    \item\label{thm:3.4-6} for $t_1, t_0 \in \mathbb{T}$ and $p_0 \in \mathcal{P}$ for which $D_X(t_1, t_0, p_0) \neq= \emptyset$, $\Phi^{X^{p_0}}_{t_1,t_0}$ is a locally bi-Lipschitz homeomorphism onto its image;
    \item\label{thm:3.4-7} $D_x$ is relatively open in $\mathbb{T}\times\mathbb{T}\times M\times\mathcal{P}$;
    \item\label{thm:3.4-8} the map 
    \begin{eqnarray*}
      \Phi^X:D_X \rightarrow  M
    \end{eqnarray*}
    is continuous;
    \item\label{thm:3.4-9} for for $(t_0, x_0, p_0) \in\mathbb{T} \times M \times \mathcal{P}$, and for $\epsilon\in\R_{>0}$, there exists $\alpha\in\R_{>0}$, a neighbourhood $\mathcal{U}$ of $x_0$, and a neighbourhood $\mathcal{O}$ of $p_0$ such that
    $$\sup J_X(t,x,p)>\sup J_X(t_0,x_0,p_0)-\epsilon,\hspace{1cm} \inf J_X(t,x,p)<\inf J_X(t_0,x_0,p_0)+\epsilon$$
    for all $(t,x,p)\in\rm{int}(\mathbb{T}_{t_0,\alpha})\times\mathcal{U}\times\mathcal{O}$;
    \item\label{thm:3.4-10} for $(t_1, t_0, x_0, p_0) \in D_X$, the curves
    $$|t_0,t_1|\ni t\mapsto \Phi^X(t,t_0,x,p)\in M$$
    converge uniformly to
    $$|t_0,t_1|\ni t\mapsto \Phi^X(t,t_0,x_0,p_0)\in M$$
    as $(x,p)\to (x_0,p_0)$.
\end{enumerate}
\end{thm}
\begin{proof} (\ref{thm:3.4-1}) Since $J_X(t_0, x_0, p_0)$ is a union of intervals, each of which contains $t_0$, it follows that it is itself an interval. To show that it is an open subset of $\mathbb{T}$, we show that, if $t \in J_X(t_0, x_0, p_0)$, there exists $\epsilon \in \R>0$ such that $\mathbb{T}_{t,\epsilon} \subseteq J_X(t_0, x_0, p_0)$.

First let us consider the case when $t$ is not an endpoint of $\mathbb{T}$, in the event that $\mathbb{T}$ contains one or both of its endpoints. In this case, by definition of $J_X(t_0, x_0, p_0)$, there is an open interval $J \subseteq \mathbb{T}$ containing $t_0$ and $t$, and an integral curve $\xi: J \rightarrow M$ for $X_{p_0}$ satisfying $\xi(t_0) = x_0$. In particular, there exists $\epsilon \in \R_{>0}$ such that $(t-\epsilon, t + \epsilon) \subseteq J \subseteq J_X(t_0, x_0, p_0)$, which gives the desired conclusion in this case.

Next suppose that $t$ is the right endpoint of $\mathbb{T}$, which we assume is contained in $\mathbb{T}$, of course. In this case, by definition of $J_X(t_0, x_0, p_0)$, there is an interval $J \subseteq \mathbb{T}$ containing $t_0$ and $t$, and an integral curve $\xi: J \rightarrow M$ for $X_{p_0}$ satisfying $\xi(t_0) = x_0$. In this case, there exists $\epsilon \in \R_{>0}$ such that
$$\mathbb{T}_{t,\epsilon}=(t-\epsilon,t]\subseteq J_X(t_0, x_0, p_0),$$
which gives the desired conclusion in this case. A similar argument gives the desired conclusion when $t$ is the left endpoint of $\mathbb{T}$.

(\ref{thm:3.4-2}) That $\gamma_{(t_0,x_0,p_0)}$ is defined in $J_X(t_0, x_0, p_0)$ was proved as part of the preceding part of the proof. The assertion that $\gamma_{(t_0,x_0,p_0)}$ is locally absolutely continuous follows from the
usual existence and uniqueness theorem.

(\ref{thm:3.4-3}) This follows similarly to part (\ref{thm:3.4-2}).

We will defer the proof of part (\ref{thm:3.4-4}) to the end of the proof.

We shall prove the assertions (\ref{thm:3.4-5}) and (\ref{thm:3.4-6}) of the theorem together, first by proving that these conditions hold locally, and then giving an extension argument to give the global version of the statement. Let us first prove a few technical lemmata that will be useful to us. First we give the initial part of the local version of the theorem.
\begin{lem}\label{lem:3.6}
Let $M$ be a $C^{\infty}$-manifold, let $\mathbb{T}\subseteq \mathbb{R}$ be a time domain, let $\mathcal{P}$ be a topological space, and let $X\in \Gamma^{\text{lip}}_{LI}(\mathbb{T},TM;\mathcal{P})$. For each $(t_0,x_0,p_0)\in \mathbb{T}\times M\times\mathcal{P}$, there exist $\alpha\in \mathbb{R}_{>0}$, a neighborhood $\mathcal{U}\subseteq M$ of $x_0$ and a neighborhood $\mathcal{O}\subseteq\mathcal{P}$ of $p_0$ such that, $(t,t_0,x,p)\in D_X$ for each $t\in \mathbb{T}_{t_0,\alpha}$, $x\in \mathcal{U}$ and $p\in\mathcal{O}$. Moreover, $\alpha$, $\mathcal{U}$ and $\mathcal{O}$ can be chosen such that:
\begin{enumerate}[(i)]
    \item\label{lem:3.6-1} the map
	$$\mathcal{U}\ni x\mapsto \Phi^{X^p}_{t_1,t_0}(x)$$
	is Lipschitz for every $p\in \mathcal{O}$ and every $t_1\in \mathbb{T}_{t_0,\alpha}$.
    \item\label{lem:3.6-2} the map 
	$$\mathbb{T}_{t_0,\alpha}\times \mathcal{U}\times\mathcal{O}\ni (t,x,p)\mapsto \Phi^X(t,t_0,x,p)$$
	is continuous.
\end{enumerate}
\end{lem}
\renewcommand\qedsymbol{$\nabla$}
\begin{proof}
We first essentially prove the local existence and uniqueness result, including the role of parameters. We make use of an arbitrarily selected Riemannian metric $\mathbb{G}$. Let $\chi^1,...,\chi^n\in C^\infty(M)$ be s.t. $\{\d\chi^1(x_0),...,\d\chi^n(x_0)\}$ is a basis for $T^*_{x_0}M$. Let $R\in \mathbb{R}_{>0}$ be s.t.
$$\mathcal{U}:=\{x\in M|d_G(x,x_0)<R\}$$
is geodesically convex \citep[Proposition IV.3.4]{MR0152974}. We choose $R$ sufficiently small that $\d\chi^1,...,\d\chi^n$ are linearly independent at points in $\mathcal{U}$. By Lemma \ref{lem:5.2}, there exists $C\in \mathbb{R}_{>0}$, such that
\begin{eqnarray}\label{eq:x1x2}
   &&C^{-1}\sup\{|\chi^j(x_1)-\chi^j(x_2)|\ \suchthat j\in\{1,...,n\}\}\nonumber\\
   &&\hspace{2cm}\leq d_G(x_1,x_2)\leq C\sup\{|\chi^j(x_1)-\chi^j(x_2)|:j\in\{1,...,n\}\},\;\;\;\;x_1,x_2\in \mathcal{U}
\end{eqnarray}
For $x\in M$ and $a\in\mathbb{R}_{>0}$, we denote by 
$$\mathcal{U}(a,x)\subseteq\cap_{j=1}^n(\chi^j)^{-1}(\chi^j(x)-a,\chi^j(x)+a)$$
the connected component of the set on the right containing $x$. With $\mathcal{U}$ chosen above, note that $\mathcal{U}(a,x)$ is a neighborhood, homeomorphic to an n-dimensional ball, of $x$ for $x\in\mathcal{U}$ and for $a$ sufficiently small that $\mathcal{U}(a,x)\subseteq \mathcal{U}$.

Following from \citep[Theorem 6.4: (v) and (vi)]{Jafarpour2014} and making use of the universal property of the initial topology, one can easily see that $X\chi^j\in C^{\text{lip}}_{\text{PLI}}(\mathbb{T};M;\mathcal{P})$ for $j\in\{1,...,n\}$. Let $r'\in\mathbb{R}_{>0}$ be such that $\mathcal{U}(r',x_0)\subseteq \mathcal{U}$. Let $r=\frac{r'}{2}$, and let $\lambda\in(0,1)$. There exists $\alpha\in \mathbb{R}_{>0}$ s.t.
\begin{equation}\label{eq:1.1}
\int_{\mathbb{T}_{t_0,\alpha}}|X\chi^j(s,x,p_0)|ds<\frac{r}{2}
\end{equation}
and
\begin{equation}\label{eq:1.2}
\int_{\mathbb{T}_{t_0,\alpha}}\text{dil} (X\chi^j)(s,x,p_0)ds<\frac{\lambda}{2C}
\end{equation}
for $x\in \mathcal{U}(r',x_0)$ and $j\in\{1,...,n\}$. Since $X\chi^j\in C^{\text{lip}}_{\text{PLI}}(\mathbb{T};M;\mathcal{P})$, $j\in\{1,...,n\}$, there exists a neighbourhood $\mathcal{O}$ of
$p_0$ such that
\begin{equation}\label{eq:1.3}
\int_{\mathbb{T}_{t_0,\alpha}}|X\chi^j(s,x,p)-X\chi^j(s,x,p_0)|ds<\frac{r}{2},
\end{equation}
and 
\begin{equation}\label{eq:1.4}
\int_{\mathbb{T}_{t_0,\alpha}}\text{dil} ((X\chi^j)^p-(X\chi^j)^{p_0})(s,x)ds<\frac{\lambda}{2C}
\end{equation}
for $x\in \mathcal{U}(r',x_0)$, $p\in \mathcal{O}$, and $j\in\{1,...,n\}$.

Applying triangle inequality to (\ref{eq:1.1})-(\ref{eq:1.4}),
\begin{equation}\label{eq:1.5}
\int_{\mathbb{T}_{t_0,\alpha}}|X\chi^j(s,x,p)|ds<r,
\end{equation}
and
\begin{equation}\label{eq:1.6}
\int_{\mathbb{T}_{t_0,\alpha}}\text{dil} (X\chi^j)(s,x,p)ds<\frac{\lambda}{C},
\end{equation}
for all $x\in \mathcal{U}(r',x_0)$, $p\in \mathcal{O}$, and $j\in\{1,...,n\}$.

The inequality (\ref{eq:1.6}) with the aid of Lemma \ref{lem:5.3}, give \begin{equation}\label{eq:1.7}
\int_{|t_0,t|}|X\chi^j(s,x_1,p)-X\chi^j(s,x_2,p)|ds<\frac{\lambda}{C}d_G(x_1,x_2),
\end{equation}
for $x_1,x_2\in \mathcal{U}(r',x_0)$, $j\in\{1,...,n\}$, and $p\in \mathcal{O}$.

If $y\in \mathcal{U}(r,x_0)$, then $\mathcal{U}(r,y)\subseteq \mathcal{U}(r',x_0)$. Therefore, (\ref{eq:1.5})-(\ref{eq:1.7}) hold for all $x,x_1,x_2\in \mathcal{U}(r,y)$, $t\in \mathbb{T}_{t_0,\alpha}$, $j\in\{1,...,n\}$ and $p\in \mathcal{O}$.

Denote $\mathcal{U}:=\mathcal{U}(r,x_0)$, and let $x\in \mathcal{U}$. For each $j\in \{1,...,n\}$, we denote by $\phi^j_0\in C^0(\mathbb{T}_{t_0,\alpha},\mathbb{R})$ the constant mapping $\phi ^j_0=\chi^j(x)$, and denote by $\overbar{B}(r,\phi^j_0)$ the closed ball of radius $r$ about $\phi^j_0$. For each $p\in \mathcal{O}$, we have the mapping
\begin{eqnarray*}
   F^j_{X^p}:&&\overbar{B}(r,\phi^j_0)\rightarrow C^0(\mathbb{T}_{t_0,\alpha},\mathbb{R})\\
	         &&\phi^j\mapsto \chi^j(x)+\int_{|t_0,t|}X\chi^j(s,\gamma_{\phi}(s),p)\d s,
\end{eqnarray*}
for $j\in \{1,...,n\}$, and where $\gamma_\phi \in C^0(\mathbb{T}_{t_0,\alpha},M)$ is the unique curve satisfying 
\begin{equation*}
\chi^j\circ \gamma_\phi=\phi^j(t),\;\;\;\;\;t\in \mathbb{T}_{t_0,\alpha},\;j\in\{1,...,n\},
\end{equation*}
cf. Lemma \ref{lem:3.3}. Note that the definition of $r$ ensures that $\gamma_\phi$ so defined takes values in $\mathcal{U}$.

Now we claim that $\text{Im}(\gamma_\phi)\in \mathcal{U}(r',x_0)$. Indeed, since $x\in \mathcal{U}=\mathcal{U}(r,x_0)$, then $$\mathcal{U}(r,x)\subseteq\cap_{j=1}^n(\chi^j)^{-1}(\chi^j(x)-r,\chi^j(x)+r)\subseteq \mathcal{U}(r',x_0).$$
Since $\phi^j\in \overbar{B}(r,\phi^j_0)$ and $\gamma_\phi (t)=(\chi^j)^{-1}\circ \phi^j(t)$ for all $j\in\{1,...,n\}$, then 
$$\gamma_\phi(t)\in \cap_{j=1}^n(\chi^j)^{-1}(\phi_0^j(t)-r,\phi_0^j(t)+r)=\cap_{j=1}^n(\chi^j)^{-1}(\chi^j(x)-r,\chi^j(x)+r)\subseteq\mathcal{U}(r',x_0).$$

We then claim that $F^j_{X^p}(\overbar{B}(r,\phi^j_0))\subseteq\overbar{B}(r,\phi^j_0),\  j\in\{1,...,n\}, p\in\mathcal{O}$. Indeed, if $\phi^j\in\overbar{B}(r,\phi^j_0),\ j\in\{1,...,n\}$, we have 
\begin{equation*}
    |F^j_{X^p}\circ\phi^j(t)-\phi_0^j(t)|\leq\int_{|t_0,t|}|X\chi^j(s,\gamma_{\phi}(s),p)|\d s< r
\end{equation*}

We also claim that the mapping 
\begin{equation}\label{eq:3.8}
    \prod\limits_{j=1}^{n} \overbar{B}(r,\phi^j_0)\ni(\phi^1,...,\phi^n)\mapsto (F^1_{X^p}\circ\phi^1,...,F^n_{X^p}\circ\phi^n)\in \prod\limits_{j=1}^{n} \overbar{B}(r,\phi^j_0)
\end{equation}
is a contraction mapping for each $p\in\mathcal{O}$, where $\prod\limits_{j=1}^{n}\overbar{B}(r,\phi^j_0)$
is given the product metric. Indeed, let $\phi^j_1,\phi^j_2\in \overbar{B}(r,\phi^j_0),\ j\in\{1,...,n\}$. Let $\gamma_1,\gamma_2\in C^0(\mathbb{T}_{t_0,\alpha},M)$ be the corresponding curves satisfying 
$$\chi^j\circ \gamma_a(t)=\phi_a^j(t),\hspace{10pt}t\in \mathbb{T}_{t_0,\alpha},\;j\in\{1,...,n\},\ a\in \{1,2\},$$
cf. Lemma \ref{lem:3.3}. Then we have, for each $j\in\{1,...,n\}$, 
\begin{eqnarray*}
   |F^j_{X^p}\circ\phi^j_1(t)-F^j_{X^p}\circ\phi^j_2(t)|
   &\leq&\int_{|t_0,t|}|X\chi^j(s,\gamma_1(s),p)-X\chi^j(s,\gamma_2(s),p)|\d s\\
   &\leq&\frac{\lambda}{C}\sup\{d_\mathbb{G}(\gamma_1(s),\gamma_2(s))\ \suchthat\ s\in|t_0,t|\}\\
   &\leq&\lambda \sup\{|\phi^k_1(s)-\phi^k_2(s)|\ \suchthat \ k\in\{1,...,n\}\},\;s\in\mathbb{T}_{t_0,\alpha}, 
\end{eqnarray*}
from which the desired conclusion follows.

By the Contraction Mapping Theorem \citep[Theorem 1.2.6]{MR960687} there exists a unique fixed point for the mapping (\ref{eq:3.8}) in $\prod\limits_{j=1}^{n} \overbar{B}(r,\phi^j_0)$; let us denote the components of this unique fixed point by $\phi^j,\ j\in\{1,...,n\}$. Let us also denote by $\xi\in C^0(\mathbb{T}_{t_0,\alpha};M)$ the corresponding curve in $M$, noting that 
$$\chi^j\circ\xi(t)=\chi^j(x)+\int_{|t_0,t|}X\chi^j(s,\xi(s),p)\d s,$$
for all $p\in \mathcal{O}$, cf. Lemma \ref{lem:3.3}. It remains to show that the $\xi$ is an integral curve for $X^p$
satisfying $\xi(t_0) = x$. $\xi(t_0) = x$ is obvious. We can, then, follow the proof of part (\ref{lem:3.2-3}) of Lemma \ref{lem:3.2} to see
that $\xi$ is an integral curve for $X^p$.

We next prove uniqueness of this integral curve on $\mathbb{T}_{t_0,\alpha}$. Suppose that $\eta: \mathbb{T}_{t_0,\alpha} \rightarrow M$
is another integral curve satisfying $\eta(t_0) = x_0$. By Lemma \ref{lem:3.2} we have
$$f\circ\eta(t)=f(x)+\int_{|t_0,t|}Xf(s,\eta(s))\d s$$
for every $f\in C^\infty(M)$. It then follows that, if we define
\begin{eqnarray*}
\phi^j:\mathbb{T}_{t_0,\alpha} &\rightarrow&  \R\\
	  t&\mapsto& \chi^j\circ \eta(t),
\end{eqnarray*}
then $(\phi^1,...,\phi^n)$ is a fixed point of the mapping (\ref{eq:3.8}). Since this fixed point is unique, we must have
$$\chi^j\circ \eta(t) = \chi^j\circ\xi(t),\hspace{10pt} j \in \{1, . . . , n\}, t \in \mathbb{T}_{t_0,\alpha}.$$ 
By Lemma \ref{lem:3.3} we conclude that $\eta = \xi$. One can also prove global uniqueness of integral curves using the standard arguments.

From the above, we conclude that
$$\mathbb{T}_{t_0,\alpha}\times\{t_0\}\times \mathcal{U}\times\mathcal{O}\subseteq D_X$$
and that
$$f\circ\Phi^X(t,t_0,x,p)=f(x)+\int_{|t,t_0|}Xf(s,\Phi^X(s,t_0,x,p),p)\d s$$
for $(t,x,p)\in\mathbb{T}_{t_0,\alpha}\times \mathcal{U}\times\mathcal{O}$ and $f\in C^{\infty}(M)$. This proves the existential part of the
lemma.

(\ref{lem:3.6-1}) Fix $p\in\mathcal{O}$. By (\ref{eq:x1x2}) and (\ref{eq:1.7}), we have 
\begin{align*}
    &\int_{|t_0,t|}|X^p\chi^j(s,x_1)-X^p\chi^j(s,x_2)|\d t\\
     &\hspace{4cm}\leq \lambda\max\{|\chi^l(x_1)-\chi^l(x_2)|\ \suchthat\ l\in\{1,...,k\}\},\ t\in\mathbb{T}_{t_0,\alpha},\ x_1,x_2\in\mathcal{U}.
\end{align*}

Let $t\in \mathbb{T}_{t_0,\alpha}$ be such that $t\geq t_0$ and let $x_1,x_2\in \mathcal{U}$. We then have
$$f\circ\Phi^{X^p}(t,t_0,x_a)=f(x_a)+\int_{|t,t_0|}Xf(s,\Phi^{X^p}(s,t_0,x_a))\d s,\;\;\;\;a\in\{1,2\},\ f\in C^\infty(M).$$
Thus, for $j\in\{1,...,n\}$,
\begin{eqnarray*}
   &&|\chi^j\circ\Phi^{X^p}(t,t_0,x_1)-\chi^j\circ\Phi^{X^p}(t,t_0,x_2)| \\
   &&\hspace{16pt}\leq |\chi^j(x_1)-\chi^j(x_2)|+\int_{|t_0,t|}|X^p\chi^j(s,\Phi^{X^p}(s,t_0,x_1))-X^p\chi^j(s,\Phi^{X^p}(s,t_0,x_2))|\d s\\
   &&\hspace{16pt}\leq |\chi^j(x_1)-\chi^j(x_2)|\\
   &&\hspace{24pt}+\lambda\sup\left\{|\chi^l\circ \Phi^{X^p}(s,t_0,x_1)-\chi^l\circ \Phi^{X^p}(s,t_0,x_2)|\ \suchthat\ s\in[t_0,t],\ l\in\{1,...,n\}\right\}
\end{eqnarray*}
Abbreviate
$$\xi^p(s)=\max\left\{|\chi^l\circ \Phi^{X^p}(s,t_0,x_1)-\chi^l\circ \Phi^{X^p}(s,t_0,x_2)|\ \suchthat\ l\in\{1,...,n\}\right\}$$
and
$$\delta^p=\sup\{\xi^p(s)\ \suchthat\ s\in\mathbb{T}_{t_0,\alpha}\}$$
The definitions then give
$$\delta^p\leq \xi^p(t_0)+\lambda \delta^p \implies \delta^p\leq (1-\lambda)^{-1}\xi^p(t_0).$$
Since 
$$\xi^p(t_0)=\max\left\{|\chi^j(x_1)-\chi^j(x_2)|\ \suchthat\ j\in\{1,...,n\}\right\},$$
together with (\ref{eq:x1x2}), we have
$$d_\mathbb{G}(\Phi^{X^p}(t, t_0, x_1), \Phi^{X^p}(t, t_0, x_2)) \leq C\xi^p(t) \leq C\delta^p \leq (1 - \lambda)^{-1}d_\mathbb{G}(x_1, x_2),$$
which shows that $\Phi^{X^p}_{t,t_0}|\mathcal{U}$ is Lipschitz. Incidentally, the Lipschitz constant is independent of $t \in \mathbb{T}_{t_0,\alpha}$ and $p\in \mathcal{O}$.

(\ref{lem:3.6-2}) Let $t_1,t_2\in \mathbb{T}_{t_0,\alpha}$ be such that $t_0\leq t_1\leq t_2$, let $x_1,x_2\in \mathcal{U}$, and let $p_1,p_2\in\mathcal{O}$. We first have, for $t\in[t_0,t_1]$,
\begin{eqnarray*}
   &&|\chi^j\circ\Phi^X(t,t_0,x_1,p_1)-\chi^j\circ\Phi^X(t,t_0,x_2,p_2)| \\
   &&\hspace{5pt}\leq|\chi^j(x_1)-\chi^j(x_2)| +\int_{|t_0,t|}|X\chi^j(s,\Phi^X(s,t_0,x_1,p_1),p_1)-X\chi^j(s,\Phi^X(s,t_0,x_2,p_2),p_2)|\d s\\
   &&\hspace{5pt}\leq|\chi^j(x_1)-\chi^j(x_2)| +\int_{|t_0,t|}|X\chi^j(s,\Phi^X(s,t_0,x_1,p_1),p_1)-X\chi^j(s,\Phi^X(s,t_0,x_2,p_2),p_1)|\d s\\
   &&\hspace{10pt}+\int_{|t_0,t|}|X\chi^j(s,\Phi^X(s,t_0,x_2,p_2),p_1)-X\chi^j(s,\Phi^X(s,t_0,x_2,p_2),p_2)|\d s
\end{eqnarray*}
for $j\in\{1,...,n\}$. Hence for $j\in\{1,...,n\}$ and $t\in[t_0,t_1]$, we use (\ref{eq:x1x2}) and (\ref{eq:1.7}) to give 
\begin{eqnarray*}
   &&\int_{|t_0,t|}|X\chi^j(s,\Phi^X(s,t_0,x_1,p_1),p_1)-X\chi^j(s,\Phi^X(s,t_0,x_2,p_2),p_1)|\d s\\
   &&\hspace{40pt}\leq \lambda\max\left\{|\chi^l\circ \Phi^{X}(s,t_0,x_1,p_1)-\chi^l\circ \Phi^{X}(s,t_0,x_2,p_2)|\ \suchthat\ l\in\{1,...,n\}\right\}
\end{eqnarray*}
We also clearly have
\begin{eqnarray*}
   &&\int_{|t_0,t|}|X\chi^j(s,\Phi^X(s,t_0,x_2,p_2),p_1)-X\chi^j(s,\Phi^X(s,t_0,x_2,p_2),p_2)|\d s\leq \rho\overset{\Delta}{=}\\
   &&\max\left\{\int_{\mathbb{T}_{t_0,\alpha}}|X\chi^j(s,\Phi^X(s,t_0,x_2,p_2),p_1)-X\chi^j(s,\Phi^X(s,t_0,x_2,p_2),p_2)|\d s\ \suchthat j\in\{1,...,n\}\right\}
\end{eqnarray*}
Let us denote
$$\xi(s)=\max\left\{|\chi^l\circ \Phi^{X^p}(s,t_0,x_1,p_1)-\chi^l\circ \Phi^{X}(s,t_0,x_2,p_2)|\ \suchthat\ l\in\{1,...,n\}\right\}$$
and
$$\delta=\sup\{\xi(s)\ \suchthat\ s\in\mathbb{T}_{t_0,\alpha}\}$$
so that 
$$\delta\leq \xi(t_0)+\lambda \delta^p \implies \delta\leq (1-\lambda)^{-1}\xi(t_0).$$
As per (\ref{eq:x1x2}), let $C \in \R_{>0}$ be such that
$$\xi(t_0) \leq C^{-1}d_\mathbb{G}(x_1, x_2).$$
Then we have 
\begin{eqnarray*}
   &&|\chi^j\circ\Phi^X(t_1,t_0,x_1,p_1)-\chi^j\circ\Phi^X(t_2,t_0,x_2,p_2)| \\
   &&\hspace{20pt}\leq|\chi^j\circ\Phi^X(t_1,t_0,x_1,p_1)-\chi^j\circ\Phi^X(t_1,t_0,x_2,p_2)|\\
   &&\hspace{28pt}+|\chi^j\circ\Phi^X(t_2,t_0,x_2,p_2)-\chi^j\circ\Phi^X(t_1,t_0,x_2,p_2)|\\
   &&\hspace{20pt}\leq \xi(t_0)+\lambda\delta+\rho+\int_{|t_1,t_2|}|X\chi^j(s,\Phi^X(s,t_0,x_2,p_2),p_2)|\d s\\
   &&\hspace{20pt}\leq \frac{C^{-1}}{1-\lambda}d_\mathbb{G}(x_1, x_2)+\rho +\int_{|t_1,t_2|}|X\chi^j(s,\Phi^X(s,t_0,x_2,p_2),p_1)|\d s\\
   &&\hspace{28pt}+\int_{|t_1,t_2|}|X\chi^j(s,\Phi^X(s,t_0,x_2,p_2),p_1)-X\chi^j(s,\Phi^X(s,t_0,x_2,p_2),p_2)|\d s
\end{eqnarray*}
Now we choose a neighbourhood $\mathcal{V}$ of $x_1$, $\sigma\in\R_{>0}$, and and a neighbourhood $\mathcal{O'}\subseteq\mathcal{O}$ of $p_1$ such that 
$$\frac{C^{-1}}{1-\lambda}d_\mathbb{G}(x_1, x_2)<\frac{\epsilon}{4},\hspace{10pt} x_2\in\mathcal{V};$$
$$\int_{|t_1,t_2|}|X\chi^j(s,\Phi^X(s,t_0,x_2,p_2),p_1)|\d s<\frac{\epsilon}{4},\hspace{10pt}|t_1-t_2|<\sigma;$$
$$\int_{|t_1,t_2|}|X\chi^j(s,\Phi^X(s,t_0,x_2,p_2),p_1)-X\chi^j(s,\Phi^X(s,t_0,x_2,p_2),p_2)|\d s<\frac{\epsilon}{4},\hspace{10pt} p_2\in\mathcal{O'};$$
$$\int_{\mathbb{T}_{t_0,\alpha}}|X\chi^j(s,\Phi^X(s,t_0,x_2,p_2),p_1)-X\chi^j(s,\Phi^X(s,t_0,x_2,p_2),p_2)|\d s<\frac{\epsilon}{4},\hspace{10pt} p_2\in\mathcal{O'}.$$
The last inequality implies that $\rho<\frac{\epsilon}{4}$. Thus we have
$$|\chi^j\circ\Phi^X(t_1,t_0,x_1,p_1)-\chi^j\circ\Phi^X(t_2,t_0,x_2,p_2)|<\epsilon,\hspace{10pt} x_2\in\mathcal{V},\ p_2\in\mathcal{O'},\ |t_1-t_2|<\sigma,$$
which gives the continuity of $(t,x,p)\mapsto \chi^j\circ\Phi^X(t,t_0,x,p)$, and so the continuity of $(t,x,p)\mapsto\Phi^X(t,t_0,x,p)$. 
\end{proof}
The next lemma is a refinement of the preceding one, giving the local version of the theorem statement.
\begin{lem}\label{lem:3.7}
Let $M$ be a $C^{\infty}$-manifold, let $\mathbb{T}\subseteq \mathbb{R}$ be a time domain, let $\mathcal{P}$ be a topological space, and let $X\in \Gamma^{\text{lip}}_{\text{LI}}(\mathbb{T},TM;\mathcal{P})$. For each $(t_0,x_0,p_0)\in \mathbb{T}\times M\times\mathcal{P}$, there exist $\alpha\in \mathbb{R}_{>0}$, a neighborhood $\mathcal{U}\subseteq M$ of $x_0$, and a neighborhood $\mathcal{O}\subseteq\mathcal{P}$ of $p_0$ such that
\begin{enumerate}[(i)]
    \item\label{lem:3.7-1} $\mathbb{T}_{t_0,\alpha}\times \{t_0\} \times\mathcal{U}\times\mathcal{O}\subseteq D_{X}$,
    \item\label{lem:3.7-2} the map 
	$$\mathbb{T}_{t_0,\alpha}\times \mathcal{U}\times\mathcal{O}\ni (t,x,p)\mapsto \Phi^X(t,t_0,x,p)$$
	is continuous, and
    \item\label{lem:3.7-3} the map
	$$\mathcal{U}\ni x\mapsto \Phi^{X^p}_{t,t_0}(x)\in M$$
	is s bi-Lipschitz homeomorphism onto its image for every $p\in \mathcal{O}$ and every $t\in \mathbb{T}_{t_0,\alpha}$.
\end{enumerate}
\end{lem}
\begin{proof}
(\ref{lem:3.7-1}) and (\ref{lem:3.7-2}) can be proven by Lemma \ref{lem:3.6}.

(\ref{lem:3.7-3}) Let $\alpha'$, $\mathcal{U'}$ and $\mathcal{O'}$ be as in Lemma \ref{lem:3.6}, and let $\alpha\in(0,\alpha']$, $\mathcal{U}\subseteq\mathcal{U'}$, and $\mathcal{O}\subseteq\mathcal{O'}$ be such that 
$$\Phi^X(t,t_0,x,p)\in \mathcal{U'},\;\;\;\;(t,x,p)\in \mathbb{T}_{t_0,\beta}\times\mathcal{U}\times\mathcal{O},$$
this is possible by Lemma \ref{lem:3.6}-(\ref{lem:3.7-1}). Let $t\in \mathbb{T}_{t_0,\alpha}$, $x\in\mathcal{U}$, $p\in \mathcal{O}$, and denote 
$$\mathcal{V}=\Phi^{X^p}_{t,t_0}(\mathcal{U})\subseteq\mathcal{U'}.$$
Since $y\overset{\Delta}{=}\Phi^X(t,t_0,x,p)\in \mathcal{U'}$ and $t\in \mathbb{T}_{t_0,\alpha}$, there exists a neighborhood $\mathcal{V'}$ of $y$ such that if $y'\in \mathcal{V'}$, then $(t,t_0,y',p)\in D_X$. Moreover, since $\Phi^{X^p}_{t_0,t}$ is continuous and Lipschitz, we can choose $\mathcal{V'}$ sufficiently small that $\Phi^{X^p}_{t_0,t}(y')\in \mathcal{U}$ if $y'\in \mathcal{V'}$. By Lemma \ref{lem:3.6}, $\Phi^{X^p}_{t_0,t}|\mathcal{V'}$ is Lipschitz. Therefore there is a neighborhood of $x$ on which the restriction of $\Phi^{X^p}_{t_0,t}$ is invertible, Lipschitz, and with a Lipschitz inverse.
\end{proof}
We now need to show that parts (\ref{thm:3.4-5}) and (\ref{thm:3.4-6}) of the theorem hold globally. Let $(t_0,x_0,p_0)\in \mathbb{T}\times M\times\mathcal{P}$, and denote by $J_+(t_0, x_0, p_0) \subseteq \mathbb{T}$ the set of $b > t_0$ such that, for each $b' \in [t_0, b)$, there exists a relatively open interval $J\subseteq \mathbb{T}$, a neighborhood $\mathcal{U}$ of $x_0$, and a neighborhood $\mathcal{O}$ of $p_0$ such that 
\begin{enumerate}
    \item $b'\in J$,
    \item $J\times\{t_0\}\times\mathcal{U}\times\mathcal{O}\subseteq D_X$,
	\item $J\times\mathcal{U}\times\mathcal{O}\ni (t,x,p)\mapsto\Phi^X(t,t_0,x,p)\in M$ is continuous, and 
	\item the map $\mathcal{U}\ni x\mapsto\Phi^X(t,t_0,x,p)$ is locally bi-Lipschitz homeomorphism onto its image for every $p\in \mathcal{O}$ and every $t\in J$.
\end{enumerate}
\begin{proof}
By Lemma \ref{lem:3.7}, $J_+(t_0,x_0,p_0)\neq\emptyset$. Then we consider the two cases.

The first case is $J_+(t_0,x_0,p_0)\cap [t_0,\infty)=\mathbb{T}\cap [t_0,\infty)$. In this case, for each $t\in \mathbb{T}$ with $t\geq t_0$, there exists a relatively open interval $J\subseteq\mathbb{T}$, a neighborhood $\mathcal{U}$ of $x_0$, and a neighborhood $\mathcal{O}$ of $p_0$ such that 
\begin{enumerate}
    \item $t\in J$,
    \item $J\times\{t_0\}\times\mathcal{U}\times\mathcal{O}\subseteq D_X$,
	\item $J\times\mathcal{U}\times\mathcal{O}\ni (\uptau,x,p)\mapsto\Phi^X(\uptau,t_0,x,p)\in M$ is continuous, and 
	\item the map $\mathcal{U}\ni x\mapsto\Phi^X(\uptau,t_0,x,p)$ is locally bi-Lipschitz homeomorphism onto its image for every $p\in \mathcal{O}$ and every $\uptau\in J$.
\end{enumerate}

The second case is  $J_+(t_0,x_0,p_0)\cap [t_0,\infty)\subsetneq\mathbb{T}\cap [t_0,\infty)$. Now we show this is impossible. In this case we let $t_1=\sup J_+(t_0,x_0,p_0)$ and note that $t_1\neq \sup \mathbb{T}$. We claim that $t_1\in J_Y(t_0,x_0,p_0)$. If this were not the case, then we must have $b\overset{\Delta}{=}\sup J_{X}(t_0,x_0,p_0)<t_1$. Since $b\in J_+(t_0,x_0,p_0)$, there must be a relatively open interval $J\subseteq \mathbb{T}$ containing $b$ such that $t\in J_{X}(t_0,x_0,p_0)$ for all $t\in J$. But, since there are $t's$ in $J$ larger than $b$, this contradicts the definition of $b$, and so we conclude that $t_1\in J_X(t_0,x_0,p_0)$.

We denote $x_1=\Phi^X(t_1,t_0,x_0,p_0)$. By Lemma \ref{lem:3.7}, there exists $\alpha_1\in\R_{>0}$, a neighborhood $\mathcal{V}_1$ of $x_1$ such that $(t,t_1,x,p)\in D_X$ for every $t\in \mathbb{T}_{t_1,\alpha_1}$, $x\in \mathcal{V}_1$, and $p\in\mathcal{O}_1$, and such that the map $$\mathbb{T}_{t_1,\alpha_1}\times \mathcal{V}_1\times\mathcal{O}_1\ni (t,x,p)\mapsto\Phi^X(t,t_1,x,p)$$
is continuous, and the map
$$\mathcal{V}_1\ni x\mapsto\Phi^X(t,t_1,x,p)$$
is locally bi-Lipschitz homeomorphism onto its image for every $t\in \mathbb{T}_{t_1,\alpha_1}$ and $p\in\mathcal{O}_1$. Let $\mathcal{V'}_1\subseteq \mathcal{V}_1$ be such that $\text{cl}(\mathcal{V'})\subseteq\mathcal{V}_1$. Since $t\mapsto\Phi^X(t,t_0,x_0,p_0)$ is continuous and $\Phi^X(t_1,t_0,x_0,p_0)=x_1$, let $\delta\in\R_{>0}$ be such that $\delta<\frac{\alpha}{2}$ and $\Phi^X(t,t_0,x_0,p_0)\in\mathcal{V'}_1$ for $t\in (t_1-\delta,t_1)$. Now let $\uptau_1\in (t_1-\delta,t_1)$ and, by our hypotheses on $t_1$, there exists an open interval $J$, a neighborhood $\mathcal{U'}_1$ of $x_0$, and a neighborhood $\mathcal{O'}_1$ of $p_0$ such that 
\begin{enumerate}
    \item $\uptau_1\in J$,
    \item $J\times\{t_0\}\times\mathcal{U'}_1\times\mathcal{O'}_1\subseteq D_X$,
	\item $J\times\mathcal{U'}_1\times\mathcal{O'}_1\ni (\uptau,x,p)\mapsto\Phi^X(\uptau,t_0,x,p)\in M$ is continuous, and 
	\item the map $\mathcal{U'}_1\ni x\mapsto\Phi^X(\uptau,t_0,x,p)$ is locally bi-Lipschitz homeomorphism onto its image for every $\uptau\in J$ and $p\in\mathcal{O'}_1$.
\end{enumerate}
We choose $J$, $\mathcal{U'}_1$ and $\mathcal{O'}_1$ sufficiently small that 
$$\{\Phi^X(t,t_0,x,p)\ \suchthat \ t\in J,\ x\in\mathcal{U'}_1,\ p\in\mathcal{O'}_1\}\subseteq\mathcal{V'}_1$$
This is possible by the continuity of the flow. We, moreover, assume that $\mathcal{O'}_1$ is sufficiently small as to be a subset of $\mathcal{O}_1$.

We claim that
$$ \mathcal{T}_{\uptau_1,\alpha_1}\times\{t_0\}\times\mathcal{U'}_1\times\mathcal{O'}_1\subseteq D_X.$$

We first show that 
\begin{equation}\label{eq:1.8}
   [\uptau_1,\uptau_1+\alpha_1)\times\{t_0\}\times\mathcal{U'}_1\times\mathcal{O'}_1\subseteq D_X.
\end{equation}
Let $x\in \mathcal{U'}_1$ and $p'\in\mathcal{O'}_1$, we have $(\uptau_1,t_0,x,p)\in D_X$ since $\uptau_1\in J$. By definition of $J$, $\Phi^X(\uptau_1,t_0,x,p)\in \mathcal{V'}_1$. By definition of $\uptau_1, t_1-\uptau_1<\delta<\frac{\alpha_1}{2}$. Then by definition of $\alpha_1$ and $\mathcal{V}_1$,
$$(t_1,\uptau_1,\Phi^X(\uptau_1,t_0,x,p'),p)\in D_X$$
for every $x\in \mathcal{U'}_1$, $p'\in\mathcal{O'}_1$, and $p\in\mathcal{O}_1$. From this we conclude that $(t_1,t_0,x,p)\in D_X$ for every $x\in \mathcal{U'}_1$ and $p\in\mathcal{O'}_1$. Now since 
$$t\in[\uptau_1,\uptau_1+\alpha_1)\implies t\in \mathbb{T}_{t_1,\alpha_1}$$
we have $(t,t_1,\Phi^X(t_1,t_0,x,p),p)\in D_X$ for every $t\in \mathbb{T}_{\uptau_1,\alpha_1}$, $x\in\mathcal{U'}_1$, and $p\in\mathcal{O'}_1$. Since 
$$(t,t_1,\Phi^Y(t_1,t_0,x,p),p)=\Phi^X(t,t_0,x,p),$$
we conclude (\ref{eq:1.8}). A similar argument gives the conclusion above.

Now we claim the map 
$$ \mathcal{T}_{\uptau_1,\alpha_1}\times\mathcal{U'}_1\times\mathcal{O'}_1\ni (t,x,p)\mapsto \Phi^X(t,t_0,x,p)$$
is continuous. This map is continuous at 
$$(t,x,p)\in (\uptau_1-\alpha_1,\uptau_1]\times\mathcal{U'}_1\times\mathcal{O'}_1$$
by definition of $\uptau_1$. For $t\in (\uptau_1,\uptau_1+\alpha_1)$ we have 
$$\Phi^X(t,t_0,x,p)=\Phi^X(t,\uptau_1,\Phi^Y(\uptau_1,t_0,x,p),p)$$
and continuity follows since the composition of continuous maps are continuous. 

Next we claim the map 
$$\mathcal{U'}_1\ni x\mapsto \Phi^Y(t,t_0,x,p)$$
is locally bi-Lipschitz homeomorshism onto its image for every $t\in \mathcal{T}_{\uptau_1,\alpha_1}$ and $p\in\mathcal{O'}_1$. By definition of $\uptau_1$, the map 
$$\Phi^{X^p}_{t,t_0}:\mathcal{U'}_1\mapsto\mathcal{V'}_1$$
is locally bi-Lipschitz onto its image for $t\in (\uptau_1-\alpha_1,\uptau_1]$ and $p\in\mathcal{O'}_1$. We also have that 
$$\Phi^{Y}_{t,\uptau_1}:\mathcal{V}_1\mapsto M$$
is locally bi-Lipschitz homeomorphism onto its image for $t\in (\uptau_1,\uptau_1+\alpha_1)$. Since the composition of locally bi-Lipschitz homeomorphisms onto their image is a locally bi-Lipschitz homeomorphism onto its image, our assertion follows.

By our arguments, we have an open interval $J'$, a neighborhood $\mathcal{U'}$ of $x_0$, and a neighbourhood $\mathcal{O'}_1$ of $p_0$ such that
\begin{enumerate}
    \item $t_1\in J'$,
    \item $J'\times\{t_0\}\times\mathcal{U'}_1\times\mathcal{O'}_1\subseteq D_X$,
	\item $J'\times\mathcal{U'}_1\times\mathcal{O'}_1\ni (t,x,p)\mapsto\Phi^X(t,t_0,x,p)\in M$ is continuous, and
	\item the map $\mathcal{U'}_1\ni x\mapsto\Phi^X(t,t_0,x,p)$ is locally bi-Lipschitz homeomorphism onto its image for every $t\in J'$ and $p\in\mathcal{O'}_1$.
\end{enumerate}
This contradicts the fact that $t_1=\sup J_+(t_0,x_0,p_0)$ and so the condition 
$$J_+(t_0,x_0,p_0)\cap [t_0,\infty)\subsetneq \mathbb{T}\cap [t_0,\infty)$$
cannot obtain. 

One similar shows that it must be the case that $J_-(t_0,x_0,p_0)\cap (-\infty,t_0]=\mathbb{T}\cap (-\infty,t_0]$, where $J_-(t_0,x_0)$ has the obvious definition. 
\end{proof}
Finally, we note that $\Phi^X_{t,t_0}$ is injective by uniqueness of solutions for $X$. Now, assertions (\ref{thm:3.4-5}) and (\ref{thm:3.4-6}) of the theorem follow, since the notions of “locally bi-Lipschitz homeomorphism” can be tested locally, i.e., in a neighbourhood of any point. We have shown something more, in fact, namely that, along with parts (\ref{thm:3.4-5}) and (\ref{thm:3.4-6}) of the theorem, the set
$$\{(t,x,p)\in\mathbb{T}\times M\times\mathcal{P}\ \suchthat\ (t,t_0,x,p)\in D_X\}$$
is open for each $t_0$, and that the mapping $(t, x, p) \mapsto \Phi^X(t, t_0, x, p)$ is continuous. We shall now use this fact and a lemma, to prove assertions (\ref{thm:3.4-7}) and (\ref{thm:3.4-8}) together.
\begin{lem}
Let $(t_1,t_0,x_0,p_0)\in D_X$. As Lemma \ref{lem:3.7}, there exist $\alpha_1\in\R_{>0}$, a neighborhood $\mathcal{U}_1$ of $x_0$, and a neighborhood $\mathcal{O}_1$ of $p_0$ such that 
$$\mathbb{T}_{t_1,\alpha_1}\times \{t_0\}\times\mathcal{U}_1\times\mathcal{O}_1\subseteq D_X,$$
and the map $(t,x,p)\mapsto \Phi^X(t,x_0,x,p)$ is continuous on this domain. Then the map 
$$(t,x,p)\mapsto \Phi^X(t_0,t,x,p)$$
is continuous for $(t,x,p)$ nearby $(t_0,x_0,p_0)$.
\end{lem}
\begin{proof}
To see this, we proceed rather as in the proof of Lemma \ref{lem:3.6}. Let $U$, $\chi^1,...,\chi^n$, and $C$ be just as in the initial part of the proof of Lemma \ref{lem:3.6}. We also choose $\alpha, r \in R_{>0}$ and a neighbourhood $\mathcal{O}$ of $p_0$ such that $\mathcal{U}(r, x_0) \in \mathcal{U} \cap \mathcal{U}_1$ and such that
\begin{equation}\label{eq:1.9}
\int_{|t,t_0|}|X\chi^j(s,x,p)|\d s<\frac{r}{2}
\end{equation}
and
\begin{equation}\label{eq:1.15}
\int_{|t,t_0|}|X\chi^j(s,x_1,p)-X\chi^j(s,x_2,p)|\d s<\frac{\lambda}{C}d_G(x_1,x_2),
\end{equation}
for all $t\in \mathbb{T}_{t_0,\alpha}$, $x,x_1,x_2\in \mathcal{U}$, $j\in\{1,...,n\}$ and $p\in \mathcal{O}$.

Let $x\in \mathcal{U}(r/2,x_0)$. For each $j\in \{1,...,n\}$, we denote by $\phi^j_0\in C^0(\mathbb{T}_{t_0,\alpha},\mathbb{R})$ the constant mapping $\phi ^j_0=\chi^j(x_0)$, and denote by $\overbar{B}(r,\phi^j_0)$ the closed ball of radius $r$ about $\phi^j_0$. For each $p\in \mathcal{O}$, we have the mapping
\begin{eqnarray*}
   F^j_{X^p}:&&\overbar{B}(r,\phi^j_0)\rightarrow C^0(\mathbb{T}_{t_0,\alpha},\mathbb{R})\\
	         &&\phi^j\mapsto \chi^j(x)+\int_{|t,t_0|}X\chi^j(s,\gamma_{\phi}(s),p)ds,
\end{eqnarray*}
for $j\in \{1,...,n\}$, and where $\gamma_\phi \in C^0(\mathbb{T}_{t_0,\alpha},M)$ is the unique curve satisfying 
\begin{equation*}
\chi^j\circ \gamma_\phi=\phi^j(t),\;\;\;\;\;t\in \mathbb{T}_{t_0,\alpha},\;j\in\{1,...,n\},
\end{equation*}
cf. Lemma \ref{lem:3.3}.  Note that similar to the proof in Lemma \ref{lem:3.6}, the definition of $r$ ensures that $\gamma_\phi$ so defined takes values in $\mathcal{U}$.

We then claim that $F^j_{X^p}(\overbar{B}(r,\phi^j_0))\subseteq\overbar{B}(r,\phi^j_0)$, $j\in\{1,...,n\}$, $p\in\mathcal{O}$. Indeed, if $\phi^j\in\overbar{B}(r,\phi^j_0),\ j\in\{1,...,n\}$, we have 
\begin{equation*}
    |F^j_{X^p}\circ\phi^j(t)-\phi_0^j(t)|\leq|\chi^j(x)-\chi^j(x_0)|+\int_{|t,t_0|}|X\chi^j(s,\gamma_{\phi}(s))|\d s< r
\end{equation*}

We also claim that the mapping 
\begin{equation}\label{eq:1.16}
    \prod\limits_{j=1}^{n} \overbar{B}(r,\phi^j_0)\ni(\phi^1,...,\phi^n)\mapsto (F^1_{X^p}\circ\phi^1,...,F^n_{X^p}\circ\phi^n)\in \prod\limits_{j=1}^{n} \overbar{B}(r,\phi^j_0)
\end{equation}
is a contraction mapping for each $p\in\mathcal{O}$, where $\prod\limits_{j=1}^{n}\overbar{B}(r,\phi^j_0)$
is given the product metric. Indeed, let $\phi^j_1,\phi^j_2\in \overbar{B}(r,\phi^j_0),\ j\in\{1,...,n\}$. Let $\gamma_1,\gamma_2\in C^0(\mathbb{T}_{t_0,\alpha},\mathbb{R})$ be the corresponding curves satisfying 
$$\chi^j\circ \gamma_a(t)=\phi_a^j(t),\;\;\;\;\;t\in \mathbb{T}_{t_0,\alpha},\;j\in\{1,...,n\},\ a\in \{1,2\},$$
using Lemma \ref{lem:3.3}. Then we have, for each $j\in\{1,...,n\}$, 
\begin{eqnarray*}
   |F^j_{X^p}\circ\phi^j_1(t)-F^j_{X^p}\circ\phi^j_2(t)|
   &\leq&\int_{|t_0,t|}|X\chi^j(s,\gamma_1(s),p)-X\chi^j(s,\gamma_2(s),p)|\d s\\
   &\leq&\frac{\lambda}{C}d_G(\gamma_1(s),\gamma_2(s))\\
   &\leq&\lambda \sup\{|\phi^k_1(s)-\phi^k_2(s)|\ \suchthat\  k\in\{1,...,n\}\},\;s\in\mathbb{T}_{t_0,\alpha} 
\end{eqnarray*}
from which the desired conclusion follows.

By the Contraction Mapping Theorem, there exists a unique fixed point for the mapping (\ref{eq:1.16}), which gives rise to a curve $\xi:\mathbb{T}_{t_0,\alpha}\mapsto \mathcal{U}$ satisfying
$$f\circ\xi(t)=f(x)+\int_{|t,t_0|}Xf(s,\xi(s),p)\d s.$$
Differentiating with respect to $t$ shows that $\xi$ is an integral curve for $X$ and $\xi(t_0)=x$. This shows that, if $t\in \mathbb{T}_{t_0,\alpha}$, $x\in \mathcal{U}(r/2,x_0)$, and $p\in\mathcal{O}$, then we have $\Phi^X(t_0,t,x,p)\in \mathcal{U}(r,x_0)$ and
$$f\circ\Phi^X(t_0,t,x,p)=f(x)+\int_{|t,t_0|}Xf(s,\Phi^Y(t_0,s,x,p),p)\d s$$
for $f\in C^\infty(M)$. We conclude that there exists $\alpha_0, r_0\in \R_{>0}$ and a neighbourhood $\mathcal{O}$ of $p_0$ such that 
$$\Phi^X(t_0,t,x,p)\in \mathcal{U}_1,\;\;\;\;\;\;(t,x,p)\in\mathbb{T}_{t_0,\alpha_0}\times\mathcal{U}(r_0,x_0)\times\mathcal{O}.$$
We then show that the map $(t,x,p)\mapsto\Phi^X(t_0,t,x,p)$ is continuous on $\mathbb{T}_{t_0,\alpha_0}\times\mathcal{U}(r_0,x_0)\times\mathcal{O}$, just as in the proof of Lemma \ref{lem:3.6}. 

Finally, if 
$$(t',t'_0,x,p)\in \mathbb{T}_{t,\alpha}\times\mathbb{T}_{t_0,\alpha_0}\times\mathcal{U}(r_0,x_0)\times\mathcal{O'},$$
then
$$\Phi^X(t',t'_0,\Phi^X(t_0,t'_0,x,p),p)=\Phi^X(t',t'_0,x,p),$$
which shows both that $D_X$ is open and that $\Phi^X$ is continuous, since that composition of continuous mappings is continuous.
\end{proof}
(\ref{thm:3.4-9}) Let $T_+ = \sup J_X(t_0, x_0, p_0)$. Then $(T_+ -\epsilon, t_0, x_0, p_0) \in D_X$. Since $D_X$ is open, there
exists a neighbourhood $\mathcal{U}$ of $x_0$ and a neighbourhood $\mathcal{O}$ of $p_0$ such that
$$\{T_+ -\frac{\epsilon}{2}\} \times \mathbb{T}_{t_0,\alpha} \times \mathcal{U} \times \mathcal{O} \subseteq D_X.$$
In other words, $[t_0, T_+ -\frac{\epsilon}{2}] \subseteq J_X(t, x, p)$ for every $(t, x, p) \in \mathbb{T}_{t_0,\alpha} \times \mathcal{U} \times \mathcal{O}$. Thus, for such $(t, x, p)$,
$$\sup J_X(t, x, p) \geq T_+ - \frac{\epsilon}{2} > T_+ - \epsilon = \sup J_X(t_0, x_0, p_0) - \epsilon,$$
as claimed. A similar argument holds for the left endpoint of intervals of existence.

(\ref{thm:3.4-10}) Let $t \in |t_0, t_1|$ and let $\epsilon \in \R_{>0}$. Following the argument and using notation inspired by the proof of part (\ref{lem:3.6-2}) of Lemma \ref{lem:3.6}, there is an interval $\mathbb{T}_t \subseteq \mathbb{T}$, a neighbourhood $\mathcal{V}_{t,\epsilon}$ of $x_0$, a neighbourhood $\mathcal{O}_{t,\epsilon} \subseteq \mathcal{P}$ of $p_0$, and $\chi^j_t \in C^\infty(M)$, $j \in \{1, . . . , n\}$, (n being the dimension
of $M$) such that
\begin{eqnarray*}
   &&|\chi^j_t\circ\Phi^X(t',t,\Phi^X(t,t_0,x,p),p)-\chi^j_t\circ\Phi^X(t',t,\Phi^X(t,t_0,x_0,p_0),p_0)|\leq C^{-1}_t\epsilon\\
   &&\hspace{8cm}(t',x,p)\in \mathbb{T}_t\times \mathcal{V}_{t,\epsilon}\times \mathcal{O}_{t,\epsilon},\ \ j\in\{1,...,n\},
\end{eqnarray*}
where $C_t\in\R_{>0}$ is such that 
$$d_\mathbb{G}(x_1, x_2) \leq C_t \max \left\{|\chi^j_t(x_1) - \chi^j_t(x_2)|\ \suchthat\  j \in\{1, . . . , n\},\ x_1, x_2 \in \mathcal{V}_{t,\epsilon}\right\}.$$
Let $t_1, . . . , t_k \in |t_0, t_1|$ be such that $|t_0, t_1| \subseteq \cup_{j=1}^{k}\mathbb{T}_{t_j}$. For $t \in |t_0, t_1|$, let $j_t \in \{1, . . . , k\}$ be such that $t \in \mathbb{T}_{t_{j_t}}$, and let $x \in \cap^k_{j=1}\mathcal{V}_{t_j,\epsilon}$ and $p \in \cap^k_{j=1}\mathcal{O}_{t_j,\epsilon}$. Let $C = \max\{C_{t_1}
, . . . , C_{t_k}\}$. Then, if $t \in |t_0, t_1|$ and with $j \in \{1, . . . , k\}$ such that $t \in \mathbb{T}_{t_j}$,
$$d_\mathbb{G}(\Phi^X(t, t_0, x, p), \Phi^X(t, t_0, x_0, p_0)) \leq C|\chi^j_t\circ \Phi^X(t, t_0, x, p) - \chi^j_t\circ\Phi^X(t, t_0, x_0, p_0)| < \epsilon,$$
which gives the desired uniform convergence.

(\ref{thm:3.4-9})\footnote{The author acknowledges Robert Kipka for furnishing a version of this proof.} Since the parameter-dependence plays no role here, we suppose we are in the parameter-independent case to simplify the notation. By definition of $I_X(t_1, x_0)$, the curve $\iota_{(t_1,x_0)}$ is well-defined. We will show that it is locally absolutely continuous. We first prove this locally, and so work with a time-varying vector field $X$. Let $(t_0, x_0) \in \mathbb{T} \times M$ and let $\alpha \in \R_{>0}$ and $\mathcal{U}$ be a neighbourhood of $x_0$ such that
$$f\circ\iota_{(t_1,x_0)}(t)= f\circ\Phi^X(t_1,t,x_0)=f(x_0)+\int_{|t,t_1|}Xf(s,\Phi^Y(s,t,x_0))\d s,$$
for $t \in \mathbb{T}_{t_0,\alpha}$ and $x \in \mathcal{U}$. Following our constructions in the preceding parts of the proof, we work with functions $\chi^1, . . . , \chi^n \in C^\infty(M)$ whose differentials are linearly independent on $\mathcal{U}$. We also let $g, l \in L^1_{\text{loc}}(\mathbb{T}; \R_{\geq 0})$ be such that
$$|X\chi^j(t, x)| \leq g(t)$$
and
$$|X\chi^j(t, x_1) - X\chi^j(t, x_2)| \leq l(t) \max\{|\chi^m(x_1) - \chi^m(x_2)|\ \suchthat\  m \in \{1, . . . , n\}\}$$
for $t \in \mathbb{T}$, $x, x_1, x_2 \in \mathcal{U}$, and $j \in \{1, . . . , n\}$, this because $X\chi^j \in C^{\text{lip}}_{\text{LI}}(\mathbb{T}; M)$ by Lemma \ref{lem:2.4}, (\ref{eq:x1x2}), and \citep[Theorem 6.4: (v) and (vi)]{Jafarpour2014}. 

Let $\delta\in \R_{>0}$ be such that, if $(a_l, b_l)$, $l \in \{1, . . . , k\}$, is a collection of pairwise disjoint subintervals of $\mathbb{T}_{t_0,\alpha}$ satisfying
$$\sum_{l=1}^{k}|b_l-a_l|<\delta,$$
then
$$\sum_{l=1}^{k}\int_{a_1}^{b_l}g(s)\d s<\epsilon\exp \left(-\int_{\mathbb{T}_{t_0,\alpha}}l(s)\d s\right).$$
Let $t_1 \in \mathbb{T}_{t_0,\alpha}$ and $j \in \{1, . . . , n\}$, and compute
\begingroup
\allowdisplaybreaks
\begin{eqnarray*}
   &&\sum_{l=1}^{k}|\chi^j\circ \iota_{(t_1,x_0)}(b_l)-\chi^j\circ \iota_{(t_1,x_0)}(a_l)|\\
   &&\hspace{1cm}=\sum_{l=1}^{k}|\chi^j\circ \Phi^X(t_1,b_l,x_0)-\chi^j\circ \Phi^X(t_1,a_l,x_0)|\\
   &&\hspace{1cm}\leq \sum_{l=1}^{k}\left|\int_{b_l}^{t_1}X\chi^j(s,\Phi^X(s,b_l,x_0))\d s-\int_{a_l}^{t_1}X\chi^j(s,\Phi^X(s,a_l,x_0))\d s\right|\\
   &&\hspace{1cm}\leq \sum_{l=1}^{k}\left|\int_{b_l}^{t_1}X\chi^j(s,\Phi^X(s,b_l,x_0))\d s-\int_{a_l}^{t_1}X\chi^j(s,\Phi^X(s,b_l,x_0))\d s\right|\\
   &&\hspace{1.4cm}+\sum_{l=1}^{k}\left|\int_{a_l}^{t_1}X\chi^j(s,\Phi^X(s,b_l,x_0))\d s-\int_{a_l}^{t_1}X\chi^j(s,\Phi^X(s,a_l,x_0))\d s\right|\\
   &&\hspace{1cm}\leq \sum_{l=1}^{k}\int_{b_l}^{b_l}|X\chi^j(s,\Phi^X(s,b_l,x_0))|\d s\\
   &&\hspace{1.4cm}+\sum_{l=1}^{k}\int_{a_l}^{t_1}|X\chi^j(s,\Phi^X(s,b_l,x_0))-X\chi^j(s,\Phi^X(s,a_l,x_0))|\d s\\
   &&\hspace{1cm}\leq\sum_{l=1}^{k}\int_{b_l}^{b_l} g(s)\d s +\sum_{l=1}^{k}\int_{a_l}^{t_1} l(s)\max\{|\chi^m\circ\Phi^X(s,b_l,x_0)\\
   &&\hspace{1.4cm}-\chi^m\circ\Phi^X(s,a_l,x_0)|\ \suchthat\ m\in\{1,...,n\}\}\d s\\
   &&\hspace{1cm}\leq\sum_{l=1}^{k}\int_{b_l}^{b_l} g(s)\d s +\sum_{l=1}^{k}\int_{a_l}^{t_1} l(s)\max\{|\chi^m\circ\Phi^X(s,b_l,x_0)\\
   &&\hspace{1.4cm}-\chi^m\circ\Phi^X(s,a_l,x_0)|\ \suchthat\ m\in\{1,...,n\}\}\d s\\
   &&\hspace{1.4cm}+\sum_{l=1}^{k}\int_{t_0-\alpha}^{a_l} l(s)\max\{|\chi^m\circ\Phi^X(s,b_l,x_0)-\chi^m\circ\Phi^X(s,a_l,x_0)|\ \suchthat\ m\in\{1,...,n\}\}\d s\\
   &&\hspace{1cm}\leq\sum_{l=1}^{k}\int_{b_l}^{b_l} g(s)\d s +\int_{t_0-\alpha}^{t_1} l(s)\sum_{l=1}^{k}\max\{|\chi^m\circ\Phi^X(s,b_l,x_0)\\
   &&\hspace{1.4cm}-\chi^m\circ\Phi^X(s,a_l,x_0)|\ \suchthat\ m\in\{1,...,n\}\}\d s.
\end{eqnarray*}
\endgroup
For $t \in \mathbb{T}_{t_0,\alpha}$ and let us denote
$$\kappa(t) = \sum_{l=1}^{k}\max{|\chi^j\circ \iota_{(t,x_0)}(b_l)- \chi^j\circ \iota_{(t,x_0)}(a_l)|\ \suchthat\  j \in \{1, . . . , m\}}.$$
Our computations above show that
$$\kappa(t)\leq\sum_{l=1}^{k}\int_{b_l}^{b_l} g(s)\d s+ \int_{t_0-\alpha}^{t_1} l(s)\kappa(s)\d s.$$
Thus, by Gronwall’s inequality \citep[Lemma C.3.1]{MR1640001}, we have
$$\kappa(t_1) \leq \exp \left(\int_{t_0-\alpha}^{t_1}l(s) \d s \right)\sum_{l=1}^{k}\int_{b_l}^{b_l} g(s)\d s.$$
This shows that, for every $t_1 \in \mathbb{T}_{t_0,\alpha}$,
$$\sum_{l=1}^{k}|\chi^j\circ \iota_{(t,x_0)}(b_l)- \chi^j\circ \iota_{(t,x_0)}(a_l)|<\epsilon.$$
Thus, in particular, $\iota(t_1,x_0)$ is absolutely continuous on $\mathbb{T}_{t_0,\alpha}$ for each $t_1 \in \mathbb{T}_{t_0,\alpha}$.
Now we prove the suitable conclusion globally. Here we make use of another lemma.
\begin{lem}
Let $\mathcal{U} \subseteq \R^n$ and $\mathcal{U}\subseteq \R^m$ be open, let $\mathbb{T} \subseteq \R$ be a time-domain, let $\Phi: \mathcal{U} \rightarrow \mathcal{V}$ be locally Lipschitz, and let $\gamma : \mathbb{T} \rightarrow \mathcal{U}$ have locally absolutely continuous components. Then $\Phi\circ \gamma$ has locally absolutely continuous components.
\end{lem}
\begin{proof}
Let $[a, b] \subseteq \mathbb{T}$ be a compact subinterval. Since $\gamma([a, b])$ is compact and $\Phi$ is locally Lipschitz, there exists $L \in \R_{>0}$ such that
$$||\Phi \circ \gamma(t_1) - \Phi \circ \gamma(t_2)|| \leq L||\gamma(t_1) - \gamma(t2)||,\hspace{10pt} t_1, t_2 \in [a, b].$$
Let $\epsilon \in \R_{>0}$ and let $\delta \in \R_{>0}$ be such that, if $((a_j , b_j ))_{j\in\{1,...,k\}}$ is a family of disjoint intervals such that
$$\sum_{j=1}^{k}|b_j-a_j|<\delta,$$
then
$$\sum_{j=1}^{k}|\gamma^a(b_j)-\gamma^a(a_j)|<\frac{\epsilon}{L\sqrt{n}}, \hspace{10pt}a\in\{1,...,n\}.$$
Then, for any $a \in \{1, . . . , n\}$ and $\alpha \in \{1, . . . , m\}$, and using standard relationships between the 2-norm and the 1-norm for $\R^n$,
\begin{eqnarray*}
   \sum_{j=1}^{k}|\Phi^\alpha\circ \gamma(b_j)-\Phi^\alpha\circ \gamma(a_j)|&\leq& \sum_{j=1}^{k}||\Phi\circ \gamma(b_j)-\Phi\circ \gamma(a_j)||\leq \sum_{j=1}^{k} L||\gamma(b_j)-\gamma(a_j)||\\
   &\leq& L\sqrt{n}\max\left\{\sum_{j=1}^{k}|\gamma^a(b_j)-\gamma^a(a_j)|\ \suchthat\ j\in\{1,...,n\} \right\}<\epsilon
\end{eqnarray*}
whenever 
$$\sum_{j=1}^{k}|b_j-a_j|<\delta.$$
\end{proof}
Now let $(t_1, t_0, x_0) \in D_X$ and let $\alpha \in \R_{>0}$ and $\mathcal{U}$ be as above. For $t \in \mathbb{T}_{t_0,\alpha}$ we have
$$f \circ \iota_{(t_1,x_0)}(t) = f \circ \Phi^X(t_1, t, x_0) = f \circ \Phi^X_{t_1,t_0}\circ \Phi^X_{t_0,t}(x_0).$$
By our computations above, the curve $t\mapsto\Phi^X_{t_0,t}(x_0)$ is absolutely continuous on $\mathbb{t_0,\alpha}$. Since $x\mapsto\Phi^X_{t_1,t_0}(x)$ is locally Lipschitz by part (\ref{thm:3.4-6}), it follows from the previous lemma that $t \mapsto f \circ \iota_{(t_1,x_0)}(t)$ is absolutely continuous on $\mathbb{T}_{t_0,\alpha}$. Since local absolute continuity is a local property, i.e., it only needs to hold in any neighbourhood of any point, it follows that $t \mapsto f \circ \iota_{(t_1,x_0)}(t)$ is locally absolutely continuous. Thus, by definition, $\iota_{(t_1,x_0)}$ is locally absolutely continuous. 
\renewcommand\qedsymbol{$\blacksquare$}
\end{proof}
The reader will note that a substantial portion of the proof is taken up with the extension of the usual local statements—such as one normally finds in presentations of ordinary differential equations—to global statements valid on the whole of the domain of the vector
field. We feel as if it worth doing this carefully once, since it is not easy to find, and impossible to find in the generality we give here. A few useful consequences of the theorem follow.
\begin{cor}[Images of compact subsets of initial conditions are compact]
Let $M$ be a $C\infty$-manifold, let $\mathbb{T} \subset \R$ be an interval, and let $X \in\Gamma^{\text{lip}}_{\text{LI}}(\mathbb{T};TM)$. Let $K \subseteq M$ be compact and let $t_0, t_1 \in \mathbb{T}$ be such that
\begin{equation}\label{eq:3.14}
    |t_0, t_1| \times \{t_0\} \times K \subseteq D_X. 
\end{equation}
Then 
\begin{equation}\label{eq:3.15}
    \bigcup_{(t,x)\in|t_0,t_1|\times K} \Phi^X(t,t_0,x)
\end{equation}
is compact.
\end{cor}
\begin{proof}
This follows since the set (\ref{eq:3.15}) is the image of the compact set (\ref{eq:3.14}) under the continuous (by part (\ref{thm:3.4-8}) of the theorem) map $\Phi^X$.
\end{proof}
\begin{cor}[Robustness of compactness by variations of parameters]\label{cor:3.11}
Let $M$ be a $C^\infty$-manifold, let $\mathbb{T} \subseteq \R$ be an interval, let $\mathcal{P}$ be a topological space, and let $X \in\Gamma^{\text{lip}}_{\text{PLI}}(\mathbb{T};TM; \mathcal{P})$. Let $K \subseteq M$ be compact, let $t_0, t_1 \in \mathbb{T}$, and let $p_0 \in \mathcal{P}$ be such that
\begin{equation*}
    |t_0, t_1| \times \{t_0\} \times K\times\{p_0\} \subseteq D_X. 
\end{equation*}
Then there exists a a neighbourhood $\mathcal{O} \subseteq \mathcal{P}$ of $p_0$ such that
\begin{equation*}
    \bigcup_{(t,x,p)\in|t_0,t_1|\times K\times\mathcal{O}} \Phi^X(t,t_0,x,p)
\end{equation*}
is well-defined and precompact.
\end{cor}
\begin{proof}
By the previous corollary,
$$K_0\overset{\Delta}{=}\bigcup_{(t,x)\in|t_0,t_1|\times K} \Phi^X(t,t_0,x)$$
is compact. Since $M$ is locally compact, let $\mathcal{V}$ be a precompact neighbourhood of $K_0$. By part (\ref{thm:3.4-10}) of the theorem, for $x \in K$, let $\mathcal{U}_x$ be a neighbourhood of $x$ and let $\mathcal{O}_x$ be a neighbourhood of $p_0$ such that
$$\bigcup_{(t,x,p)\in|t_0,t_1|\times (K\cap\mathcal{U}_x)\times\mathcal{O}_x} \Phi^X(t,t_0,x,p)\subseteq\mathcal{V}.$$
By compactness of $K$, let $x_1, . . . , x_m \in K$ be such that
$K=\cup_{j=1}^{m}K\cap \mathcal{U}_{x_j}$ and let $\mathcal{O}=\cap_{j=1}^{k}\mathcal{O}_{x_j}$. Then 
$$\Phi^X(t,t_0,x,p)\subseteq\mathcal{V},\hspace{10pt}(t,x,p)\in|t_0,t_1|\times K\times\mathcal{O},$$
as desired.
\end{proof}
\begin{cor}[Uniform Lipschitz character of flows]\label{cor:3.12}
Let $M$ be a $C^\infty$-manifold, let $\mathbb{T} \subseteq \R$ be an interval, let $\mathcal{P}$ be a topological space, and let $X \in\Gamma^{\text{lip}}_{\text{PLI}}(\mathbb{T};TM; \mathcal{P})$. Let $K \subseteq M$ be compact, let $t_0, t_1 \in \mathbb{T}$, and let $p_0 \in \mathcal{P}$ be such that
\begin{equation*}
    |t_0, t_1| \times \{t_0\} \times K\times\{p_0\} \subseteq D_X. 
\end{equation*}
Then there exists a a neighbourhood $\mathcal{O} \subseteq \mathcal{P}$ of $p_0$ and $C\in\R_{>0}$ such that
$$d_\mathbb{G}(\Phi^X(t, t_0, x_1, p), \Phi^X(t, t_0, x_2, p)) \leq Cd_\mathbb{G}(x_1, x_2),\hspace{10pt} t \in |t_0, t_1|,\  x_1, x_2 \in K, \ p \in \mathcal{O}.$$
\end{cor}
\begin{proof}
As in the proof of Lemma \ref{lem:3.6}(\ref{lem:3.6-1}) from the theorem (see, especially, the last few lines of that part of the proof), for $(t, x) \in |t_0, t_1| \times K$, there exists an open interval $\mathbb{T}_{(t,x)} \subseteq |t_0, t_1|$ containing $t$, a neighbourhood $\mathcal{U}_{(t,x)} \subseteq M$ of $x$, and a neighbourhood $\mathcal{O}_{(t,x)}$ of $p_0$ such that
$$\mathbb{T}_{(t,x)} \times \{t_0\} \times\mathcal{U}_{(t,x)}\times\mathcal{O}_{(t,x)} \subseteq D_X,$$
and such that there exists $C_{(t,x)} \in \R_{>0}$ for which
\begin{eqnarray}\label{eq:3.16}
   &&d_\mathbb{G}(\Phi^X(t', t_0, x_1, p), \Phi^X(t', t_0, x_2, p)) \leq C_{(t,x)}d_\mathbb{G}(\Phi^X(t, t_0, x_1, p), \Phi^X(t, t_0, x_2, p)),\nonumber\\
   &&\hspace{8cm}t \in |t_0, t_1|,\  x_1, x_2 \in K, \ p \in \mathcal{O}.
\end{eqnarray}
Denote 
$$\mathcal{N}_{(t,x)} = \{\Phi^X(t', t_0, x', p)\  \suchthat (t', x', p) \in \mathbb{T}_{(t,x)} \times \mathcal{U}_{(t,x)} \times \mathcal{O}_{(t,x)}\},$$
noting that $\mathcal{N}_{(t,x)}$ is a filter neighbourhood of $\Phi^X(t, t_0, x, p_0)$. By compactness of
$$K_x \overset{\Delta}{=}\{\Phi^X(t, t_0, x, p_0)\  \suchthat\  t \in |t_0, t_1|\},$$
there exist $t_{x,1}, . . . , t_{x,m_x} \in |t_0, t_1|$ such that
$$K_x\subseteq\bigcup_{j=1}^{m_x}\text{int}(\mathcal{N}_{(t_{x,j},x)}).$$
Let us also choose $t_{x,0} = t_0$ and so add to this finite cover the set $\mathcal{N}_{(t_{x,0},x)}$ associated with $t = t_0$. Let $\mathcal{O}_x=\cap_{j=0}^{m_x}\mathcal{O}_{(t_{x,j},x)}$ and $\mathcal{U}_x=\cap_{j=0}^{m_x}\mathcal{U}_{(t_{x,j},x)}$, and note that 
$$|t_0, t_1| \times \{t_0\} \times K\cap \mathcal{U}_x\times\mathcal{O}_x \subseteq D_X. $$
Also note that
$$\mathcal{V}_{x} \overset{\Delta}{=} \{\Phi^X(t, t_0, x', p)\  \suchthat (t, x', p) \in |t_0,t_1| \times \mathcal{U}_{x} \times \mathcal{O}_{x}\}$$
is a filter neighbourhood of $K_x$. Since the intervals $\mathbb{T}_{(t_{x,j} ,x)}$, $j \in \{0, . . . , m_x\}$, cover $|t_0, t_1|$
and since $t_{x,0} = t_0$, for any $j \in \{1, . . . , m\}$ we can write
\begin{equation}\label{eq:3.17}
    \Phi^X(t_{x,j} , t_0, x, p) = \Phi^X _{t_{x,j_m},t_{x,j_{m-1}}}\circ \Phi^X_{t_{x,j_1},t_0}
\end{equation}
for some $j_1, . . . , j_m \in \{1, . . . , m_x\}$ satisfying $t_{x,j_l} \in \mathbb{T}_{x,j_{l-1}}$, $l \in \{1, . . . , m\}$. Moreover, we can do this with at most $m_x$ compositions.

Now let $x_1, x_2 \in \mathcal{U}_x$ and note that, for $t \in |t_0, t_1|$, $t \in\mathbb{T}_{(t_{x,j} ,x)}$ for some $j \in \{0, 1, . . . , m_x\}$.
We also have $x_1, x_2 \in \mathcal{U}_{(t_{x,j} ,x)}$. Therefore, with this $j$ chosen and for $p \in \mathcal{O}_x$,
$$d_\mathbb{G}(\Phi^X(t, t_0, x_1, p), \Phi^X(t, t_0, x_2, p)) \leq C_{t_j,x}d_\mathbb{G}(\Phi^X(t_{x,j}, t_0, x_1, p), \Phi^X(t_{x,j}, t_0, x_2, p)).$$
Using the composition (\ref{eq:3.17}) and the bound (\ref{eq:3.16}) for each term in the composition, we have
$$d_\mathbb{G}(\Phi^X(t, t_0, x_1, p), \Phi^X(t, t_0, x_2, p)) \leq C_xd_\mathbb{G}(x_1, x_2),\hspace{10pt} t \in |t_0, t_1|,$$
taking 
$$C_x = \max\{C_{(t_{x,1},x)}, . . . , C_{(t_{x,m_x},x)}\}^{m_x+1}$$
Choose $x_1, . . . , x_m \in K$ so that $K = \cup^m_{j=1}(K \cap \mathcal{U}_{x_j})$. Let $\mathcal{O} = \cap^m_{j=1}\mathcal{O}_{x_j}$. If necessary and by Corollary \ref{cor:3.11}, shrink $\mathcal{O}$ so that
$$\{\Phi^X(t, t_0, x, p)\ \suchthat\  t \in |t_0, t_1|, x \in K, p \in \mathcal{O}\}$$
is precompact. Let $M \in \R_{>0}$ and $x_0 \in K$ be such that
$$d_\mathbb{G}(x, x_0) \leq M,\hspace{10pt} x \in K.$$
Note that
$$\{\Phi^X(t, t_0, x, p)\ \suchthat\  t \in |t_0, t_1|, x \in K, p \in \mathcal{O}\}\subseteq\bigcup_{j=1}^{m}\mathcal{V}_{x_j}.$$
By the Lebesgue Number Lemma \citep[Theorem 1.6.11]{MR1835418}, let $r \in \R_{>0}$ be such that, if $x_1, x_2 \in K$ satisfy $d_\mathbb{G}(x_1, x_2) < r$, then there exists $j \in \{1, . . . , m\}$ such that $x_1, x_2 \in \mathcal{U}_{x_j}$. Let
$$C = \max \left\{ C_{x_1}, . . . , C_{x_m},\frac{2M}{r}\right\}.$$
Finally, let $t \in |t_0, t_1|$, let $x_1, x_2 \in K$, and let $p \in \mathcal{O}$. If $d_\mathbb{G}(x_1, x_2) < r$, let $j \in \{1, . . . , m\}$ be such that $x_1, x_2 \in\mathcal{U}_{x_j}$. Then we have
$$d_\mathbb{G}(\Phi^X(t, t_0, x_1, p), \Phi^X(t, t_0, x_2, p)) \leq C_{x_j}d_\mathbb{G}(x_1, x_2)\leq Cd_\mathbb{G}(x_1, x_2).$$
If $d_\mathbb{G}(x_1, x_2) \geq r$, then
\begin{eqnarray*}
   d_\mathbb{G}(\Phi^X(t, t_0, x_1, p), \Phi^X(t, t_0, x_2, p))&\leq& d_\mathbb{G}(\Phi^X(t, t_0, x_1, p),x_0)+d_\mathbb{G}(\Phi^X(t, t_0, x_2, p),x_0)\\
   &\leq&\frac{2M}{r}r\leq Cd_\mathbb{G}(x_1, x_2),
\end{eqnarray*}
as desired.
\end{proof}
\subsection{Continuous dependence of fixed-time flow on parameter}
\begin{thm}
Let $m\in \mathbb{Z}_{\geq 0}$, let $m'\in\{0,\rm{lip}\}$, let $\nu\in\{m+m',\infty,\omega,\rm{hol}\}$ satisfy $\nu\geq \rm{lip}$, and let $r\in \{\infty,\omega,\rm{hol}\}$as appropriate. Let $M$ be a $C^r$-manifold, let $\mathbb{T}\subseteq\R$ be an interval, let $\mathcal{P}$ be a topological space. For $X\in\Gamma^\nu_{\rm{PLI}}(\mathbb{T};TM;\mathcal{P})$, let $t_0,t_1\in\mathbb{T}$ and $p_0\in\mathcal{P}$ be such that there exists a precompact open set $\mathcal{U}\subseteq M$ such that $\rm{cl}(\mathcal{U})\subseteq \rm{D}_X(t_1,t_0,p_0)$. Then there exists a neighbourhood $\mathcal{O}\subseteq\mathcal{P}$ of $p_0$ such that mapping
\begin{eqnarray*}
  \mathcal{O}\ni p \mapsto \Phi^{X^p}_{t_1,t_0}\in C^\nu(\mathcal{U};M)
\end{eqnarray*}
is well-defined and continuous.
\end{thm}
We break down the proof into the various classes of regularity. The proofs bear a strong resemblance to one another, so we go through the details carefully in the first case we prove, the locally Lipschitz case, and then merely outline where the arguments differ for the other regularity classes.
\begin{proof}
\subsubsection{The \texorpdfstring{$C^\text{lip}$}{}-case}\label{sec:3.3.1}
Note that $\rm{cl}(\mathcal{U})$ is compact. Therefore, by Corollary \ref{cor:3.11}, there exists a compact set $K' \subseteq M$ and a neighbourhood $\mathcal{O}$ of $p_0$ such that
$$\Phi^{X^p}_{t_1,t_0}(x)\in \text{int}(K'),\hspace{10pt}(x,p)\in\text{cl}(\mathcal{U})\times \mathcal{O}.$$
This gives the well-definedness assertion of the theorem. Note, also, that it gives the welldefinedness assertion for all $\nu \in \{m + m', \infty, \omega, \text{hol}\}$, and so we need not revisit this for the remainder of the proof. The compact set $K' \subseteq M$ and the neighbourhood $\mathcal{O}$ of $p_0$ will be used in all parts of the proof without necessarily referring to our constructions here.

For continuity, first we show that the mapping
\begin{eqnarray*}
  \mathcal{O}\ni p \mapsto \Phi^{X^p}_{t_1,t_0}\in C^0(\mathcal{U};M)
\end{eqnarray*}
is continuous. The topology for $C^0(\mathcal{U};M)$ is the uniform topology defined by the semimetrics 
$$d^0_{K,f}(\Phi_1,\Phi_2)=\sup\{|f\circ\Phi_1(x)-f\circ\Phi_2(x)|\ | \ x\in K\},\hspace{10pt} f\in C^\infty(M), \ K\subseteq \mathcal{U}\ \text{compact}.$$
Thus, we must show that, for $f_1,...,f_m\in C^\infty(M)$, for $K_1,...,K_m\subseteq\mathcal{U}$ compact, and for $\epsilon_1,...,\epsilon_m\in\R_{>0}$, there exists a neighborhood $\mathcal{O}$ of $p_0$ such that 
$$|f_j\circ\Phi^{X^p}_{t_1,t_0}(x_j)-f_j\circ \Phi^{X^{p_0}}_{t_1,t_0}(x_j)|<\epsilon_j,\hspace{10pt} x_j\in K_j, \ p\in\mathcal{O}, \ j\in\{1,...,m\}.$$
It will suffice to show that, for $f\in C^\infty(M)$, for $K\subseteq\mathcal{U}$ compact, and for $\epsilon\in\R_{>0}$, we have  
$$|f\circ\Phi^{X^p}_{t_1,t_0}(x)-f\circ \Phi^{X^{p_0}}_{t_1,t_0}(x)|<\epsilon,\hspace{10pt} x\in K, \ p\in\mathcal{O}.$$
Indeed, if we show that then, taking $K=\cup_{j=1}^{k}K_j$ and $\epsilon
=\min\{\epsilon_1,...,\epsilon_m\}$, we have 
$$|f_j\circ\Phi^{X^p}_{t_1,t_0}(x)-f_j\circ \Phi^{X^{p_0}}_{t_1,t_0}(x)|<\epsilon,\hspace{10pt} x\in K, \ p\in\mathcal{O}, \ j\in\{1,...,m\}$$
for a suitable $\mathcal{O}$. This suffices to give the desired conclusion.

It is useful to consider the space $C^0(\mathbb{T};M)$ with the topology (indeed, uniformity) defined by the family of semimetrics
$$d_{\mathbb{S},M}(\gamma_1,\gamma_2)=\{d_{\mathbb{G}}(\gamma_1(t),\gamma_2(t))\ |\ t\in \mathbb{S}\},\hspace{10pt}\mathbb{S}\subseteq\mathbb{T} \text{ a compact interval.}$$
For $g\in C^0_{\text{LI}}(|t_0,t_1|;M)$, we also have the mapping
\begin{eqnarray*}
  \Psi_{|t_0,t_1|,M,g}:C^0(|t_0,t_1|;M) &\rightarrow & L^1_{\rm{loc}}(|t_0,t_1|;\R)\\
       \gamma&\mapsto& (t\mapsto g_t(\gamma(t))).
\end{eqnarray*}
The following lemma gives the well-definedness and continuity of this mapping.
\renewcommand\qedsymbol{$\nabla$}
\begin{lem}\label{lem:3.14}
If $\mathbb{T}\subseteq\R$ is an interval and if $g\in C^0_{\text{LI}}(\mathbb{T};M)$, then $\Psi_{\mathbb{T},M,g}$ is well-defined and continuous. 
\end{lem}
\begin{proof}
We first show that $t\mapsto g(t,\gamma(t))$ is measurable on $\mathbb{T}$. Note that 
$$t\mapsto g(t,\gamma(s))$$
is measurable for each $s \in \mathbb{T}$ and that
\begin{equation}\label{eq:3.18}
    s\mapsto g(t,\gamma(s))
\end{equation}
is continuous for each $t \in \mathbb{T}$. Let $[a, b] \subseteq \mathbb{T}$ be compact, let $k \in \mathbb{Z}_{>0}$, and denote
$$t_{k,j} = a + \frac{j-1}{k}(b - a), \hspace{10pt} j \in \{1, . . . , k + 1\}.$$
Also denote 
$$\mathbb{T}_{k,j} = [t_{k,j} , t_{k,j+1}), \hspace{10pt} j \in \{1, . . . , k - 1\},$$
and $\mathbb{T}_{k,k} = [t_{k,k}, t_{k,k+1}]$. Then define $g_k : \mathbb{T} \rightarrow \R$ by
$$g_k(t)=\sum_{j=1}^{k}g(t,\gamma(t_{k,j}))\chi_{t_{k,j}}.$$
Note that $g_k$ is measurable, being a sum of products of measurable functions \citep[Proposition 2.1.7]{MR3098996}. By continuity of (\ref{eq:3.18}) for each $t \in \mathbb{T}$, we have
$$\lim_{k\to\infty}g_k(t) = g(t,\gamma(t)),\hspace{10pt} t \in [a, b],$$
showing that $t\mapsto g(t,\gamma(t))$ is measurable on $[a,b]$, as pointwise limits of measurable functions are measurable \citep[Proposition 2.1.5]{MR3098996}. Since the compact interval $[a, b] \subseteq \mathbb{T}$ is arbitrary, we conclude that $t \mapsto  g(t,\gamma(t))$ is measurable on $\mathbb{T}$.

Let $\mathbb{S} \subseteq \mathbb{T}$ be compact and let $K \subseteq M$ be a compact set for which $\gamma(\mathbb{S}) \subseteq K$. Since $g\in C^0_{\text{LI}}(\mathbb{T};M)$ is continuous, there exists $h\in L^1(\mathbb{S};\R_{\geq 0})$ be such that 
$$|g(t,x)|\leq h(t) \hspace{20pt} (t,x)\in\mathbb{S}\times K,$$
In particular, this shows that $t \mapsto g(t,\gamma(t))$ is integrable on $\mathbb{S}$ and so locally integrable on $\mathbb{T}$. This gives the well-definedness of $\Psi_{\mathbb{T},M,g}$.

For continuity, let $\gamma_j\in C^0(\mathbb{S};M),\ j\in\mathbb{Z}_{>0}$, be a sequence of curves converging uniformly to $\gamma\in C^0(\mathbb{S};M)$. Let $\mathbb{S}\subseteq\mathbb{T}$ be a compact interval and let $K\subseteq M$ be compact. Since $\text{image}(\gamma)\cup K$ is compact and $M$ is locally compact, we can find a precompact neighbourhood $\mathcal{U}$ of $\text{image}(\gamma)\cup K$. Then for $N\in \mathbb{Z}_{>0}$ 0 sufficiently large, we have $\text{image}(\gamma_j)\subseteq\mathcal{U}$ for all $j\geq N$ by uniform convergence. Therefore, we can find a compact set $K'\subseteq M$ such that $\text{image}(\gamma_j)\subseteq K'$ for all $j\geq N$ and $\text{image}(\gamma)\subseteq K'$. Then for fixed $t\in\mathbb{S}$, continuity of $x\mapsto g(t,x)$ ensures that $\lim_{j\to\infty}g(t,\gamma_j(t))=g(t,\gamma(t))$. We also have
$$|g(t,\gamma_j(t))|\leq h(t) \hspace{20pt} t\in\mathbb{S}.$$
Therefore, by the Dominated Convergence Theorem
$$\lim_{j\to\infty}\int_{\mathbb{S}}g(t,\gamma_j(t))\ dt=\int_{\mathbb{S}}g(t,\gamma(t))\ dt,$$
which gives the desired continuity.
\end{proof}
Now consider the following mapping 
\begin{eqnarray*}
  \Phi_{\mathbb{T}, K,\mathcal{O}}:K\times\mathcal{O} &\rightarrow & C^0(\mathbb{T};M)\\
       (x,p)&\mapsto& (t\mapsto\Phi^X(t,t_0,x,p)).
\end{eqnarray*}
By Theorem \ref{thm:3.4}(\ref{thm:3.4-10}), this mapping is continuous. We also have the continuous mapping
\begin{eqnarray*}
  \iota_{|t_0,t_1|}:C^0(\mathbb{T};M) &\rightarrow& C^0(|t_0,t_1|;M)\\
       \gamma&\mapsto& \gamma||t_0,t_1|.
\end{eqnarray*}
Let $f\in C^\infty(M)$, let $\epsilon>0$, and let $x\in K$. Combining the observations of two paragraphs previous with the lemma, the mapping
$$\Psi_{|t_0,t_1|,M,X^{p_0}f}\circ \iota_{|t_0,t_1|}\circ \Phi_{\mathbb{T}, K,\mathcal{O}}: K\times \mathcal{O}\rightarrow L^1(|t_0,t_1|;\R)$$
is continuous. Thus there exists a relative neighbourhood $\mathcal{V}_x\subseteq K$ of $x$ and a neighbourhood $\mathcal{O}_x\subseteq\mathcal{O}$ of $p_0$ such that
$$\int_{|t_0,t_1|}|X^{p_0}f(s,\Phi^{X^p}_{s,t_0}(x'))-X^{p_0}f(s,\Phi^{X^{p_0}}_{s,t_0}(x'))|\ ds<\frac{\epsilon}{2} \hspace{15pt} x'\in\mathcal{V}_x, \ p\in\mathcal{O}_x.$$
Let $x_1,...,x_m\in K$ be such that $K= \cup_{j=1}^{m}\mathcal{V}_{x_j}$ and define a neighbourhood $\mathcal{O}_1=\cap_{j=1}^{k}\mathcal{O}_{x_j}$ of $p_0$. Then we have
\begin{equation}\label{eq:5.19}
    \int_{|t_0,t_1|}|X^{p_0}f(s,\Phi^{X^p}_{s,t_0}(x))-X^{p_0}f(s,\Phi^{X^{p_0}}_{s,t_0}(x))|\ ds<\frac{\epsilon}{2},\hspace{15pt}x\in K,\ p\in\mathcal{O}_1 .
\end{equation}
By (\ref{eq:2.4}), we can further shrink $\mathcal{O}_1$ if necessary so that
$$\int_{|t_0,t_1|}|X^{p}f(s,x)-X^{p_0}f(s,x)|\ ds<\frac{\epsilon}{2} \hspace{15pt} x'\in K, \ p\in\mathcal{O}_1.$$
Then we have 
\begin{eqnarray*}
   &&|f\circ\Phi^{X^p}_{t_1,t_0}(x)-f\circ \Phi^{X^{p_0}}_{t_1,t_0}(x)|\\
   &&\hspace{2cm}\leq\int_{|t_0,t_1|}|X^{p}f(s,\Phi^{X^p}_{s,t_0}(x))-X^{p_0}f(s,\Phi^{X^{p_0}}_{s,t_0}(x))|\ ds\\
   &&\hspace{2cm}\leq \int_{|t_0,t_1|}|X^{p}f(s,\Phi^{X^p}_{s,t_0}(x))-X^{p_0}f(s,\Phi^{X^{p}}_{s,t_0}(x))|\ ds\\
   &&\hspace{2.4cm}+\int_{|t_0,t_1|}|X^{p_0}f(s,\Phi^{X^p}_{s,t_0}(x))-X^{p_0}f(s,\Phi^{X^{p_0}}_{s,t_0}(x))|\ ds\\
   &&\hspace{2cm}\leq \frac{\epsilon}{2}+\frac{\epsilon}{2}=\epsilon,
\end{eqnarray*}
for $x\in K$ and $p\in\mathcal{O}_1$, as desired. 

Now, to show the mapping 
\begin{eqnarray*}
  \mathcal{O}\ni p \mapsto \Phi^{X^p}_{t_1,t_0}\in C^{\rm{lip}}(\mathcal{U};M)
\end{eqnarray*}
is continuous, we must include our analysis the dilatation semimetrics
$$\lambda^0_{K,f}(\Phi_1,\Phi_2)=\sup\{\text{dil}(f\circ\Phi_1-f\circ\Phi_2)(x)\ | \ x\in K\},\hspace{10pt} f\in C^\infty(M), \ K\subseteq \mathcal{U}\ \text{compact},$$
for $C^{\rm{lip}}(\mathcal{U};M)$. As we argued above, it will suffice to show that, for $f\in C^\infty(M)$, for $K\subseteq\mathcal{U}$ compact, and for $\epsilon\in\R_{>0}$, there exists a neighborhood $\mathcal{O}$ of $p_0$ such that 
$$\text{dil}(f\circ\Phi^{X^p}_{t_1,t_0}-f\circ \Phi^{X^{p_0}}_{t_1,t_0})(x)<\epsilon,\hspace{10pt} x\in K, \ p\in\mathcal{O}.$$
By Corollary \ref{cor:3.12}, we let $C\in\R_{>0}$ and let $\mathcal{O'}\subseteq\mathcal{O}$ be a neighbourhood of $p_0$ such that
$$d_{\mathbb{G}}(\Phi^X(t,t_0,x_1,p),\Phi^X(t,t_0,x_2,p))\leq Cd_{\mathbb{G}}(x_1,x_2), \hspace{10pt} t\in|t_0,t_1|,\ x_1,x_2\in K,\ p\in\mathcal{O'}.$$
Let $f\in C^\infty(M)$, let $K\subseteq\mathcal{U}$ be compact, and let $\epsilon>0$. First we estimate 
$$\int_{|t_0,t_1|}\text{dil}(X_{s}^{p_0}f\circ \Phi^{X^{p}}_{s,t_0}-X_{s}^{p_0}f\circ \Phi^{X^{p_0}}_{s,t_0})(x)ds.$$
We shall use a strategy similar to the above for the $p^0_K$ seminorm, and borrow the notation from the computations there. For $g\in C^{\text{lip}}_{\text{LI}}(|t_0,t_1|;M)$, we also have the mapping
\begin{eqnarray*}
  \Psi_{|t_0,t_1|,M,g}:C^0(|t_0,t_1|;M) &\rightarrow & L^1_{\rm{loc}}(|t_0,t_1|;\R)\\
       \gamma&\mapsto& (t\mapsto \text{dil}\ g_t(\gamma(t))),
\end{eqnarray*} 
We may show that this mapping is well-defined in a manner similar to Lemma \ref{lem:3.14}, merely changing absolute value for dilatation. In like manner, since $g\in C^{\text{lip}}_{\text{LI}}(|t_0,t_1|;M)$, by virtue of Lemma \ref{lem:lsdic}, we conclude that $\text{dil}\ g\in C^{0}_{\text{LI}}(|t_0,t_1|;M)$. From Lemma \ref{lem:3.14}, $\Psi_{|t_0,t_1|,M,g}$ is continuous.

Now consider the following mapping 
\begin{eqnarray*}
  \Phi_{\mathbb{T}, K,\mathcal{O}}:K\times\mathcal{O} &\rightarrow & C^0(\mathbb{T};M)\\
       (x,p)&\mapsto& (t\mapsto\Phi^X(t,t_0,x,p)).
\end{eqnarray*}
By Theorem \ref{thm:3.4} (\ref{thm:3.4-10}), this mapping is continuous. We also have the continuous mapping
\begin{eqnarray*}
  \iota_{|t_0,t_1|}:C^0(\mathbb{T};M) &\rightarrow& C^0(|t_0,t_1|;M)\\
       \gamma&\mapsto& \gamma||t_0,t_1|.
\end{eqnarray*}
Let $f\in C^\infty(M)$, let $\epsilon>0$, and let $x\in K$. Combining the observations of two paragraphs previous, the mapping
$$\Psi_{|t_0,t_1|,M,X^{p_0}f}\circ \iota_{|t_0,t_1|}\circ \Phi_{\mathbb{T}, K,\mathcal{O}}: K\times \mathcal{O}\rightarrow L^1(|t_0,t_1|;\R)$$
is continuous. Thus there exists a relative neighbourhood $\mathcal{V}_x\subseteq K$ of $x$ and a neighbourhood $\mathcal{O}_x\subseteq\mathcal{O}$ of $p_0$ such that
$$\int_{|t_0,t_1|}|\text{dil}(X^{p_0}_sf)(\Phi^{X^p}_{s,t_0}(x'))-\text{dil}(X^{p_0}_sf)(\Phi^{X^{p_0}}_{s,t_0}(x'))|\ ds<\frac{\epsilon}{2} \hspace{15pt} x'\in\mathcal{V}_x, \ p\in\mathcal{O}_x.$$
Let $x_1,...,x_m\in K$ be such that $K= \cup_{j=1}^{m}\mathcal{V}_{x_j}$ and define a neighbourhood $\mathcal{O}_2=\cap_{j=1}^{k}\mathcal{O}_{x_j}$ of $p_0$. Then, for $x\in K$ and $p\in\mathcal{O}_2$, we have
\begin{equation}\label{eq:5.20}
   \int_{|t_0,t_1|}|\text{dil}(X^{p_0}_sf)(\Phi^{X^p}_{s,t_0}(x'))-\text{dil}(X^{p_0}_sf)(\Phi^{X^{p_0}}_{s,t_0}(x'))|\ ds<\frac{\epsilon}{2}. 
\end{equation}
Next we estimate
$$\int_{|t_0,t_1|}\text{dil}((X^p_sf-X^{p_0}_sf)\circ\Phi^{X^p}_{s,t_0})(x)ds.$$

For $(s,x)\in|t_0,t_1|\times K$, let $\mathcal{V}_{(s,x)}\subseteq K$ be precompact neighbourhood of $\Phi^X(s,t_0,x,p_0)$ as in Lemma \ref{lem:lsdasd}. Abbreviate $\mathcal{U}_{(s,x)}=(\Phi^{p_0}_{s,t_0})^{-1}(\mathcal{V}_{(s,x)})$, a precompact neighbourhood of $x$. By definition of local sectional dilatation,
$$\text{dil}((X^p_sf-X^{p_0}_sf)\circ\Phi^{X^p}_{s,t_0})(x')\leq l_{\text{cl}(\mathcal{U}_{(s,x)})}((X^p_sf-X^{p_0}_sf)\circ\Phi^{X^p}_{s,t_0}),\ \ p\in\mathcal{O'},\ x'\in K\cap\mathcal{U}_{(s,x)}.$$
Since the composition of locally Lipschitz maps is locally Lipschitz,  \citep [Proposition 1.2.2]{MR1832645}, we have
$$\text{dil}((X^p_sf-X^{p_0}_sf)\circ\Phi^{X^p}_{s,t_0})(x)\leq C l_{\text{cl}(\mathcal{U}_{(s,x)})}(X^p_sf-X^{p_0}_sf),\ \ p\in\mathcal{O'}.$$
Since 
$$\{\Phi^x(s,t_0,x,p_0)\  |\ s\in|t_0,t_1|,\ x\in K\}$$
is compact, we can cover this set with finitely many sets $\mathcal{V}_{(s_1,x_1)},...,\mathcal{V}_{(s_m,x_m)}$. By  Lemma \ref{lem:lsdasd} and (\ref{eq:2.4}), we can shrink $\mathcal{O}_2$ so that
$$\int_{|t_0,t_1|}l_{\text{cl}(\mathcal{U}_{(s_j,x_j)})}(X^p_sf-X^{p_0}_sf) ds<\frac{\epsilon}{2C}, \ \ p\in\mathcal{O}_2,\ j\in\{1,...,m\}.$$
This gives 
$$\int_{|t_0,t_1|}\text{dil}((X^p_sf-X^{p_0}_sf)\circ\Phi^{X^p}_{s,t_0})(x)ds\leq \frac{\epsilon}{2}\ \ x\in K,\ p\in\mathcal{O}_2.$$
Combining the preceding calculations, we have
\begin{eqnarray*}
   \text{dil}(f\circ\Phi^{X^p}_{t_1,t_0}-f\circ \Phi^{X^{p_0}}_{t_1,t_0})(x)&\leq&\int_{|t_0,t_1|}\text{dil}(X^{p}f\circ\Phi^{X^p}_{s,t_0}-X^{p_0}f\circ\Phi^{X^{p_0}}_{s,t_0})(s,x)\ ds\\
   &\leq& \int_{|t_0,t_1|}\text{dil}((X^pf-X^{p_0}f)\circ\Phi^{X^p}_{s,t_0})(s,x)ds\\
   &&+\int_{|t_0,t_1|}\text{dil}(X^{p_0}f\circ\Phi^{X^p}_{s,t_0}-X^{p_0}f\circ\Phi^{X^{p_0}}_{s,t_0})(s,x)\ ds\\
   &<& \frac{\epsilon}{2}+\frac{\epsilon}{2}=\epsilon,
\end{eqnarray*}
for $x\in K$ and $p\in\mathcal{O}_2$.

Finally, combining the two parts of the proof for the two sorts of different seminorms for $C^{\rm{lip}}(\mathcal{U};M)$, we ascertain that, for every compact $K\subseteq\mathcal{U}$, every $f\in C^\infty(M)$, and every $\epsilon\in\R_{>0}$, if $p\in\mathcal{O}_1\cap\mathcal{O}_2$, then we ave 
$$p^0_K(f\circ\Phi^{X^p}_{t_1,t_0}-f\circ \Phi^{X^{p_0}}_{t_1,t_0})<\epsilon,\ \ \lambda^0_K(f\circ\Phi^{X^p}_{t_1,t_0}-f\circ \Phi^{X^{p_0}}_{t_1,t_0})<\epsilon,$$
which gives the desired result.

\subsubsection{The \texorpdfstring{$C^{m}$}{}-case}\label{sec:3.3.2}
The topology for $C^m(\mathcal{U};M)$ is the uniform topology defined by the semimetrics 
$$d^m_{K,f}(\Phi_1,\Phi_2)=\sup\{\|j_m(f\circ\Phi_1)(x)-j_m(f\circ\Phi_2)(x)\|_{\mathbb{G}_{M,m}}\ | \ x\in K\},\  f\in C^\infty(M), \ K\subseteq \mathcal{U}\ \text{compact}.$$
As in the preceding section when we proved $C^0$ continuity, it suffices to show that, for
$f\in C^\infty(M)$, $K\subseteq\mathcal{U}$ compact, and for $\epsilon\in\R_{>0}$, there exists a neighbourhood $\mathcal{O'}$ of $p_0$ such that 
$$\|j_m(f\circ\Phi^{X^p}_{t_1,t_0})(x)-j_m(f\circ\Phi^{X^{p_0}}_{t_1,t_0})(x)\|_{\mathbb{G}_{M,m}}<\epsilon,\hspace{10pt}x\in K,\ p\in\mathcal{O'}.$$

Thus let $f\in C^\infty(M)$, $K\subseteq\mathcal{U}$ compact, and let $\epsilon\in\R_{>0}$. 

Consider the mapping
\begin{eqnarray*}
  \Phi_{|t_0,t_1|, K,\mathcal{O}}:K\times\mathcal{O} &\rightarrow & C^0(|t_0,t_1|;J^m(\mathcal{U};M))\\
       (x,p)&\mapsto& (t\mapsto j_m\Phi^{X^p}_{t,t_0}(x)),
\end{eqnarray*}
which is well-defined and continuous. For $(x,p)\in K\times\mathcal{O}$ and for $t\in|t_0,t_1|$, we can think of $j_m\Phi^{X^p}_{t,t_0}(x))$ as a linear mapping 
\begin{eqnarray*}
  j_m\Phi^{X^p}_{t,t_0}(x):J^m(M;\R)_{\Phi^{X^p}_{t,t_0}(x)} &\rightarrow & J^m(M;\R)_x \\
       j_m g(\Phi^{X^p}_{t,t_0}(x))&\mapsto& j_m (g\circ\Phi^{X^p}_{t,t_0})(x).
\end{eqnarray*}
Now, fixing $(x,p)\in\mathcal{U}\times \mathcal{O}$ for the moment, recall the constructions of Section \ref{sec:2.3.2}, particularly those preceding the statement of Lemma \ref{lem:2.5}. We consider the notation from those constructions with
\begin{enumerate}
    \item $N=M$,
    
    \item $E=F=J^m(M;\R)$,
    
    \item $\Gamma(t)=j_m\Phi^{X^p}_{t,t_0}(x)\in \text{Hom}_{\R}(J^m(M,\R)_{\Phi^{X^p}_{t,t_0}(x)};J^m(M;\R)_x)$, and
    
    \item $\xi=j_m(X^{p_0}f)$.
\end{enumerate}
Thus, again in the notation from Section \ref{sec:2.3.2}, we have
$$\gamma_M(t)=\Phi^{X^p}_{t,t_0}(x),\hspace{10pt}\gamma_N(t)=x.$$
We then have the integrable section of $E=J^m(M;\R)$ given by 
\begin{eqnarray*}
  \xi_\Gamma:|t_0,t_1| &\rightarrow & E\\
       t&\mapsto& (t\mapsto j_m (X^p_tf\circ\Phi^{X^p}_{t,t_0})(x))
\end{eqnarray*}
to obtain continuity of the mapping
\begin{eqnarray*}
  \Psi_{|t_0,t_1|,J^m(M;\R),j_m (X^{p_0}f)}:C^0(|t_0,t_1|;J^m(\mathcal{U};M)) &\rightarrow & L^1_{\text{loc}}(|t_0,t_1|;J^m(M;\R))\\
       \Gamma&\mapsto& (t\mapsto \Gamma(t)(j_m(X^{p_0}_tf)(\gamma_M(t)))),
\end{eqnarray*}
and so of the composition
$$\Psi_{|t_0,t_1|,J^m(M;\R),j_m (X^{p_0}f)}\circ \Phi_{|t_0,t_1|, K,\mathcal{O}}: K\times\mathcal{O}\rightarrow L^1_{\text{loc}}(|t_0,t_1|;J^m(M;\R)).$$
Note that this is precisely the continuity of the mapping
$$K\times\mathcal{O}\ni (x,p)\mapsto (t\mapsto j_m(X^{p_0}f\circ \Phi^{X^p}_{t,t_0}(x)))\in L^1_{\text{loc}}(|t_0,t_1|;J^m(M;\R)).$$
In order to convert this continuity into a continuity statement involving the fibre norm for $J^m(M;\R)$, we note that for $x\in K$,  there exists a neighbourhood $\mathcal{V}_x$ and affine functions $F^1_x,..., F^{n+k}_x\in\text{Aff}^\infty(J^m(M;\R))$ which are coordinates for $\rho^{-1}_m(\mathcal{V}_x)$. We can choose a Riemannian metric for $J^m(M;\R)$, whose restriction to fibres agrees with the fibre metric (\ref{eq:fibremetric}) \citep[\S 4.1]{lewis2020geometric}. It follows, therefore, from Lemma \ref{lem:5.2} that there exists $C_x\in\R_{>0}$ such that
$$\|j_m g_1(x')-j_m g_2(x')\|_{\mathbb{G}_{M,m}}\leq C_x|F^l_x\circ j_m g_1(x')-F^l_x\circ j_m g_2(x')|,$$
for $g_1,g_2\in C^\infty(M)$, $x'\in\mathcal{V}_x$, $l\in\{1,...,n+k\}$. By the continuity proved in the preceding paragraph, we can take a relative neighbourhood $\mathcal{V}_x\subseteq K$ of $x$ sufficiently small and a neighbourhood $\mathcal{O}_x\subseteq\mathcal{O}$ of $p_0$ such that 
$$\int_{|t_0,t_1|}|F^l_x\circ j_m(X^{p_0}f\circ \Phi^{X^p}_{t,t_0}(x'))-F^l_x\circ j_m(X^{p_0}f\circ \Phi^{X^{p_0}}_{t,t_0}(x'))|dt<\frac{\epsilon}{2C_x},$$
for all $x'\in\mathcal{V}_x$, $p\in\mathcal{O}_x$, and $l\in\{1,...,n+k\}$, recalling from Section \ref{sec:2.3.2} the definition of the topology for $L^1(|t_0, t_1|; J^m(M; \R)$. Therefore,
$$\int_{|t_0,t_1|}\|j_m(X^{p_0}f\circ \Phi^{X^p}_{t,t_0}(x'))-j_m(X^{p_0}f\circ \Phi^{X^{p_0}}_{t,t_0}(x'))\|_{\mathbb{G}_{M,m}} ds <\frac{\epsilon}{2}$$
for all $x'\in\mathcal{V}_x$, $p\in\mathcal{O}_x$. Now let $x_1,...,x_s\in K$ be such that $K= \cup_{r=1}^{s}\mathcal{V}_{x_r}$ and define a neighbourhood $\mathcal{O'}=\cap_{r=1}^{s}\mathcal{O}_{x_r}$ of $p_0$. Then we have 
\begin{equation}\label{eq:5.21}
    \int_{|t_0,t_1|}\|j_m(X^{p_0}_sf\circ \Phi^{X^p}_{s,t_0})(x')-j_m(X^{p_0}_sf\circ \Phi^{X^{p_0}}_{s,t_0})(x')\|_{\mathbb{G}_{M,m}} ds <\frac{\epsilon}{2}
\end{equation}
for all $x'\in K$, $p\in\mathcal{O'}$. By (\ref{eq:2.4}), we can further shrink $\mathcal{O'}$ if necessary so that 
$$\int_{|t_0,t_1|}\|j_m(X^{p}f)(s,y)-j_m(X^{p_0}f)(s,y)\|_{\mathbb{G}_{M,m}}\d s<\frac{\epsilon}{2} \hspace{15pt} y'\in K', \ p\in\mathcal{O'}.$$
Then we have 
\begin{eqnarray*}
   &&\|j_m (f\circ\Phi^{X^p}_{t_1,t_0})(x)-j_m(f\circ \Phi^{X^{p_0}}_{t_1,t_0})(x)\|_{\mathbb{G}_{M,m}}\\
   &&\hspace{40pt}\leq\int_{|t_0,t_1|}\|j_m(X^{p}f\circ \Phi^{X^p}_{t,t_0})(x)-j_m(X^{p_0}f\circ \Phi^{X^{p_0}}_{t,t_0})(x)\|_{\mathbb{G}_{M,m}} ds\\
   && \hspace{40pt}\leq\int_{|t_0,t_1|}\|j_m(X^{p}f\circ \Phi^{X^p}_{t,t_0})(x)-j_m(X^{p_0}f\circ \Phi^{X^{p}}_{t,t_0})(x)\|_{\mathbb{G}_{M,m}} ds\\
   &&\hspace{50pt}+\int_{|t_0,t_1|}\|j_m(X^{p_0}f\circ \Phi^{X^p}_{t,t_0})(x)-j_m(X^{p_0}f\circ \Phi^{X^{p_0}}_{t,t_0})(x)\|_{\mathbb{G}_{M,m}} ds\\
   &&\hspace{40pt}\leq \frac{\epsilon}{2}+\frac{\epsilon}{2}=\epsilon,
\end{eqnarray*}
for $x\in K$ and $p\in\mathcal{O'}$, as desired. 
\subsubsection{The \texorpdfstring{$C^{m+\text{lip}}$}{}-case}
We will note a few facts here.
\begin{enumerate}
    \item The $C^{m+\text{lip}}$ topology for $C^{m+\text{lip}}$ is the initial topology induced by the $C^{\text{lip}}$-topology for $\Gamma^{\text{lip}}(J^m(M;\R))$ under the mapping $f\mapsto j_m f$. 
    
    \item The mapping 
    $$\mathcal{O}\ni p\mapsto \Phi^{v_m X^p}_{t_1,t_0}=j_m\Phi^{X^p}_{t_1,t_0}\in C^{\text{lip}}(\rho^{-1}_m(\mathcal{U});J^m(M;M))$$
    is well-defined and continuous (by the $C^m$-case proved in Section \ref{sec:3.3.2}). Thus in the diagram 
    \begin{center}
    \begin{tikzcd}[column sep=large]
    C^{m+\text{lip}}(\mathcal{U};M)
    \arrow[r,"\Phi\mapsto j_m\Phi"] &C^{\text{lip}}(\rho^{-1}_m(\mathcal{U});J^m(M;M)) 
     \\
    \mathcal{O} \arrow[u, "p\mapsto\Phi^{X^p}_{t_1,t_0}"]
    \arrow[ur, "p\mapsto j_m\Phi^{X^p}_{t_1,t_0}"', sloped]  
    \end{tikzcd}
    \end{center}
    the continuity of the diagonal map gives the continuity of the vertical map, as desired.

\end{enumerate}
\subsubsection{The \texorpdfstring{$C^{\infty}$}{}-case}
From the result in the $C^m$-case for $m\in\mathbb{Z}_{\geq 0}$, the mapping
$$\mathcal{O}\ni p\mapsto \Phi^{X^p}_{t_1,t_0}\in C^m(\mathcal{U};M)$$
is continuous for each $m\in\mathbb{Z}_{\geq 0}$. From the diagram 
\begin{center}
    \begin{tikzcd}[column sep=large]
    C^{\infty}(\mathcal{U};M)
    \arrow[r,"\hookrightarrow"] &C^{m}(\mathcal{U};M) 
     \\
    \mathcal{O} \arrow[u, "p\mapsto\Phi^{X^p}_{t_1,t_0}"]
    \arrow[ur, "p\mapsto j_m\Phi^{X^p}_{t_1,t_0}"', sloped]  
    \end{tikzcd}
\end{center}
and noting that the diagonal mappings in the diagram are continuous, we obtain the continuity of the vertical mapping as a result of the fact that the $C^\infty$-topology is the initial topology induced by the $C^m$-topologies, $m\in\mathbb{Z}_{\geq 0}$.

\subsubsection{The \texorpdfstring{$C^{\text{hol}}$}{}-case}
Since the $C^{\text{hol}}$-topology is the $C^0$
-topology, with the scalars extended
to be complex and the functions restricted to be holomorphic, the analysis in Section \ref{sec:3.3.1} can be carried out verbatim to give the theorem in the holomorphic case.

\subsubsection{The \texorpdfstring{$C^{\omega}$}{}-case}
The topology for $C^{\omega}(\mathcal{U};M)$ is the uniform topology defined by the semimetrics 
$$d^\omega_{K,f}(\Phi_1,\Phi_2)=\sup\{a_0a_1...a_m\|j_m (f\circ\Phi_1-f\circ\Phi_2)(x)\|_{G_{M,\pi,m}}\ | \ x\in K,\ m\in \mathbb{Z}_{\geq 0}\},$$
where $f\in C^{\omega}(M)$, $\ K\subseteq \mathcal{U}\ \text{compact}$ and $\boldsymbol{a}=(a_j)_{j\in \mathbb{Z}_{\geq 0}}\in c_0(\mathbb{Z}_{\geq 0};\R_{>0})$.

Thus, we must show that, for $f_1,...,f_m\in C^{\omega}(M)$, for $K_1,...,K_m\subseteq\mathcal{U}$ compact, for $\boldsymbol{a}^1,... \boldsymbol{a}^m\in c_0(\mathbb{Z}_{\geq 0};\R_{>0})$, and for $\epsilon_1,...,\epsilon_m\in\R_{>0}$, there exists a neighborhood $\mathcal{O'}$ of $p_0$ such that 
$$a^j_0a^j_1...a^j_n\|j_n(f_j\circ\Phi^{X^p}_{t_1,t_0}-f_j\circ \Phi^{X^{p_0}}_{t_1,t_0})(x_j)\|_{G_{M,\pi,n}}<\epsilon_j,\hspace{10pt} x_j\in K_j, \ p\in\mathcal{O'}, \ j\in\{1,...,m\},\ n\in\mathbb{Z}_{\geq 0}.$$

It will suffice to show that, for $f\in C^\omega(M)$, $K\subseteq\mathcal{U}$ compact, $\boldsymbol{a}=(a_j)_{j\in\mathbb{Z}_{\geq 0}}\in c_0(\mathbb{Z}_{\geq 0};\R_{>0})$, and for $\epsilon\in\R_{>0}$, there exists a neighborhood $\mathcal{O'}$ of $p_0$ such that 
$$a_0a_1...a_n\|j_n(f\circ\Phi^{X^p}_{t_1,t_0}-f\circ \Phi^{X^{p_0}}_{t_1,t_0})(x)\|_{G_{M,\pi,n}}<\epsilon,\hspace{10pt} x\in K, \ p\in\mathcal{O'}, \ n\in\mathbb{Z}_{\geq 0}.$$

Indeed, if we show that then, taking $K=\cup_{j=1}^{k}K_j$ and $\epsilon
=\min\{\epsilon_1,...,\epsilon_m\}$, we have 
$$a^j_0a^j_1...a^j_n\|j_n(f_j\circ\Phi^{X^p}_{t_1,t_0}-f_j\circ \Phi^{X^{p_0}}_{t_1,t_0})(x)\|_{G_{M,\pi,n}}<\epsilon,\hspace{10pt} x\in K, \ p\in\mathcal{O}_j, \ j\in\{1,...,m\},\ n\in\mathbb{Z}_{\geq 0}$$
for a suitable $\mathcal{O}_j$ for each $j\in\{1,...,m\}$. It then suffices to show that  
$$a_0a_1...a_n\|j_n(f\circ\Phi^{X^p}_{t_1,t_0}-f\circ \Phi^{X^{p_0}}_{t_1,t_0})(x)\|_{G_{M,\pi,n}}<\epsilon,\hspace{10pt} x\in K, \ p\in\mathcal{O}, \ n\in\mathbb{Z}_{\geq 0}$$
for $\mathcal{O'}=\cap_{j=1}^{m}\mathcal{O}_j$. This suffices to give the desired conclusion.

Let $f\in C^\omega(M)$, $K\subseteq\mathcal{U}$ compact, $\boldsymbol{a}=(a_j)_{j\in\mathbb{Z}_{\geq 0}}\in c_0(\mathbb{Z}_{\geq 0};\R_{>0})$, and let $\epsilon\in\R_{>0}$. As in all preceding cases, the estimate is broken into two parts.

For the first part, we let $\overbar{M}$ be a holomorphic extension of $M$ and, by Lemma \ref{lem:heotdvf}, let $\overbar{\mathcal{U}}\subseteq\overbar{M}$ be a neighbourhood of $M$ and let $\overbar{X}^{p_0}\in\Gamma^{\text{hol}}_{\text{LI}}(|t_0,t_1|;T\overbar{\mathcal{U}})$ be such that $\overbar{X}^{p_0}|M=X^{p_0}$. We also let $\overbar{f}$ be the extension of $f$, possibly after shrinking $\overbar{\mathcal{U}}$. Then the mapping
$$K\times\mathcal{O}\ni (z,p)\mapsto (t\mapsto \overbar{X}^{p_0}\overbar{f}\circ \Phi^{X^p}_{t,t_0}(z)))\in L^1(|t_0,t_1|;\mathbb{C})$$
is continuous, just as in the $C^0$-case from Section \ref{sec:3.3.1}. Therefore, restricting to $M$,
\begin{equation}\label{eq:4.3}
    K\times\mathcal{O}\ni (x,p)\mapsto (t\mapsto X^{p_0}f\circ \Phi^{X^p}_{t,t_0}(x)))\in L^1(|t_0,t_1|;\mathbb{R})
\end{equation}
is continuous. By Cauchy estimates for holomorphic sections \citep[Proposition 4.2]{Jafarpour2014}, there exists $C,r\in \R_{>0}$ such that 
\begin{eqnarray*}
   &&\|j_m(X^{p_0}f\circ\Phi^{X^p}_{t,t_0})(x)-j_m(X^{p_0}f\circ\Phi^{X^{p_0}}_{t,t_0})(x)\|_{\mathbb{G}_{M,\pi,m}}\\
   &&\hspace{80pt}\leq Cr^{-m}|X^{p_0}f\circ\Phi^{X^p}_{t,t_0}(x)-X^{p_0}f\circ\Phi^{X^{p_0}}_{t,t_0}(x)|, \hspace{20pt} x\in K, t\in|t_0,t_1|.
\end{eqnarray*} 
Without loss of generality, we can take $r\in (0,1)$. Let $N\in\mathbb{Z}_{\geq 0}$ be the smallest integer for which $a_j\leq r$ for $j\geq N$. Then for $m\in\mathbb{Z}_{\geq 0}$, we have 
$$Ca_0\frac{a_1}{r}\cdots\frac{a_m}{r}\leq \left\{\begin{array}{lcl} Ca_0\frac{a_1}{r}\cdots\frac{a_m}{r},\hspace{45pt}m\in\{0,1,...,N\},\\
Ca_0\frac{a_1}{r}\cdots\frac{a_N}{r},\hspace{45pt}m\geq N+1.\end{array}\right.$$
Denote 
$$M=\max\left\{Ca_0\frac{a_1}{r}\cdots\frac{a_m}{r}\ \big|\ m\in\{o,1,...,N\}\right\}$$
and for $x\in K$, let $\mathcal{V}_x\subseteq K$ be a relative neighborhood of $x$ and let $\mathcal{O}_x\subseteq\mathcal{O}$ be a
neighbourhood of $p_0$ such that
$$\int_{|t_0,t_1|}|X^{p_0}f(t,\Phi^{X^p}_{t,t_0}(x'))-X^{p_0}f(t,\Phi^{X^{p_0}}_{t,t_0}(x'))|\ dt<\frac{\epsilon}{2M} \hspace{15pt} x'\in\mathcal{V}_x, \ p\in\mathcal{O}_x,$$
this being possible by continuity of the mapping (\ref{eq:4.3}). We then have
\begin{eqnarray*}
   &&\int_{|t_0,t_1|}a_0a_1...a_m\|j_m(X^{p_0}f\circ\Phi^{X^p}_{t,t_0})(x')-j_m(X^{p_0}f\circ\Phi^{X^{p_0}}_{t,t_0})(x')\|_{\mathbb{G}_{M,\pi,m}}dt\\
   &&\hspace{5pt}\leq \int_{|t_0,t_1|}Ca_0\frac{a_1}{r}\cdots\frac{a_m}{r}|X^{p_0}f\circ\Phi^{X^p}_{t,t_0}(x')-X^{p_0}f\circ\Phi^{X^{p_0}}_{t,t_0}(x')|dt<\frac{\epsilon}{2},\hspace{5pt} x'\in\mathcal{V}_x, \ p\in\mathcal{O}_x,\ m\in\mathbb{Z}_{\geq 0}.
\end{eqnarray*}
Letting $x_1,...,x_s\in K$ be such that $K\subseteq \cup_{r=1}^{s}\mathcal{V}_{x_r}$ and defining $\mathcal{O'}=\cap_{r=1}^{k}\mathcal{O}_{x_r}$, we have  for $x\in K$ and $p\in\mathcal{O'}$, we have
\begin{eqnarray}\label{eq:5.23}
   &&\int_{|t_0,t_1|}a_0a_1...a_m\|j_m(X^{p_0}f\circ\Phi^{X^p}_{t,t_0})(x')-j_m(X^{p_0}f\circ\Phi^{X^{p_0}}_{t,t_0})(x')\|_{\mathbb{G}_{M,\pi,m}}dt<\frac{\epsilon}{2}\\
   &&\hspace{300pt}\ x'\in K, \ p\in\mathcal{O'},\ m\in\mathbb{Z}_{\geq 0}.\nonumber
\end{eqnarray}
For the second half of the estimate, routinely by (\ref{eq:2.4}), we can further shrink $\mathcal{O'}$ if necessary so that
$$\int_{|t_0,t_1|}a_0a_1...a_m\|j_mX^{p}f(s,x)-j_mX^{p_0}f(s,x)\|_{\mathbb{G}_{M,\pi,m}}\ ds<\frac{\epsilon}{2} \hspace{15pt} x\in K', \ p\in\mathcal{O'},m\in\mathbb{Z}_{\geq 0}.$$
Combining the two estimates, for a fixed $\boldsymbol{a}=(a_j)_{j\in\mathbb{Z}_{\geq 0}}\in c_0(\mathbb{Z}_{\geq 0};\R_{>0})$ and $m\in\mathbb{Z}_{\geq 0}$, we have 
\begin{eqnarray*}
   &&a_0a_1...a_m\|j_m(f\circ\Phi^{X^p}_{t_1,t_0}-f\circ \Phi^{X^{p_0}}_{t_1,t_0})(x)\|_{G_{M,\pi,m}}\\
   &&\hspace{15pt}\leq a_0a_1...a_m\left\|j_m\left(\int_{|t_0,t_1|}X^{p}f(t,\Phi^{X^p}_{t,t_0}(x))-X^{p_0}f(t,\Phi^{X^{p_0}}_{t,t_0}(x))\ dt\right)\right\|_{G_{M,\pi,m}}\\
   &&\hspace{15pt}\leq \int_{|t_0,t_1|}a_0a_1...a_m\|j_m(X^{p}f\circ\Phi^{X^p}_{t,t_0})(x)-j_m(X^{p_0}f\circ\Phi^{X^{p_0}}_{t,t_0})(x)\|_{\mathbb{G}_{M,\pi,m}}dt\\
   &&\hspace{15pt}\leq \int_{|t_0,t_1|}a_0a_1...a_m\|j_m(X^{p}f\circ\Phi^{X^p}_{t,t_0})(x)-j_m(X^{p_0}f\circ\Phi^{X^{p}}_{t,t_0})(x)\|_{\mathbb{G}_{M,\pi,m}}dt\\
   &&\hspace{20pt}+\int_{|t_0,t_1|}a_0a_1...a_m\|j_m(X^{p_0}f\circ\Phi^{X^p}_{t,t_0})(x)-j_m(X^{p_0}f\circ\Phi^{X^{p_0}}_{t,t_0})(x)\|_{\mathbb{G}_{M,\pi,m}}dt\\
   &&\hspace{15pt} \leq\frac{\epsilon}{2}+\frac{\epsilon}{2}=\epsilon,
\end{eqnarray*}
for $x\in K$ and $p\in\mathcal{O'}$, as desired. 
\end{proof}

\section{The exponential map}\label{sec:4}
For the development of the exponential map, we shall use the language of category theory and sheaf theory. For both vector fields and flows, we will work with presheaves defined by prescribing local sections over a basis for the topology. We also wish to talk about classes of vector fields with various properties, e.g., certain regularity or geometric properties.

In the last section, we have shown that exponential maps are defined locally on open cubes $\mathbb{S'}\times \mathbb{S}\times \mathcal{U}\subseteq\mathbb{T}\times\mathbb{T}\times M$. More precisely, we established the well-definedness of the map 
$$\overbar{\text{exp}}:  L^1_{\text{loc}}(\mathbb{T}; \Gamma^\nu(TM))\supseteq \mathcal{V}^{\nu}(\mathbb{S'}\times\mathbb{S}\times\mathcal{
U})\rightarrow \text{LocFlow}^\nu(\mathbb{S'};\mathbb{S};\mathcal{U})$$
from $\mathcal{V}^{\nu}(\mathbb{S'}\times\mathbb{S}\times\mathcal{
U})$, the space of time-verying sections whose flows are defined on $\mathbb{S'}\times \mathbb{S}\times \mathcal{U}$, to the space of their local flows of regularity $\nu$. This plays a key role in constructing the ultimate localisation of the map $\overbar{\text{exp}}$ between the presheaves of such spaces,
$$\text{exp}:\{\text{presheaves of vector fields}\}\rightarrow\{\text{presheaves of local flows}\}.$$

\subsection{Categories of time-varying sections}
We will define what we shall call a ``time-varying section." This is nothing more than a time-varying section, defined locally.
\begin{defin}[$C^\nu$-time-varying local section]
Let $m\in \mathbb{Z}_{\geq 0}$, let $m'\in\{0,\text{lip}\}$, let $\nu\in\{m+m',\infty,\omega,\text{hol}\}$, and let $r\in \{\infty,\omega,\text{hol}\}$, as required. A \textbf{\textit{(locally) integrally bounded local section of class}} $\mathbf{C^\nu}$ is a quadruple $(\mathbb{T},E,\mathcal{U},\xi)$ where 
\begin{enumerate}[(i)]
    \item $\mathbb{T}\subseteq \R$ is an interval,
    \item $\pi:E\rightarrow M$ is a $C^r$-vector bundle,
    \item $\mathcal{U}\subseteq M$ is open, and
    \item $\xi\in \Gamma^\nu_{\text{I}}(\mathbb{T};E|\mathcal{U})$ ($\in \Gamma^\nu_{\text{LI}}(\mathbb{T};E|\mathcal{U})$)
\end{enumerate}
\end{defin}
We denote $L^1_{\text{loc}}(\mathbb{T};\Gamma^\nu(E|\mathcal{U}))$ the set of integrally bounded local sections of class $C^\nu$. The reader may notice that we have relaboured constructions we have already made in Section \ref{sec:2.3}, the only deference being that we have introduced an open subset $\mathcal{U}\subseteq M$. This is true. The relabouring of this is simply for the purpose of pedagogy so that we have symmetry with our categories of flows when we introduce them in 3.2.1. 

Equipped with the seminorms defined by (\ref{eq:2.3}), the category we are building in this section is a subcategory of the category $\mathcal{LCTVS}$ of locally convex topological vector spaces. 

Now let us consider morphisms in the category we are building. We give the definition in the locally integrable case, but the obvious definition can also be made in the integrable case.
\begin{defin}[Morphisms of sets of local time-varying sections]\label{def:4.3}
Let $m\in \mathbb{Z}_{\geq 0}$, let $m'\in\{0,\text{lip}\}$, let $\nu\in\{m+m',\infty,\omega,\text{hol}\}$, and let $r\in \{\infty,\omega,\text{hol}\}$, as required. Let $\mathbb{T},\mathbb{S}\subseteq\R$ be intervals. A mapping 
$$\alpha:L^1_{\text{loc}}(\mathbb{T};\Gamma^\nu(E|\mathcal{U}))\rightarrow L^1_{\text{loc}}(\mathbb{S};\Gamma^\nu(F|\mathcal{V}))$$
is a morphism of sets of $C^\nu$-local time-varying sections if 
\begin{enumerate}
    \item[(i)]{
    $\alpha$ is continuous, and
    }
    \item[(ii)]{
    there exists a mapping $\tau_\alpha:\mathbb{T}\rightarrow\mathbb{S}$ and a $C^r$-vector bundle mapping $\Phi_\alpha\in \text{VB}^r(E;F)$ over $\Phi_{\alpha,0}\in C^r(M;N)$ such that:
    \begin{enumerate}
    \item[(a)]{
    $\tau_\alpha$ is the restriction to $\mathbb{T}$ of a nonconstant affine mapping,
    }
    \item[(b)]{
    $\Phi_{\alpha,0}(\mathcal{U})\subseteq\mathcal{V}$, and
    }
    \item[(c)]{
    $\Phi_\alpha\circ \xi(t,x)=\alpha(\xi)(t,\Phi_{\alpha,0}(x))$ for $t\in\mathbb{T}$ and $x\in\mathcal{U}$.
	}
\end{enumerate}
	}
\end{enumerate}
\end{defin}
Let us give some examples of morphisms of sets of time-varying local sections.
\begin{exam}[Morphisms of sets of local time-varying sections] 
If $E=F$, and if $\mathbb{S}\subseteq\mathbb{T}$ and $\mathcal{V}\subseteq\mathcal{U}$, then the ``restriction morphism"
$$\rho_{\mathbb{T}\times \mathbb{U},\mathbb{S}\times \mathcal{V}}:L^1_{\text{loc}}(\mathbb{T};\Gamma^\nu(E|\mathcal{U}))\rightarrow L^1_{\text{loc}}(\mathbb{S};\Gamma^\nu(E|\mathcal{V}))$$
given by 
$$\rho_{\mathbb{T}\times \mathbb{U},\mathbb{S}\times \mathcal{V}}(\xi)(t,x)=\xi(t,x),\hspace{10pt} (t,x)\in \mathbb{S}\times \mathcal{V}.$$
In this case we have ``$\tau_\alpha=\iota_{\mathbb{T},\mathbb{S}}$" and ``$\Phi_\alpha=\iota_{E|\mathcal{U},E|\mathcal{V}}$" where $\iota_{\mathbb{T},\mathbb{S}}$ and $\iota_{E|\mathcal{U},E|\mathcal{V}}$ are the inclusions. Of course, an entirely similar construction holds in the integrable case. 

One can directly verify that morphisms can be composed and that composition is associative. One also has an identity morphism
$$\text{id}:L^1_{\text{loc}}(\mathbb{T};\Gamma^\nu(E|\mathcal{U}))\rightarrow L^1_{\text{loc}}(\mathbb{T};\Gamma^\nu(E|\mathcal{U}))$$
given by ``$\tau_{\text{id}}=\text{id}_{\mathbb{T}}$" and ``$\Phi_{\text{id}}=\text{id}_{E|\mathcal{U}}$." This, then, gives the category $\mathcal{G}^\nu_{\text{LI}}$ of locally
integrally bounded local sections of class $C^\nu$ whose objects are the sets $L^1_{\text{loc}}(\mathbb{T};\Gamma^\nu(E|\mathcal{U}))$. of locally integrable local sections of class $C^\nu$ and whose morphisms are as defined above. This category is a subcategory of $\mathcal{LCTVS}$. As we shall see, we like the category $\mathcal{LCTVS}$ because direct limits exist in this category. One similarly denotes by $\mathcal{G}^\nu_{\text{I}}$ the category of
integrable local sections of class $C^\nu$. 
\end{exam}
\subsection{Presheaves of time-varying vector fields}
Throughout the following studies, we will investigate the open and connected subsets (thus a domain) $\mathcal{W}\subseteq \mathbb{T}\times\mathbb{T}\times M$ such that all the points $p_1=(t_1,t_0,x_0)\in\mathcal{W}$ satisfy the following properties:
\begin{enumerate}[(a)]
    \item if $t_1=t_0$, then there exists $p_2=(t_2,t_0,x_0)\in \mathcal{W}$ where $t_2\neq t_0$, and the segment from $p_1$ to $p_2$ must also be contained in $\mathcal{W}$;
    
    \item if $t_1\neq t_0$, then there exists $p_0=(t_0,t_0,x_0)\in \mathcal{W}$, and the segment from $p_1$ to $p_0$ must also be contained in $\mathcal{W}$.
\end{enumerate}

For convenience, we will call the subsets of $\mathbb{T}\times\mathbb{T}\times M$ with this property \textbf{flow admissible}. We have the following lemma.
\begin{lem}
Let $\mathcal{W}\subseteq\mathbb{T}\times\mathbb{T}\times M$ be open and flow admissible. Then there exists a countable cover $\{\mathbb{S'}_i\times\mathbb{S}_i\times \mathcal{U}_i\}_{i\in \mathbb{Z}_{>0}}$ that is flow admissible for each  $\mathbb{S'}_i\times\mathbb{S}_i\times \mathcal{U}_i$, $i\in \mathbb{Z}_{>0}$.
\end{lem}
\begin{proof}
Let $p_1=(t_1,t_0,x_0)\in \mathcal{W}$ be an arbitrary point, then $p_1$ is an interior point as $\mathcal{W}$ is open. Hence there exists an open ball $B_r(p_1)$ of radius $r$ cantered at $p_1$ that lies inside $\mathcal{W}$.
\begin{enumerate}[(i)]
    \item When $t_1=t_0$, then $p_1$ is be covered by an open cube $T\times T\times U\subset \mathbb{T}\times\mathbb{T}\times\mathcal{U}$ lies inside $B_r(p_1)$, and this cube is flow admissible.
    
    \item When $t_1\neq t_0$, since $\mathcal{W}$ is flow admissible, there exists a $p_0=(t_0,t_0,x_0)\in\mathcal{W}$ and the line segment from $p_1$ to $p_0$ lies inside $\mathcal{W}$. Since this line segment is compact, there is a finite open cover $\{T'_k \times T_k\times U_k\}_{k=1}^n$ such that $\bigcup_{k=1}^n T'_k \times T_k\times U_k\subset \mathcal{W}$. 
    
    Now denote 
    $$T'=\bigcup_{k=1}^n T'_k,\hspace{10pt} T=\bigcap_{k=1}^n T_k,\hspace{10pt} U=\bigcap_{k=1}^n U_k.$$
    Hence $T'\times T\times U$ is flow admissible, and covers the line segment from $p_0$ to $p_1$.
    
\end{enumerate}
 
\end{proof}

Now, consider the product space $\displaystyle\prod_{i\in\mathbb{Z}_{>0}}\mathcal{V}^\nu_{\mathbb{S'}_i\times\mathbb{S}_i\times\mathcal{U}_i}$ with the initial topology, i.e., the coarsest topology such that the following map
\begin{eqnarray*}
  \pi_{k}:\displaystyle\prod_{i\in\mathbb{Z}_{>0}}\mathcal{V}^\nu_{\mathbb{S'}_i\times\mathbb{S}_i\times\mathcal{U}_i}\rightarrow \mathcal{V}^\nu_{\mathbb{S'}_k\times\mathbb{S}_k\times\mathcal{U}_k}
\end{eqnarray*}
is continuous for each $k\in\mathbb{Z}_{>0}$.

\textbf{Properties of this topology:}
\begin{enumerate}[(i)]
    \item It is Hausdorff (as the product of Hausdorff spaces is Hausdorff).
    
    \item It is complete (as the product of complete spaces is complete).
    
    \item It is separable (as the product of separable spaces is separable).
    
    \item It is Suslin (as the product of Suslin spaces is Suslin).
\end{enumerate}

For any open subset $\mathcal{W}\subseteq \mathbb{T}\times\mathbb{T}\times M$, we will define the presheaf of sets over $\mathcal{W}$, denoted by $\mathscr{G}^{\nu}_{\text{LI}}(\mathbb{T};TM)(\mathcal{W})$. Let $\mathcal{W}$ be covered by a family $\{\mathbb{S'}_i\times\mathbb{S}_i\times \mathcal{U}_i\}_{i\in \mathbb{Z}_{>0}}$ that is flow admissible for each $i\in\mathbb{Z}_{>0}$. We assign
$$\mathscr{G}^{\nu}_{\text{LI}}(\mathbb{T};TM)(\mathcal{W})\subseteq\displaystyle\prod_{i\in\mathbb{Z}_{>0}}\mathcal{V}^\nu_{\mathbb{S'}_i\times\mathbb{S}_i\times\mathcal{U}_i}$$
consisting of all sequences $(X_i)_{i\in\mathbb{Z}_{>0}}$, $X_i\in \mathcal{V}^\nu_{\mathbb{S'}_i\times\mathbb{S}_i\times\mathcal{U}_i}$ for which 
$\Phi^{X_j}|_{R}=\Phi^{X_k}|_{R}$ whenever $R=(\mathbb{S'}_j\times\mathbb{S}_j\times\mathcal{U}_j)\cap(\mathbb{S'}_k\times\mathbb{S}_k\times\mathcal{U}_k)$ for some $j,k\in\mathbb{Z}_{>0}$. For convenience, we call this the \textbf{``overlap condition."}
\begin{lem}\label{lem:2.2}
Let $\mathbb{S'}\times\mathbb{S}\times\mathcal{U},\Tilde{\mathbb{S'}}\times\Tilde{\mathbb{S}}\times\Tilde{\mathcal{U}}\subseteq \mathbb{T}\times M$ be flow admissible and intersecting each other. Let $X\in \mathcal{V}^\nu_{\mathbb{S'}\times\mathbb{S}\times\mathcal{U}}$ and $Y\in \mathcal{V}^\nu_{\Tilde{\mathbb{S'}}\times\Tilde{\mathbb{S}}\times\Tilde{\mathcal{U}}}$ satisfy the overlap condition. Then there exists a $Z\in\mathcal{V}^\nu_{(\mathbb{S'}\times\mathbb{S}\times\mathcal{U})\cup(\Tilde{\mathbb{S'}}\times\Tilde{\mathbb{S}}\times\Tilde{\mathcal{U}})}$ such that $\Phi^Z|_{\mathbb{S'}\times\mathbb{S}\times\mathcal{U}}=\Phi^X$ and $\Phi^Z|_{\Tilde{\mathbb{S'}}\times\Tilde{\mathbb{S}}\times\Tilde{\mathcal{U}}}=\Phi^Y$.  
\end{lem}
\begin{proof}
Denote $(\mathbb{S'}\times\mathbb{S}\times\mathcal{U})\cap(\Tilde{\mathbb{S'}}\times\Tilde{\mathbb{S}}\times\Tilde{\mathcal{U}})=R$. Since $X\in \mathcal{V}^\nu_{\mathbb{S'}\times\mathbb{S}\times\mathcal{U}}$ and $Y\in \mathcal{V}^\nu_{\Tilde{\mathbb{S'}}\times\Tilde{\mathbb{S}}\times\Tilde{\mathcal{U}}}$, then the flow for $X$ is the mapping
\begin{eqnarray*}
    \Phi^X:\mathbb{S'}\times\mathbb{S}\times\mathcal{U} &\rightarrow&  M\\
	  (t,t_0,x_0)&\mapsto& \xi(t),
\end{eqnarray*}
where $\xi$ is a locally absolutely continuous curve $\xi:\mathbb{I}\subseteq\mathbb{T}\rightarrow M$ satisfying $\xi(t_0)=x_0$ for every $(t_0,x_0)\in \mathbb{S}\times\mathcal{U}$ and $\xi'(t)=X(t,\xi(t))$ for almost every $t\in\mathbb{S'}$; the flow for $Y$ is the mapping
\begin{eqnarray*}
    \Phi^Y:\Tilde{\mathbb{S'}}\times\Tilde{\mathbb{S}}\times\Tilde{\mathcal{U}} &\rightarrow&  M\\
	  (t,t_0,x_0)&\mapsto& \eta(t),
\end{eqnarray*}
where $\eta$ is a locally absolutely continuous curve $\eta:\mathbb{I'}\subseteq\mathbb{T}\rightarrow M$ satisfying $\eta(t_0)=x_0$ for every $(t_0,x_0)\in \Tilde{\mathbb{S}}\times\Tilde{\mathcal{U}}$ and $\eta'(t)=Y(t,\eta(t))$ for almost every $t\in\Tilde{\mathbb{S'}}$.

Since $X$ and $Y$ satisfy the overlap condition, then $\Phi^X|_{R}=\Phi^Y|_{R}$, i.e., $\xi(t)=\eta(t)$ for each $(t,t_0,x_0)\in R$. Now let $\Psi$ be a flow such that 
\begin{eqnarray*}
    \Psi:(\mathbb{S'}\times\mathbb{S}\times\mathcal{U})\cup (\Tilde{\mathbb{S'}}\times\Tilde{\mathbb{S}}\times\Tilde{\mathcal{U}} ) &\rightarrow&  M\\
	  (t,t_0,x_0)&\mapsto& \zeta(t),
\end{eqnarray*}
and $\Psi|_{\mathbb{S'}\times\mathbb{S}\times\mathcal{U}}=\Phi^X$ and $\Psi|_{\Tilde{\mathbb{S'}}\times\Tilde{\mathbb{S}}\times\Tilde{\mathcal{U}}}=\Phi^Y$, i.e., $\eta$ is a locally absolutely continuous curve $\eta:\mathbb{T'}\subseteq\mathbb{T}\rightarrow M$ satisfying
\begin{enumerate}[(a)]
    \item if $(t_0,x_0)\in\mathbb{S}\times \mathcal{U}$, then $\zeta(t_0)=x_0$ and $\zeta(t)=\xi(t)$ for almost every $t\in\mathbb{S'}$;
    \item if $(t_0,x_0)\in\Tilde{\mathbb{S}}\times\Tilde{\mathcal{U}}$, then $\zeta(t_0)=x_0$ and $\zeta(t)=\eta(t)$ for almost every $t\in\Tilde{\mathbb{S'}}$.
\end{enumerate}
Then let $Z$ be a vector field such that $\zeta'(t)=Z(t,\zeta(t))$ for almost every $t\in\Tilde{\mathbb{T'}}$. Then it is obvious that $\Psi$ is the flow for $Z$ such that $\Psi|_{\mathbb{S'}\times\mathbb{S}\times\mathcal{U}}=\Phi^X$ and $\Psi|_{\Tilde{\mathbb{S'}}\times\Tilde{\mathbb{S}}\times\Tilde{\mathcal{U}}}=\Phi^Y$.
\end{proof}
\begin{lem}
Let $\mathbb{S'}\times\mathbb{S}\times\mathcal{U},\Tilde{\mathbb{S'}}\times\Tilde{\mathbb{S}}\times\Tilde{\mathcal{U}}\subseteq \mathbb{T}\times M$ be flow admissible and that $\Tilde{\mathbb{S'}}\times\Tilde{\mathbb{S}}\times\Tilde{\mathcal{U}}\subseteq\mathbb{S'}\times\mathbb{S}\times\mathcal{U}$. Then the map
$$\rho_{\mathbb{S'}\times\mathbb{S}\times\mathcal{U},\Tilde{\mathbb{S'}}\times\Tilde{\mathbb{S}}\times\Tilde{\mathcal{U}}}:\mathcal{V}^\nu_{\mathbb{S'}\times\mathbb{S}\times\mathcal{U}}\rightarrow \mathcal{V}^\nu_{\Tilde{\mathbb{S'}}\times\Tilde{\mathbb{S}}\times\Tilde{\mathcal{U}}}$$
given by 
$$\rho_{\mathbb{S'}\times\mathbb{S}\times\mathcal{U},\Tilde{\mathbb{S'}}\times\Tilde{\mathbb{S}}\times\Tilde{\mathcal{U}}}(X)=X$$
is homeomorphism onto its image.
\end{lem}
\begin{proof}
Let $X\in\mathcal{V}^\nu_{\mathbb{S'}\times\mathbb{S}\times\mathcal{U}}$, i.e., $\mathbb{S'}\times\mathbb{S}\times\mathcal{U}\subseteq D_X$. Then $\Tilde{\mathbb{S'}}\times\Tilde{\mathbb{S}}\times\Tilde{\mathcal{U}}\subseteq D_X$, whence $\mathcal{V}^\nu_{\mathbb{S'}\times\mathbb{S}\times\mathcal{U}}\subseteq\mathcal{V}^\nu_{\Tilde{\mathbb{S'}}\times\Tilde{\mathbb{S}}\times\Tilde{\mathcal{U}}}$. Therefore $\rho_{\mathbb{S'}\times\mathbb{S}\times\mathcal{U},\Tilde{\mathbb{S'}}\times\Tilde{\mathbb{S}}\times\Tilde{\mathcal{U}}}$ is an inclusion map, hence homeomorphism onto its image.
\end{proof}
\begin{thm}\label{thm:2.3}
$\mathscr{G}^{\nu}_{\text{LI}}(\mathbb{T};TM)(\mathcal{W})$ is unique up to homeomorphisms, i.e., it is independent of the choices of the covers for $\mathcal{W}$. 
\end{thm}
\begin{proof}
Let $\{\mathbb{S'}_i\times\mathbb{S}_i\times \mathcal{U}_i\}_{i\in \mathbb{Z}_{>0}}$ and $\{\Tilde{\mathbb{S'}}_j\times\Tilde{\mathbb{S}}_j\times \Tilde{\mathcal{U}}_j\}_{j\in \mathbb{Z}_{>0}}$ be two open covers of $\mathcal{W}$ and satisfy the semi-uniform condition for each $i,j\in\mathbb{Z}_{>0}$, and let $\mathcal{P}\subseteq \displaystyle\prod_{i\in\mathbb{Z}_{>0}}\mathcal{V}^\nu_{\mathbb{S'}_i\times\mathbb{S}_i\times\mathcal{U}_i}$ and $\mathcal{Q}\subseteq \displaystyle\prod_{j\in\mathbb{Z}_{>0}}\mathcal{V}^\nu_{\Tilde{\mathbb{S'}}_j\times \Tilde{\mathbb{S}}_j\times\Tilde{\mathcal{U}}_j}$ be the subsets that satisfies the overlap condition. Consider the map
\begin{eqnarray*}
  T:\mathcal{P}&\rightarrow & \mathcal{Q}\\
       (X_1,X_2,...)&\mapsto& (Y_1,Y_2,...)
\end{eqnarray*}
where $X_i\in \mathcal{V}^\nu_{\mathbb{S'}_i\times\mathbb{S}_i\times\mathcal{U}_i},\ i\in\mathbb{Z}_{>0}$ and $Y_j\in \mathcal{V}^\nu_{\Tilde{\mathbb{S'}}_j\times \Tilde{\mathbb{S}}_j\times\Tilde{\mathcal{U}}_j},\ j\in\mathbb{Z}_{>0}$ satisty the over lap condition, and if $R=(\mathbb{S'}_k\times\mathbb{S}_k\times\mathcal{U}_k)\cap(\Tilde{\mathbb{S'}}_l\times\Tilde{\mathbb{S}}_l\times\Tilde{\mathcal{U}}_l)\neq \emptyset$, then $\Phi^{X_k}|_{R}=\Phi^{Y_l}|_{R}$.
Consider the map
\begin{eqnarray*}
  T':\mathcal{Q}&\rightarrow & \mathcal{P}\\
       (Y_1,Y_2,...)&\mapsto& (X_1,X_2,...)
\end{eqnarray*}
where $X_i\in \mathcal{V}^\nu_{\mathbb{S'}_i\times\mathbb{S}_i\times\mathcal{U}_i},\ i\in\mathbb{Z}_{>0}$ and $Y_j\in \mathcal{V}^\nu_{\Tilde{\mathbb{S'}}_j\times \Tilde{\mathbb{S}}_j\times\Tilde{\mathcal{U}}_j},\ j\in\mathbb{Z}_{>0}$ satisty the over lap condition, and if $R=(\mathbb{S'}_k\times\mathbb{S}_k\times\mathcal{U}_k)\cap(\Tilde{\mathbb{S'}}_l\times\Tilde{\mathbb{S}}_l\times\Tilde{\mathcal{U}}_l)\neq \emptyset$, then $\Phi^{X_k}|_{R}=\Phi^{Y_l}|_{R}$.

By the overlap condition, $T\circ T'=T'\circ T=\text{Id}$, thus $T$ is an isomorphism with $T^{-1}=T'$. Now we need to show that T is a homeomorphism.

Let $\{\Tilde{\mathbb{S'}}_{n_l}\times\Tilde{\mathbb{S}}_{n_l}\times\Tilde{\mathcal{U}}_{n_l}\}_{l\in\mathbb{Z}_{>0}}$ be an open cover for $\mathbb{S'}_k\times\mathbb{S}_k\times \mathcal{U}_k$ for some $n_l\in\{1,2,...\}$, such that $(\mathbb{S'}_k\times\mathbb{S}_k\times \mathcal{U}_k)\cap(\Tilde{\mathbb{S'}}_{n_l}\times\Tilde{\mathbb{S}}_{n_l}\times\Tilde{\mathcal{U}}_{n_l})\neq \emptyset$ for all $l\in\mathbb{Z}_{>0}$. Denote $R_{kl}=(\mathbb{S'}_k\times\mathbb{S}_k\times \mathcal{U}_k)\cap(\Tilde{\mathbb{S'}}_{n_l}\times\Tilde{\mathbb{S}}_{n_l}\times\Tilde{\mathcal{U}}_{n_l})$. Consider the following diagram 
\begin{center}
\begin{tikzcd}[column sep=large]
\displaystyle\prod_{i\in\mathbb{Z}_{>0}}\mathcal{V}^\nu_{\mathbb{S'}_i\times\mathbb{S}_i\times\mathcal{U}_i}\supseteq\mathcal{P} \arrow[d,dashed, start anchor={[xshift=8ex]}, end anchor={[xshift=8ex]}, "\phi"] \arrow[r,"\pi_k"] &\mathcal{R}\subseteq\mathcal{V}^\nu_{\mathbb{S'}_k\times\mathbb{S}_k\times\mathcal{U}_k} 
 \\
\displaystyle\prod_{j\in\mathbb{Z}_{>0}}\mathcal{V}^\nu_{\Tilde{\mathbb{S'}}_j\times \Tilde{\mathbb{S}}_j\times\Tilde{\mathcal{U}}_j}\supseteq\mathcal{Q} \arrow[u,dashed,start anchor={[xshift=7ex]}, end anchor={[xshift=7ex]}, "\phi'"]
\arrow[ur, start anchor={[xshift=5ex]}, "\psi"]  
\end{tikzcd}
\end{center}
where $\psi$ is the map
\begin{eqnarray*}
  \psi:\hspace{10pt}\mathcal{Q}&\rightarrow & \mathcal{R}\subseteq\mathcal{V}^\nu_{\mathbb{S'}_k\times\mathbb{S}_k\times\mathcal{U}_k}\\
       (X_1,X_2,...)&\mapsto& Y
\end{eqnarray*}
such that $Y\in \mathcal{V}^\nu_{\mathbb{S'}_k\times\mathbb{S}_k\times\mathcal{U}_k}$ satisfies $\Phi^Y|_{R_{kl}}=\Phi^{X_{n_l}}|_{R_{kl}}$ for all $l\in\mathbb{Z}_{>0}$. Lemma \ref{lem:2.2} shows that that $\psi$ is well-defined. 

Now, we claim that $\psi$ is continuous.

\textbf{Apporach 1:} Consider the following diagram:
\begin{center}
\begin{tikzcd}[column sep=large]
\displaystyle\prod_{j\in\mathbb{Z}_{>0}}\mathcal{V}^\nu_{\Tilde{\mathbb{S'}}_j\times \Tilde{\mathbb{S}}_j\times\Tilde{\mathcal{U}}_j}\supseteq\mathcal{Q} \arrow[r, "\pi'_k"] \arrow[rrd, "\psi"] \arrow[rr,"T", bend left, end anchor={[xshift=2ex]}, start anchor={[xshift=5ex]}] &\mathcal{V}^\nu_{\Tilde{\mathbb{S'}}_{n_l}\times \Tilde{\mathbb{S}}_{n_l}\times\Tilde{\mathcal{U}}_{n_l}} \arrow[r, "\rho_{\Tilde{\mathbb{S'}}_{n_l}\times \Tilde{\mathbb{S}}_{n_l}\times\Tilde{\mathcal{U}}_{n_l},R_{kl}}"]& \mathcal{V}^\nu_{R_{kl}}
 \\
 &  &\mathcal{V}^\nu_{\mathbb{S'}_k\times\mathbb{S}_k\times\mathcal{U}_k}\arrow[u, "\rho_{\mathbb{S'}_k\times\mathbb{S}_k\times\mathcal{U}_k,R_{kl}}"]
\end{tikzcd}
\end{center}
It is convincing that this diagram commutes. Now let us describe the map $\rho_{\mathbb{S'}_k\times\mathbb{S}_k\times\mathcal{U}_k,R_{kl}}$,
\begin{eqnarray*}
  \rho_{\mathbb{S'}_k\times\mathbb{S}_k\times\mathcal{U}_k,R_{kl}}:\mathcal{V}^\nu_{\mathbb{S'}_k\times\mathbb{S}_k\times\mathcal{U}_k}&\rightarrow & \mathcal{V}^\nu_{R_{kl}}\\
       X&\mapsto& Y,
\end{eqnarray*}
where $\Phi^X|_{R_{kl}}=\Phi^Y$. Let $\mathcal{O}\subseteq \text{Im}(\psi)\subseteq \mathcal{V}^\nu_{\mathbb{S'}_k\times\mathbb{S}_k\times\mathcal{U}_k}$ be open. Since $\rho_{\mathbb{S'}_k\times\mathbb{S}_k\times\mathcal{U}_k,R_{kl}}$ is an open map, $\rho_{\mathbb{S'}_k\times\mathbb{S}_k\times\mathcal{U}_k,R_{kl}}(\mathcal{O})$ is open in $\mathcal{V}^\nu_{R_{kl}}$. Denote $T:=\rho_{\Tilde{\mathbb{S'}}_{n_l}\times \Tilde{\mathbb{S}}_{n_l}\times\Tilde{\mathcal{U}}_{n_l},R_{kl}}\circ\pi'_k$. Since $\rho_{\Tilde{\mathbb{S'}}_{n_l}\times \Tilde{\mathbb{S}}_{n_l}\times\Tilde{\mathcal{U}}_{n_l},R_{kl}}$ and $\pi'_k$ are both continuous, so is $T$. Then there exists an open set $U\subseteq \mathcal{Q}$ such that $T (U)\subseteq \rho_{\mathbb{S'}_k\times\mathbb{S}_k\times\mathcal{U}_k,R_{kl}}(\mathcal{O})$. Since for any map $f:A\rightarrow B$, $D\subseteq f^{-1}(f(D))$ for all subsets $D\subseteq A$., hence 
$$U\subseteq T^{-1}\circ T(U)\subseteq T^{-1}\circ \rho_{\mathbb{S'}_k\times\mathbb{S}_k\times\mathcal{U}_k,R_{kl}} (\mathcal{O}).$$
Since $T=\rho_{\mathbb{S'}_k\times\mathbb{S}_k\times\mathcal{U}_k,R_{kl}}\circ \psi$, then $T^{-1}=\psi^{-1}\circ (\rho_{\mathbb{S'}_k\times\mathbb{S}_k\times\mathcal{U}_k,R_{kl}})^{-1}$. Hence 
\begin{eqnarray*}
	\psi(U)
	    & \subseteq & \psi\circ T^{-1}\circ \rho_{\mathbb{S'}_k\times\mathbb{S}_k\times\mathcal{U}_k,R_{kl}} (\mathcal{O})\\
 	    & = &\psi\circ (\psi^{-1}\circ (\rho_{\mathbb{S'}_k\times\mathbb{S}_k\times\mathcal{U}_k,R_{kl}})^{-1})\circ \rho_{\mathbb{S'}_k\times\mathbb{S}_k\times\mathcal{U}_k,R_{kl}} (\mathcal{O}) \\
 	    & = &(\psi\circ \psi^{-1})\circ [(\rho_{\mathbb{S'}_k\times\mathbb{S}_k\times\mathcal{U}_k,R_{kl}})^{-1}\circ \rho_{\mathbb{S'}_k\times\mathbb{S}_k\times\mathcal{U}_k,R_{kl}}] (\mathcal{O})\\
	    & = & (\rho_{\mathbb{S'}_k\times\mathbb{S}_k\times\mathcal{U}_k,R_{kl}})^{-1}\circ \rho_{\mathbb{S'}_k\times\mathbb{S}_k\times\mathcal{U}_k,R_{kl}} (\mathcal{O})=\mathcal{O}
\end{eqnarray*}
provided that $(\rho_{\mathbb{S'}_k\times\mathbb{S}_k\times\mathcal{U}_k,R_{kl}})^{-1}\circ \rho_{\mathbb{S'}_k\times\mathbb{S}_k\times\mathcal{U}_k,R_{kl}}(\mathcal{O})=\text{Id}(\mathcal{O})$. Indeed, this is true under the condition that $\rho_{\mathbb{S'}_k\times\mathbb{S}_k\times\mathcal{U}_k,R_{kl}}(\mathcal{O})$ is injective. Hence $\psi$ is continuous. 

\textbf{Approach 2:} $\psi$ is the map
\begin{eqnarray*}
  \psi:\displaystyle\prod_{j\in\mathbb{Z}_{>0}}\mathcal{V}^\nu_{\Tilde{\mathbb{S'}}_j\times \Tilde{\mathbb{S}}_j\times\Tilde{\mathcal{U}}_j}\supseteq\hspace{10pt}\mathcal{Q}&\rightarrow & \mathcal{R}\subseteq\mathcal{V}^\nu_{\mathbb{S'}_k\times\mathbb{S}_k\times\mathcal{U}_k}\\
      \boldsymbol{X}:=(X_1,X_2,...)&\mapsto& Y
\end{eqnarray*}
such that $Y\in \mathcal{V}^\nu_{\mathbb{S'}_k\times\mathbb{S}_k\times\mathcal{U}_k}$ satisfies $\Phi^Y|_{R_{kl}}=\Phi^{X_{n_l}}|_{R_{kl}}$ for all $l\in\mathbb{Z}_{>0}$.
\begin{enumerate}
    \item $\psi$ is a linear map. 
    
    \textbf{Proof:} Let $\boldsymbol{X_1}=(X_1^1,X_2^1,...), \boldsymbol{X_2}=(X^2_1,X^2_2,...)\in \displaystyle\prod_{j\in\mathbb{Z}_{>0}}\mathcal{V}^\nu_{\Tilde{\mathbb{S'}}_j\times \Tilde{\mathbb{S}}_j\times\Tilde{\mathcal{U}}_j}$, and let $\psi(\boldsymbol{X_1}+\boldsymbol{X_2})=Y$, $\psi(\boldsymbol{X_1})=Y_1$, $\psi(\boldsymbol{X_2})=Y_2$. Since $Y_1$ is such that $\Phi^{Y_1}|_{R_{kl}}=\Phi^{X^1_{n_l}}|_{R_{kl}}$ and $Y_2$ is such that $\Phi^{Y_2}|_{R_{kl}}=\Phi^{X^2_{n_l}}|_{R_{kl}}$ for all $l\in\mathbb{Z}_{>0}$, then 
    $$Y_1|_{R_{kl}}=(\Phi^{Y_1}|_{R_{kl}})'=(\Phi^{X^1_{n_l}}|_{R_{kl}})'=X^1_{n_l}|_{R_{kl}}\hspace{10pt}\text{ for all }l\in\mathbb{Z}_{>0};$$
    $$Y_2|_{R_{kl}}=(\Phi^{Y_2}|_{R_{kl}})'=(\Phi^{X^2_{n_l}}|_{R_{kl}})'=X^2_{n_l}|_{R_{kl}}\hspace{10pt}\text{ for all }l\in\mathbb{Z}_{>0}.$$
    On the other hand,
    $$Y|_{R_{kl}}=(\Phi^Y|_{R_{kl}})'=(\Phi^{X^1_{n_l}+X^2_{n_l}}|_{R_{kl}})'=(X^1_{n_l}+X^2_{n_l})|_{R_{kl}}=X^1_{n_l}|_{R_{kl}}+X^2_{n_l}|_{R_{kl}}\hspace{10pt}\text{ for all }l\in\mathbb{Z}_{>0},$$
    hence 
    $$Y|_{R_{kl}}=Y_1|_{R_{kl}}+Y_2|_{R_{kl}}\hspace{10pt}\text{ for all }l\in\mathbb{Z}_{>0},$$
    i.e., $Y=Y_1+Y_2$.
    
    \item $\psi$ is continuous. \textbf{Proof:} Let $\boldsymbol{X}:=(X_{i})_{i\in\mathbb{Z}_{>0}}$ be an element in $\mathcal{Q}$ and let $p_{K,\mathbb{I},1}$ be a seminorm on $\mathcal{R}\subseteq\mathcal{V}^\nu_{\mathbb{S'}_k\times\mathbb{S}_k\times\mathcal{U}_k}$, where $K\subseteq \mathbb{M}$ and $\mathbb{I}\subseteq\mathbb{T}$ compact and $\mathcal{U}_k\subseteq K$, $\mathbb{S'}_k\subseteq\mathbb{I}$. Let $\{\Tilde{\mathbb{S'}}_{n_l}\times\Tilde{\mathbb{S}}_{n_l}\times\Tilde{\mathcal{U}}_{n_l}\}_{l\in\mathbb{Z}_{>0}}$ be an open cover for $\mathbb{S'}_k\times\mathbb{S}_k\times \mathcal{U}_k$ for some $n_l\in\{1,2,...\}$, such that $(\mathbb{S'}_k\times\mathbb{S}_k\times \mathcal{U}_k)\cap(\Tilde{\mathbb{S'}}_{n_l}\times\Tilde{\mathbb{S}}_{n_l}\times\Tilde{\mathcal{U}}_{n_l})\neq \emptyset$ for all $l\in\mathbb{Z}_{>0}$. For a fixed $l\in \mathbb{Z}_{>0}$, let $K'_{n_l}:=K\cup \Tilde{\mathcal{U}}_{n_l}$ and $\mathbb{I'}_{n_l}:=\mathbb{I}\cup \Tilde{\mathbb{S'}}_{n_l}$. Then $p_{K'_{n_l},\mathbb{I'}_{n_l},1}$ is a seminorm on $\mathcal{V}^\nu_{\Tilde{\mathbb{S'}}_{n_l}\times\Tilde{\mathbb{S}}_{n_l}\times\Tilde{\mathcal{U}}_{n_l}}$. Since $$\pi_{n_l}:\displaystyle\prod_{j\in\mathbb{Z}_{>0}}\mathcal{V}^\nu_{\Tilde{\mathbb{S'}}_j\times \Tilde{\mathbb{S}}_j\times\Tilde{\mathcal{U}}_j}\supseteq\mathcal{Q}\rightarrow\mathcal{V}^\nu_{\Tilde{\mathbb{S'}}_{n_l}\times\Tilde{\mathbb{S}}_{n_l}\times\Tilde{\mathcal{U}}_{n_l}}$$ 
    is continuous, there exists a seminorm $q_{n_l}$ on $\mathcal{Q}$ and $C_l\in\R_{>0}$ such that 
    $$q_{n_l}(\boldsymbol{X})\leq C_l\ p_{K'_{n_l},\mathbb{I'}_{n_l},1}(X_{n_l})$$
    Then there exists a $q\in\{q_{n_1}, q_{n_2},...\}$ on $\mathcal{Q}$ such that 
    $$q(\boldsymbol{X})\leq q_{n_l}(\boldsymbol{X}) \hspace{15pt} \text{for all } l\in\mathbb{Z}_{>0}.$$

    Since $K\subseteq K'_{n_l}$ and $\mathbb{I}\subseteq\mathbb{I'}_{n_l}$, then $p_{K'_{n_l},\mathbb{I'}_{n_l},1}(X_{n_l})\leq p_{K,\mathbb{I},1}(X_{n_l})$. This is true for every $l\in\mathbb{Z}_{>0}$. 
\end{enumerate}

Since $\iota,\ \iota',\ \pi_k,\ \psi$ are continuous, by the universal property of products and subspaces, there exist $\phi$ and $\phi'$ such that
\begin{equation}\label{eq:2.1}
   \psi\circ\iota'\circ \phi=\pi_k\circ\iota \text{ and } 
\end{equation}
\begin{equation}\label{eq:2.2}
    \pi_k\circ\iota\circ\phi'=\psi\circ\iota'.
\end{equation}
Now we need to show that $T=\phi$ and $T'=\phi'$. Indeed, $T$ has the property that satisfies (\ref{eq:2.1}), hence $T=\phi$ by the uniqueness of the map $\phi$. Similarly, $T'$ has the property that satisfies (\ref{eq:2.2}), hence $T=\phi'$ by the uniqueness of the map $\phi'$.
\end{proof}

\subsection{Category of time-varying local flows}
We have thoroughly developed in Section 4.1 and 4.2 our presheaf theoretic notion of what a vector field is. In this section we begin to develop the presheaf/groupoid point of view for flows. What we do in this section is develop the space that will be the codomain of the exponential map. We do this by developing a vector field independent theory of flows that is analogous to our theory of vector fields in Section 4.1, we develop a notion of ``category of flows” that we will use as the basis for defining
the ``flow presheaf.” 
\begin{defin}[$C^\nu$-local flow]
Let $m\in \mathbb{Z}_{\geq 0}$, let $m'\in\{0,\text{lip}\}$, let $\nu\in\{m+m',\infty,\omega,\text{hol}\}$, and let $r\in \{\infty,\omega,\text{hol}\}$, as required. A ``$C^\nu$-local flow" is a quintuple $(\mathbb{T'},\mathbb{T},M,\mathcal{U},\Phi)$ where 
\begin{enumerate}
    \item[(i)]{
    $\mathbb{T}\subseteq \mathbb{T'}\subseteq\R$ are intervals,
    }
    \item[(ii)]{
    $M$ is a $C^\nu$-manifold,
    }
    \item[(iii)]{
    $\mathcal{U}\subseteq M$ is open, and
    }
    \item[(iv)]{
    $\Phi:\mathbb{T'}\times\mathbb{T}\times \mathcal{U}\rightarrow M$ is such that:
    \begin{enumerate}
    \item[(a)]{
    $\Phi(t_0,t_0,x)=x,\;(t_0,t_0,x)\in \mathbb{T'}\times\mathbb{T}\times \mathcal{U}$;
    }
    \item[(b)]{
    $\Phi(t_2,t_1,\Phi(t_1,t_0,x))=\Phi(t_2,t_0,x)$, $t_0,t_1,\in\mathbb{T},\; t_2\in\mathbb{T'}\;x\in\mathcal{U}$;
    }
    \item[(c)]{
    the map $x\mapsto\Phi(t_1,t_0,x)$ is a $C^\nu$-diffeomorphism onto its image for every $t_0\in\mathbb{T}$ and $t_1\in\mathbb{T'}$;
	}
	\item[(d)]{
    the map $(t_1,t_0)\mapsto\Phi_{t_1,t_0}\in C^\nu(\mathcal{U};M)$ is continuous, where $\Phi_{t_1,t_0}(x)=\Phi(t_1,t_0,x)$.
	}
\end{enumerate}
	}
\end{enumerate}
\end{defin}
We denote
$$\text{LocFlow}^\nu(\mathbb{S'}\times\mathbb{S}\times\mathcal{U})=\{\Phi:\mathbb{S'}\times\mathbb{S}\times \mathcal{U}\rightarrow M\;|\;(\mathbb{S'},\mathbb{S},M,\mathcal{U},\Phi) \text{ is a $C^\nu$-local flow}\}.$$
We will be working with a category whose objects are the sets of local flows $\text{LocFlow}^\nu(\mathbb{S'}\times\mathbb{S}\times\mathcal{U})$. We shall think of this category as a subcategory of the category $\mathpzc{Top}$ of topological spaces
and continuous maps. In particular, we shall place topologies on the spaces $\text{LocFlow}^\nu(\mathbb{S'}\times\mathbb{S}\times\mathcal{U})$. This we do as follows.

Let $M$ be a $C^r$-manifold and let $\mathbb{T}\subseteq \R$ be an interval. Let $\mathbb{S},\mathbb{S'}\subseteq\mathbb{T}$ be subintervals, $\mathbb{S}\subseteq\mathbb{S'}$, and let $\mathcal{U}\subseteq M$ be open. Note that a local flow $\Phi\in \text{LocFlow}^\nu(\mathbb{S'};\mathbb{S};\mathcal{U})$ defines defines absolutely continuous curves with respect to its final time,
$$\hat{\Phi}\in \rm{AC}(\mathbb{S'};C^0(\mathbb{S};C^\nu(\mathcal{U};M)))$$
by $\hat{\Phi}(t)(t_0)(x)=\Phi(t,t_0,x)$. Then we topologise $\text{LocFlow}^\nu(\mathbb{S'};\mathbb{S};\mathcal{U})$ as follows.

First, we give $C^0(\mathbb{S};C^\nu(\mathcal{U}))$ the topology defined by the seminorms 
$$p^\nu_{K,\mathbb{I}}(g)=\sup\{p^\nu_K\circ g(t_0)\ |\ t_0\in\mathbb{I})\},$$
$\mathbb{I}\subseteq\mathbb{S}$ a compact interval, $p^\nu_K$ is the appropriate seminorm defined by (\ref{eq:3.2}) for $C^\nu(\mathcal{U})$.

Second, we give $\rm{AC}(\mathbb{S'};C^0(\mathbb{S};C^\nu(\mathcal{U})))$ the topology difined by the seminorms 
$$q^\nu_{K,\mathbb{I},\mathbb{I'}}(g)=\max\{p^\nu_{K,\mathbb{I},\mathbb{I'},\infty}(g),\ \hat{p}^\nu_{K,\mathbb{I},\mathbb{I'},1}(g)\}$$
where 
$$p^\nu_{K,\mathbb{I},\mathbb{I'},\infty}(g)=\sup\{p^\nu_{K,\mathbb{I}}\circ g(t)\ |\ t\in\mathbb{I'}\}\hspace{20pt}\text{ and }\hspace{20pt} \hat{p}^\nu_{K,\mathbb{I},\mathbb{I'},1}(g)=\int_{\mathbb{I'}}p^\nu_{K,\mathbb{I}}\circ g'(t) dt,$$
$\mathbb{I}\subseteq\mathbb{S}$, $\mathbb{I'}\subseteq\mathbb{S'}$ compact intervals. 

Finally, we give $\rm{AC}(\mathbb{S'};C^0(\mathbb{S};C^\nu(\mathcal{U};M)))$ the initial topology associated with the mappings
\begin{eqnarray*}
    \Psi_f:\rm{AC}(\mathbb{S'};C^0(\mathbb{S};C^\nu(\mathcal{U};M)))&\rightarrow& \rm{AC}(\mathbb{S'};C^0(\mathbb{S};C^\nu(\mathcal{U})))\\
       \Phi&\mapsto& f\circ\Phi,
\end{eqnarray*}
for $f\in C^\nu(M)$.

More explicitly for the topology of the space $\rm{AC}(\mathbb{S'};C^0(\mathbb{S};C^\nu(\mathcal{U})))$, given $f\in C^\nu(M)$ and $\Phi\in\text{LocFlow}^\nu(\mathbb{S'};\mathbb{S};\mathcal{U})$,
$$p_{K,\mathbb{I},\mathbb{I'},\infty}^{\nu}(f\circ \Phi)=\sup \left\{p_{K}^{\nu}(f\circ\Phi(t_1,t_0,x_0))\ |\ (t_1,t_0)\in\mathbb{I'}\times\mathbb{I}\right\}$$
and 
\begin{eqnarray*}
   \hat{p}^\nu_{K,\mathbb{I},\mathbb{I'},1}(f\circ\Phi)&:=&\int_{\mathbb{I'}} p^\nu_{K,\mathbb{I}} \left(\frac{d}{dt}(f\circ\Phi(t,t_0,x_0))\right) \ dt\\
   &=& \int_{\mathbb{I'}} p^\nu_{K,\mathbb{I}} \left(\langle df(\Phi(t,t_0,x_0)),\frac{d}{dt}\Phi(t,t_0,x_0)\rangle\right) \ dt\\
   &=& \int_{\mathbb{I'}} \sup\left\{p^\nu_K\left(\langle df(\Phi(t,t_0,x_0)),\frac{d}{dt}\Phi(t,t_0,x_0)\rangle\right)\ \bigg|\ t_0\in\mathbb{I}\right\} \ dt
\end{eqnarray*}
where $K\subseteq M$ and $\mathbb{I}\subseteq\mathbb{I'}\subseteq \mathbb{T}$ are compact. The topology defined properly above for flows is important in that they allow us to consider such spaces in the category of topological spaces, which allows us to construct the presheaf of local flows.
\begin{defin}[Morphisms of sets of local flows]
Let $m\in \mathbb{Z}_{\geq 0}$, let $m'\in\{0,\text{lip}\}$, let $\nu\in\{m+m',\infty,\omega,\text{hol}\}$, and let $r\in \{\infty,\omega,\text{hol}\}$, as required. Let $M$ and $N$ be $C^r$-manifolds, and let $\mathbb{T}\subseteq\R$ be a time interval. Let $\mathbb{S'}\times\mathbb{S}\times\mathcal{U}\subseteq\mathbb{T}\times\mathbb{T}\times M$ and $\Tilde{\mathbb{S'}}\times\Tilde{\mathbb{S}}\times\Tilde{\mathcal{U}}\subseteq\mathbb{T}\times\mathbb{T}\times N$ be flow admissible. A mapping 
$$\alpha:\text{LocFlow}^\nu(\mathbb{S'};\mathbb{S};\mathcal{U})\rightarrow \text{LocFlow}^\nu(\Tilde{\mathbb{S'}};\Tilde{\mathbb{S}};\Tilde{\mathcal{U}})$$
is a morphism of sets of $C^\nu$-local flows if 
\begin{enumerate}
    \item[(i)]{
    $\alpha$ is continuous, and
    }
    \item[(ii)]{
    there exist mappings $\tau_\alpha:\mathbb{S}\rightarrow\Tilde{\mathbb{S}}$, $\tau'_\alpha:\mathbb{S'}\rightarrow\Tilde{\mathbb{S'}}$, and $\phi_\alpha\in C^\nu(M;N)$ such that:
    \begin{enumerate}
    \item[(a)]{
    $\tau_\alpha$ and $\tau'_\alpha$ are the restrictions to $\Tilde{\mathbb{S}}$ and $\Tilde{\mathbb{S'}}$ of nonconstant affine mappings, respectively;
    }
    \item[(b)]{
    $\phi_\alpha(\mathcal{U})\subseteq\mathcal{V}$;
    }
    \item[(c)]{
    $\phi_\alpha\circ \Phi(t_1,t_0,x)=\alpha(\Phi)(\tau'_\alpha(t_1),\tau_\alpha(t_0),\phi_\alpha(x))$ for every $\Phi\in \text{LocFlow}^\nu(\mathbb{T};\mathcal{U})$ for $t_0,t_1\in\mathbb{T}$ and $x\in\mathcal{U}$.
	}
    \end{enumerate}
	}
\end{enumerate}
\end{defin}
Let us give some examples of morphisms of sets of flows.
\begin{exam}\label{exm:flow}(Morphisms of sets of local time-varying sections) 
If $N=M$, and if $\Tilde{\mathbb{S'}}\times\Tilde{\mathbb{S}}\times\Tilde{\mathcal{U}}\subseteq\mathbb{S'}\times\mathbb{S}\times\mathcal{U}$, then the ``restriction morphism"
$$\rho_{\mathbb{S'}\times\mathbb{S}\times\mathcal{U},\Tilde{\mathbb{S'}}\times\Tilde{\mathbb{S}}\times\Tilde{\mathcal{U}}}:\text{LocFlow}^\nu(\mathbb{S'};\mathbb{S};\mathcal{U})\rightarrow \text{LocFlow}^\nu(\Tilde{\mathbb{S'}};\Tilde{\mathbb{S}};\Tilde{\mathcal{U}})$$
given by 
$$\rho_{\mathbb{S'}\times\mathbb{S}\times\mathcal{U},\Tilde{\mathbb{S'}}\times\Tilde{\mathbb{S}}\times\Tilde{\mathcal{U}}}(\Phi)(t_1,t_0,x)=\Phi(t_1,t_0,x),\hspace{10pt} (t_1,t_0,x)\in \rho_{\mathbb{S'}\times\mathbb{S}\times\mathcal{U},\Tilde{\mathbb{S'}}\times\Tilde{\mathbb{S}}\times\Tilde{\mathcal{U}}}.$$
In this case we have ``$\tau_\alpha=\iota_{\mathbb{S'},\Tilde{\mathbb{S'}}}$" and ``$\phi_\alpha=\iota_{E|\mathcal{U},E|\mathcal{V}}$" where $\iota_{\mathbb{T},\mathbb{S}}$ and $\iota_{E|\mathcal{U},E|\Tilde{\mathcal{U}}}$ are the inclusions.
\end{exam}
\subsection{Presheaves of time-varying flows}
The manner in which we construct the presheaves of time-varying flows is similar to the manner in which we constructed the presheaves of vector fields in Section 4.2. For any open flow admissible subset $\mathcal{W}\subseteq \mathbb{T}\times\mathbb{T}\times M$, let $\{\mathbb{S'}_i\times\mathbb{S}_i\times \mathcal{U}_i\}_{i\in \mathbb{Z}_{>0}}$ be an open cover of $\mathcal{W}$ and is flow admissible for each $i\in\mathbb{Z}_{>0}$. We will define the presheaf of sets over $\mathcal{W}$, denoted by $\localflow^\nu(\mathbb{T};\mathbb{T};TM)(\mathcal{W})$, the subset of $\displaystyle\prod_{i\in\mathbb{Z}_{>0}}\text{LocFlow}^\nu(\mathbb{S'}_i\times\mathbb{S}_i\times\mathcal{U}_i)$ consisting of all sequences $(\Phi_i)_{i\in\mathbb{Z}_{>0}}$, $\Phi_i\in \text{LocFlow}^\nu(\mathbb{S'}_i\times\mathbb{S}_i\times\mathcal{U}_i)$ for which 
$\Phi_j|_{R}=\Phi_k|_{R}$ whenever $R=(\mathbb{S'}_j\times\mathbb{S}_j\times\mathcal{U}_j)\cap(\mathbb{S'}_k\times\mathbb{S}_k\times\mathcal{U}_k)$ for some $j,k\in\mathbb{Z}_{>0}$.
\begin{thm}
$\localflow^\nu(\mathbb{T};\mathbb{T};TM)(\mathcal{W})$ is unique up to homeomorphisms, i.e., it is independent of the choices of the covers for $\mathcal{W}$. 
\end{thm}
\begin{proof}
Same argument as Theorem \ref{thm:2.3}.
\end{proof}
\subsection{The exponential map and its  homeomorphism}
To establish the exponential map properly, we will use the language of category theory. Consider the following commutative diagram in the category of topological spaces:
\begin{center}
\begin{tikzcd}[row sep=large]
\displaystyle\prod_{i\in\mathbb{Z}_{>0}}\mathcal{V}^\nu_{\mathbb{S'}_i\times\mathbb{S}_i\times\mathcal{U}_i} \supseteq\mathscr{G}^{\nu}_{\text{LI}}(\mathbb{T};TM)(\mathcal{W}) \arrow[d,"\pi_j", start anchor={[xshift=5ex]}, end anchor={[xshift=5ex]} ] \arrow[r, "exp"] \arrow[dr, sloped, "\overbar{\text{exp}}_{\mathbb{S'}_j\times\mathbb{S}_j\times\mathcal{U}_j}\circ\pi_j", start anchor={[xshift=-7ex]}, end anchor={[xshift=-8ex]}] &\localflow^\nu(\mathbb{T};\mathbb{T};TM)(\mathcal{W})\subseteq\displaystyle\prod_{i\in\mathbb{Z}_{>0}}\text{LocFlow}^\nu(\mathbb{S'}_i\times\mathbb{S}_i\times\mathcal{U}_i) \arrow[d,"\text{pr}_j", start anchor={[xshift=-13ex]}, end anchor={[xshift=-13ex]}]
 \\
\mathcal{V}^\nu_{\mathbb{S'}_j\times\mathbb{S}_j\times\mathcal{U}_j}\supseteq Q_j
\arrow[r, "\overbar{\text{exp}}_{\mathbb{S'}_j\times\mathbb{S}_j\times\mathcal{U}_j}"]& P_j\subseteq\text{LocFlow}^\nu(\mathbb{S'}_j\times\mathbb{S}_j\times\mathcal{U}_j)   
\end{tikzcd}
\end{center}
where $\pi_j$ and $\text{pr}_j$ are the canonical projections hence continuous. Suppose that $\overbar{\text{exp}}_{\mathbb{S'}_j\times\mathbb{S}_j\times\mathcal{U}_j}$ are continuous, then by the universal property of the subspace of product spaces, there exists a well-defined continuous mapping 
\begin{eqnarray*}
exp:\displaystyle\prod_{i\in\mathbb{Z}_{>0}}\mathcal{V}^\nu_{\mathbb{S'}_i\times\mathbb{S}_i\times\mathcal{U}_i} \supseteq\mathscr{G}^{\nu}_{\text{LI}}(\mathbb{T};TM)(\mathcal{W})&\longrightarrow&\localflow^\nu(\mathbb{T};\mathbb{T};TM)(\mathcal{W}).
\end{eqnarray*}
Moreover, we have an explicit description of this map by 
$$exp((X_1,...,X_k,... ))=(\overbar{\text{exp}}_{\mathbb{S'}_1\times\mathbb{S}_1\times\mathcal{U}_1}(X_1),...,\overbar{\text{exp}}_{\mathbb{S'}_k\times\mathbb{S}_k\times\mathcal{U}_k}(X_k),...)$$

We have shown in Section \ref{sec:3.2}, that for each $(t_0,x_0)\in \mathbb{T}\times M$, there exist $\mathbb{S'},\mathbb{S}\subseteq \mathbb{T}$, $\mathcal{U}\subseteq M$ such that $(t_0,x_0)\in\mathbb{S}\times\mathcal{U}$, and an open subset $\mathcal{V}^\nu_{\mathbb{S'}\times\mathbb{S}\times\mathcal{U}}\subseteq\Gamma^\nu_{\text{LI}}(\mathbb{T};TM)$ such that the map
\begin{eqnarray}\label{eq:expbar}
  \overbar{\text{exp}}_{\mathbb{S'}\times\mathbb{S}\times\mathcal{U}}:\mathcal{V}^\nu_{\mathbb{S'}\times\mathbb{S}\times\mathcal{U}}&\rightarrow & \text{LocFlow}^\nu(\mathbb{S'};\mathbb{S};\mathcal{U})\nonumber\\
       X&\mapsto& \Phi^X
\end{eqnarray}
is well-defined. We will show that this map is homeomorphism onto its image by showing it is a continuous map with an continuous inverse. Then the continuity of $exp$ will follow by the following well-known lemma.
\begin{lem}
Let $A,\ B\ ,C\ ,D\ $ be topological spaces and let $f:A\rightarrow B$ and $g:C\rightarrow D$ be continuous maps. Then the map $f\times g: A\times C\rightarrow B\times D$ given by $$(f\times g) (x,y)= f(x)\times g(y)$$
is continuous.
\end{lem}
 We first consider the continuity. 
\subsubsection{Continuity}

\begin{propo}
Let $m\in \mathbb{Z}_{\geq 0}$, let $m'\in\{0,\text{lip}\}$, let $\nu\in\{m+m',\infty,\omega,\text{hol}\}$, and let $r\in \{\infty,\omega,\text{hol}\}$, as required. Let $M$ be a $C^r$-manifold and let $\mathbb{T}\subseteq \R$ be an interval.Let $\mathbb{S},\ \mathbb{S'}\subseteq \mathbb{T}$ and $\mathcal{U}\subseteq M$ be open. The map
\begin{eqnarray*}
  \text{exp}_{\mathbb{S'}\times\mathbb{S}\times\mathcal{U}}:\mathcal{N}\subseteq\mathcal{V}^\nu_{\mathbb{S'}\times\mathbb{S}\times\mathcal{U}}&\rightarrow & \text{LocFlow}^\nu(\mathbb{S'};\mathbb{S};\mathcal{U})\\
       X&\mapsto& \Phi^X
\end{eqnarray*}
is continuous.
\end{propo}
\begin{proof} 
Consider the following mappings
\begin{eqnarray*}
  \mathcal{N}\subseteq\mathcal{V}^\nu_{\mathbb{S'}\times\mathbb{S}\times\mathcal{U}}\xrightarrow{\text{exp}_{\mathbb{S'}\times\mathbb{S}\times\mathcal{U}}} \text{LocFlow}^\nu(\mathbb{S'};\mathbb{S};\mathcal{U})\xrightarrow{\Psi_f}\rm{AC}(\mathbb{S'};C^0(\mathbb{S};C^\nu(\mathcal{U}))).
\end{eqnarray*}
Since $\text{LocFlow}^\nu(\mathbb{S'};\mathbb{S};\mathcal{U})$ is equipped with the initial topology with $\Psi_f$, it is enough to show the continuity of $\Psi_f\circ\text{exp}$ for a fixed $f\in C^\nu(M)$.

Let $\{K_i\}_{i\in\mathbb{Z}_{>0}}\subset M$ be compact neighborhoods of $x_0$ and such that $K_{j}\subset\text{int}(K_{j+1})$ and $x_0\in \text{int}(K_1)$, and $M=\bigcup\limits_{i\in\mathbb{Z}_{>0}}K_i$. Denote $K=\bigcap\limits_{i\in \mathbb{Z}_{>0}}K_i$. Similarly, let $\{\mathbb{I}_i\}_{i\in\mathbb{Z}_{>0}}$ be compact neighborhoods of $t_0$ and such that $\mathbb{I}_{j}\subset\text{int}(\mathbb{I}_{j+1})$ and $t_0\in \text{int}(\mathbb{I}_1)$, and $\mathbb{T}=\bigcup\limits_{i\in\mathbb{Z}_{>0}}\mathbb{I}_i$.

For each $f\in C^\nu(M)$, $X\in\mathcal{V}^\nu_{\mathbb{S'}\times\mathbb{S}\times\mathcal{U}}$, let $\mathcal{R}$ be a neighborhood of $\text{exp}(X)$. Then there exist increasing sequences $\{i_1,i_2,...,i_m\}\subset \mathbb{Z}_{>0}$, $\{j'_1,j'_2,...,j'_n\}\subset \mathbb{Z}_{>0}$ and $\{j_1,j_2,...,j_n\}\subset \mathbb{Z}_{>0}$ such that $\mathbb{I}_{j_k}\subseteq\mathbb{I}_{j'_k}$ for all $k\in\{1,2,...,n\}$ and
$$\bigcap\limits^m_{k=1}\bigcap\limits^n_{l=1}\left\{\Phi\in \text{LocFlow}^\nu(\mathbb{S'};\mathbb{S};\mathcal{U})\;|\;\;q_{K_{i_k},\mathbb{I}_{j_l},\mathbb{I}_{j'_l},f}^{\nu}(\Phi-\Phi^X)<r\right\}\subseteq\mathcal{R}$$
forms a neighborhood of $\Phi^X$.

Observe, $i_a<i_b$ implies $K_{i_a}\subset K_{i_b}$ which gives 
$$q_{K_{i_b},\mathbb{I}_{j_l},\mathbb{I}_{j'_l},f}^{\nu^{-1}}([0,r))\subseteq q_{K_{i_a},\mathbb{I}_{j_l},\mathbb{I}_{j'_l},f}^{\nu^{-1}}([0,r)),$$
and $j_a<j_b$ (thus $j'_a<j'_b$ ) implies $\mathbb{I}_{j_a}\subset \mathbb{I}_{j_b}$ and $\mathbb{I}_{j'_a}\subset \mathbb{I}_{j'_b}$, which gives
$$q_{K_{i_k},\mathbb{I}_{j_b},\mathbb{I}_{j'_b},f}^{\nu^{-1}}([0,r))\subseteq q_{K_{i_k},\mathbb{I}_{j_a},\mathbb{I}_{j'_a},f}^{\nu^{-1}}([0,r)).$$
Observe that 
\[
\rotatebox{0}{$
\begin{array}{ccccccc}
\mathbb{I}_{j_1} & \subseteq & \mathbb{I}_{j_2} & \subseteq & ...&\subseteq &\mathbb{I}_{j_n}\\
\rotatebox{-90}{$\subseteq$}& &\rotatebox{-90}{$\subseteq$}& &\rotatebox{-90}{$\subseteq$}& &\rotatebox{-90}{$\subseteq$}\\[9pt]
\mathbb{I}_{j'_1} & \subseteq & \mathbb{I}_{j'_2} & \subseteq & ...&\subseteq &\mathbb{I}_{j'_n}
\end{array}
$}
\]
and 
\[
\rotatebox{0}{$
\begin{array}{ccccccc}
K_{i_1} & \subseteq & K_{i_2} & \subseteq & ...&\subseteq &K_{i_m}.
\end{array}
$}
\]
Let $K:=K_{i_m}$, $\mathbb{I}:=\mathbb{I}_{j_n}$ and $\mathbb{I'}:=\mathbb{I}_{j'_n}$, then $\mathbb{I}\subseteq\mathbb{I'}$. 
Now, let
$$Q:=\left\{\Phi\in \text{LocFlow}^\nu(\mathbb{S'};\mathbb{S};\mathcal{U})\;|\;\;q_{K,\mathbb{I},\mathbb{I'},f}^{\nu}(\Phi-\Phi^X)<r\right\}.$$
By (\ref{eq:5.19}-\ref{eq:5.23}), there exists a neighborhood $\mathcal{N}$ of $X$ and compact $K'\subseteq M$ such that $\Phi^Y(t_1,t_0,x)\in K'$ for all $(t_1,t_0,x)\in \mathbb{I'}\times\mathbb{I}\times K$ and $Y\in \mathcal{N}$, and that
$$\sup\left\{\int_{\mathbb{I'}}p^\nu_K\left(Xf(s,\Phi^{X}_{s,t_0})-Xf(s,\Phi^Y_{s,t_0})\right)\d s\ \bigg|\ t_0\in\mathbb{I}\right\} \leq \frac{r}{2}$$
for all $Y\in\mathcal{N}$. Let 
$$\mathcal{N'}:=\left\{Y\in \mathcal{V}^\nu_{\mathbb{S'}\times\mathbb{S}\times\mathcal{U}}\;|\;\;p_{K',\mathbb{I'},f}^{\nu}(Y-X)<\frac{r}{2}\right\},$$
and denote $\mathcal{O}:=\mathcal{N}\cap \mathcal{N'}$. We claim that $\text{exp}_{\mathbb{S'}\times\mathbb{S}\times\mathcal{U}}(\mathcal{O})\subseteq Q$. Indeed, for any $Y\in \mathcal{O}$,
\begin{eqnarray*}
   &&p_{K,\mathbb{I},\mathbb{I'},\infty}^{\nu}(f\circ\Phi^X-f\circ\Phi^Y)\\
   &&\hspace{15pt}=\sup\{p_K^{\nu}(f\circ(\Phi^{X}-\Phi^{Y})(t_1,t_0))\ |\ (t_1,t_0)\in \mathbb{I'}\times\mathbb{I}\}\\
   &&\hspace{15pt}= \sup \left\{p_K^{\nu}\left(\int_{|t_0,t_1|}Xf(s,\Phi^{X}_{s,t_0})-Yf(s,\Phi^Y_{s,t_0})\d s\right)\ \bigg|\ (t_1,t_0)\in\mathbb{I'}\times\mathbb{I}\right\}\\
   &&\hspace{15pt}\leq \sup \left\{p_K^{\nu}\left(\int_{|t_0,t_1|}Xf(s,\Phi^{X}_{s,t_0})-Xf(s,\Phi^Y_{s,t_0})\d s\right)\ \bigg|\ (t_1,t_0)\in\mathbb{I'}\times\mathbb{I}\right\}\\
   &&\hspace{25pt}+\sup \left\{p_K^{\nu}\left(\int_{|t_0,t_1|}Xf(s,\Phi^{Y}_{s,t_0})-Yf(s,\Phi^Y_{s,t_0}) \d s\right)\ \bigg|\ (t_1,t_0)\in\mathbb{I'}\times\mathbb{I}\right\}\\
   &&\hspace{15pt}\leq  \sup\left\{\int_{\mathbb{I'}}p^\nu_K\left(Xf(s,\Phi^{X}_{s,t_0})-Xf(s,\Phi^Y_{s,t_0})\right) \d s\ \bigg|\ t_0\in\mathbb{I}\right\}\\
   &&\hspace{25pt}+\sup\left\{p_{K'}^{\nu}\left(\int_{|t_0,t_1|}Xf(s,y)-Yf(s,y)\d s\right)\ \bigg|\ (t_1,t_0)\in\mathbb{I'}\times\mathbb{I}\right\}\\
   &&\hspace{15pt}\leq \frac{r}{2}+\int_{\mathbb{I'}}p_{K'}^{\nu}(Xf(s,y)-Yf(s,y))\d s\hspace{20pt}(\text{becasue } \mathbb{I}\subseteq\mathbb{I'})\\
   &&\hspace{15pt}\leq \frac{r}{2}+p_{K',\mathbb{I'},f}^{\nu}(X-Y). \\
   &&\hspace{15pt}<\frac{r}{2}+\frac{r}{2}\\
   &&\hspace{15pt}=r
\end{eqnarray*}
and 
\begin{eqnarray*}
   &&\hat{p}_{K,\mathbb{I},\mathbb{I'},1}^{\nu}(f\circ\Phi^X-f\circ\Phi^Y)\\
   &&=\int_{\mathbb{I'}} p^\nu_{K,\mathbb{I}} \left(\frac{d}{dt}(f\circ\Phi^X(t,t_0,x_0)-f\circ\Phi^Y(t,t_0,x_0))\right) \d t\\
   &&=\int_{\mathbb{I'}} \sup\left\{p^\nu_K\left(\langle df(\Phi^X_{t,t_0}),\frac{d}{dt}\Phi^X_{t,t_0}\rangle-\langle df(\Phi^Y_{t,t_0}),\frac{d}{dt}\Phi^Y_{t,t_0}\rangle\right)\ \bigg|\ t_0\in\mathbb{I}\right\} \d t\\
   &&=\int_{\mathbb{I'}} \sup\left\{p^\nu_K\left(Xf(t,\Phi^X_{t,t_0})- Yf(t,\Phi^Y_{t,t_0}\right)\ \bigg|\ t_0\in\mathbb{I}\right\} \d t\\
   &&\leq\int_{\mathbb{I'}} \sup\left\{p^\nu_K\left(Xf(t,\Phi^X_{t,t_0})- Xf(t,\Phi^Y_{t,t_0})\right)\ \bigg|\ t_0\in\mathbb{I}\right\} \d t\\
   &&\hspace{10pt}+\int_{\mathbb{I'}} \sup\left\{p^\nu_K\left(Xf(t,\Phi^Y_{t,t_0})- Yf(t,\Phi^Y_{t,t_0})\right)\ \bigg|\ t_0\in\mathbb{I}\right\} \d t\\
   &&\leq \frac{r}{2}+ \int_{\mathbb{I'}}p_{K'}^{\nu}(Xf(s,y)-Yf(s,y))\d s\\
   &&\leq \frac{r}{2}+p_{K',\mathbb{I'},f}^{\nu}(X-Y). \\
   &&<\frac{r}{2}+\frac{r}{2}\\
   &&=r
\end{eqnarray*}
Hence 
$$q^\nu_{K,\mathbb{I},\mathbb{I'},f}(f\circ\Phi^X-f\circ\Phi^Y)=\max\{p_{K,\mathbb{I},\mathbb{I'},\infty}^{\nu}(f\circ\Phi^X-f\circ\Phi^Y),\ \hat{p}_{K,\mathbb{I},\mathbb{I'},1}^{\nu}(f\circ\Phi^X-f\circ\Phi^Y)\}<r$$
for all $Y\in\mathcal{O}$. Therefore $\text{exp}_{\mathbb{S'}\times\mathbb{S}\times\mathcal{U}}(Y)\in Q$, whence the continuity of $\Psi_f\circ\text{exp}_{\mathbb{S'}\times\mathbb{S}\times\mathcal{U}}$ holds true.
\end{proof}
\subsubsection{Openness}
To show the openness, it is useful to consider the parameter-dependent local flows, i.e., the continuous maps
$$\mathcal{P}\ni p\mapsto \Phi^p\in \text{LocFlow}^\nu(\mathbb{S'};\mathbb{S};\mathcal{U}),$$
where $\mathcal{P}$ is a topological space. We denote this set by $\text{LocFlow}^\nu(\mathbb{S'};\mathbb{S};\mathcal{U};\mathcal{P})$.

We start by showing the following lemma which will help us to establish the continuity of the inverse of (\ref{eq:expbar}). 
\begin{lem}\label{thm:5.11}
Let $m\in \mathbb{Z}_{\geq 0}$, let $m'\in\{0,\rm{lip}\}$, let $\nu\in\{m+m',\infty,\omega,\rm{hol}\}$ satisfy $\nu\geq \rm{lip}$, and let $r\in \{\infty,\omega,\rm{hol}\}$as appropriate. Let $M$ be a $C^r$-manifold, let $\mathbb{T}\subseteq\R$ be an interval, let $\mathcal{P}$ be a topological space. Let $\mathbb{S}\subseteq\mathbb{S'}\subseteq\mathbb{T}$ and $\mathcal{U}\subseteq M$ be open. Let $\Phi^p\in\rm{LocFlow}^\nu(\mathbb{S'};\mathbb{S};\mathcal{U};\mathcal{P})$, let $f\in C^r(M)$, and let $(t_1,t_0,p_0)\in \mathbb{S'}\times\mathbb{S}\times\mathcal{P}$ be fixed and $t_0<t_1$. Then for any $\epsilon\in\R_{>0}$, there exists a neighborhood $\mathcal{O}\subseteq\mathcal{P}$ of $p_0$ and a compact $K\subseteq M$  such that $\Phi^{p}_{t,t_0}(x)\in \rm{int}(K)$ for all $(x,p)\in\mathcal{U}\times\mathcal{O}$ and that 
\begin{equation*}
    \int_{|t_0,t_1|}p^{\nu}_K\left(\frac{d}{d\tau}\bigg|_{\tau=s}f\circ\Phi^{p_0}_{\tau,t_0}\circ(\Phi^{p}_{s,t_0})^{-1}(x)-\frac{d}{d\tau}\bigg|_{\tau=s}f\circ\Phi^{p_0}_{\tau,t_0}\circ(\Phi^{p_0}_{s,t_0})^{-1}(x)\right)\d s<\epsilon \hspace{15pt}x\in K,\  p\in\mathcal{O}.
\end{equation*}
Or, equivalently, the map 
\begin{eqnarray*}
  \mathcal{O}\ni p \mapsto \frac{d}{d\tau}\big|_{\tau=s}f\circ\Phi^{p_0}_{\tau,t_0}\circ(\Phi^{p}_{s,t_0})^{-1}\in C^\nu(\mathcal{U};\R)
\end{eqnarray*}
is continuous.
\end{lem}
\begin{proof}

\textbf{(The $C^{\text{lip}}$-case)}. Now consider the following mapping 
\begin{eqnarray*}
  \Phi_{\mathbb{T}, M,\mathcal{P}}:M\times\mathcal{P} &\rightarrow & C^0(\mathbb{T};M)\\
       (x,p)&\mapsto& (t\mapsto\Phi^p(t,t_0,x)).
\end{eqnarray*}
By Theorem \ref{thm:3.4} (\ref{thm:3.4-10}), this mapping is continuous. It is obvious that
\begin{eqnarray*}
  \Phi^*_{\mathbb{T}, M,\mathcal{P}}:M\times\mathcal{P} &\rightarrow & C^0(\mathbb{T};M)\\
       (x,p)&\mapsto& (t\mapsto(\Phi_{t,t_0}^p)^{-1}(x))
\end{eqnarray*}
is also continuous. Let us prove the following lemma that will help us in our proof later. 
\begin{lem}\label{lem:4.15}
Let $f\in C^\infty(M)$ and $p_0\in\mathcal{P}$ be fixed. Define a map 
\begin{eqnarray*}
  g:\mathbb{T}\times M &\rightarrow& \R\\
       (t,x)&\mapsto& \frac{d}{d\tau}\big|_{\tau=t}f\circ\Phi^{p_0}_{\tau,t_0}(x).
\end{eqnarray*}
Then for $\gamma\in C^0(\mathbb{T};M)$, the mapping
\begin{eqnarray*}
  \Psi_{\mathbb{T},M,g}:C^0(\mathbb{T};M) &\rightarrow & L^1_{\rm{loc}}(\mathbb{T};\R)\\
       \gamma&\mapsto& (s\mapsto g(s,\gamma(s)))
\end{eqnarray*}
is well-defined and continuous.
\end{lem}
\renewcommand\qedsymbol{$\nabla$}
\begin{proof}
Denote $g_x:t\mapsto g(t,x)$ and $g_t: x\mapsto g(t,x)$. We first note that $g_x\in L^1_{\text{loc}}(\mathbb{T};\R)$ since $f\circ\Phi^{p_0}_{\tau,t_0}$ is locally absolutely continuous. We first show that $t\mapsto g(t,\gamma(t))$ is measurable on $\mathbb{T}$. Note that 
$$t\mapsto g(t,\gamma(s))$$
is measurable for each $s \in \mathbb{T}$ and that
\begin{equation}\label{eq:4.4}
    s\mapsto g(t,\gamma(s))
\end{equation}
is continuous for each $t \in \mathbb{T}$ (this since both $x\mapsto g_t(x)$ and $\gamma$ are continuous). Let $[a, b] \subseteq \mathbb{T}$ be compact, let $k \in \mathbb{Z}_{>0}$, and denote
$$t_{k,j} = a + \frac{j-1}{k}(b - a), \hspace{10pt} j \in \{1, . . . , k + 1\}.$$
Also denote 
$$\mathbb{T}_{k,j} = [t_{k,j} , t_{k,j+1}), \hspace{10pt} j \in \{1, . . . , k - 1\},$$
and $\mathbb{T}_{k,k} = [t_{k,k}, t_{k,k+1}]$. Then define $g_k : \mathbb{T} \rightarrow \R$ by
$$g_k(t)=\sum_{j=1}^{k}g(t,\gamma(t_{k,j}))\chi_{t_{k,j}}.$$
Note that $g_k$ is measurable, being a sum of products of measurable functions \citep[Proposition 2.1.7]{MR3098996}. By continuity of (\ref{eq:4.4}) for each $t \in \mathbb{T}$, we have
$$\lim_{k\to\infty}g_k(t) = g(t,\gamma(t)),\hspace{10pt} t \in [a, b],$$
showing that $t\mapsto g(t,\gamma(t))$ is measurable on $[a,b]$, as pointwise limits of measurable functions are measurable \citep[Proposition 2.1.5]{MR3098996}. Since the compact interval $[a, b] \subseteq \mathbb{T}$ is arbitrary, we conclude that $t \mapsto  g(t,\gamma(t))$ is measurable on $\mathbb{T}$.

Let $\mathbb{S} \subseteq \mathbb{T}$ be compact and let $K \subseteq M$ be a compact set for which $\gamma(\mathbb{S}) \subseteq K$. Since $g_x\in L^1_{\text{loc}}(\mathbb{T};\R)$, there exists $C\in \R_{> 0}$ be such that 
$$\int_\mathbb{S}|g(t,x)| \d t\leq C \hspace{20pt} x\in K.$$
In particular, this shows that $t \mapsto g(t,\gamma(t))$ is integrable on $\mathbb{S}$ and so locally integrable on $\mathbb{T}$. This gives the well-definedness of $\Psi_{\mathbb{T},M,g}$.

For continuity, let $\gamma_j\in C^0(\mathbb{S};M),\ j\in\mathbb{Z}_{>0}$, be a sequence of curves converging uniformly to $\gamma\in C^0(\mathbb{S};M)$. Let $\mathbb{S}\subseteq\mathbb{T}$ be a compact interval and let $K\subseteq M$ be compact. Since $\text{image}(\gamma)\cup K$ is compact and $M$ is locally compact, we can find a precompact neighbourhood $\mathcal{U}$ of $\text{image}(\gamma)\cup K$. Then for $N\in \mathbb{Z}_{>0}$ 0 sufficiently large, we have $\text{image}(\gamma_j)\subseteq\mathcal{U}$ for all $j\geq N$ by uniform convergence. Therefore, we can find a compact set $K'\subseteq M$ such that $\text{image}(\gamma_j)\subseteq K'$ for all $j\geq N$ and $\text{image}(\gamma)\subseteq K'$. Then for fixed $t\in\mathbb{S}$, continuity of $x\mapsto g(t,x)$ ensures that $\lim_{j\to\infty}g(t,\gamma_j(t))=g(t,\gamma(t))$. We also have
$$\int_\mathbb{S}|g(t,\gamma_j(t))| \d t\leq C \hspace{20pt} t\in\mathbb{S}$$
for some $C\in\R_{>0}$. Therefore, by the Dominated Convergence Theorem
$$\lim_{j\to\infty}\int_{\mathbb{S}}g(t,\gamma_j(t))\ dt=\int_{\mathbb{S}}g(t,\gamma(t))\ dt,$$
which gives the desired continuity.
\end{proof}

Denote the following continuous mapping 
\begin{eqnarray*}
  \iota_{|t_0,t_1|}:C^0(\mathbb{T};M) &\rightarrow& C^0(|t_0,t_1|;M)\\
       \gamma&\mapsto& \gamma||t_0,t_1|.
\end{eqnarray*}
Then the mapping
$$\Psi_{|t_0,t_1|,M,g}\circ \iota_{|t_0,t_1|}\circ \Phi_{\mathbb{T}, M,\mathcal{P}}: M\times \mathcal{P}\rightarrow L^1_{\text{loc}}(|t_0,t_1|;\R)$$
is continuous being a composition of continuous maps.

Let $K\subseteq M$ be compact. Then for $f\in C^\infty(M)$, $\epsilon>0$, and $x\in K$, there exist a relative neighbourhood $\mathcal{V}_x\subseteq K$ of $x$ and a neighbourhood $\mathcal{O}_x\subseteq\mathcal{O}$ of $p_0$ such that
$$\int_{|t_0,t_1|}\bigg|\frac{d}{d\tau}\big|_{\tau=s}f\circ\Phi^{p_0}_{\tau,t_0}\circ(\Phi^{p}_{s,t_0})^{-1}(x')-\frac{d}{d\tau}\big|_{\tau=s}f\circ\Phi^{p_0}_{\tau,t_0}\circ(\Phi^{p_0}_{s,t_0})^{-1}(x')\bigg|ds<\epsilon \hspace{15pt} x'\in\mathcal{V}_x, \ p\in\mathcal{O}_x.$$
Let $x_1,...,x_m\in K$ be such that $K= \cup_{j=1}^{m}\mathcal{V}_{x_j}$ and define a neighbourhood $\mathcal{O}_1=\cap_{j=1}^{k}\mathcal{O}_{x_j}$ of $p_0$. Then we have
\begin{equation}\label{eq:4.5}
    \int_{|t_0,t_1|}\bigg|\frac{d}{d\tau}\big|_{\tau=s}f\circ\Phi^{p_0}_{\tau,t_0}\circ(\Phi^{p}_{s,t_0})^{-1}(x)-\frac{d}{d\tau}\big|_{\tau=s}f\circ\Phi^{p_0}_{\tau,t_0}\circ(\Phi^{p_0}_{s,t_0})^{-1}(x)\bigg|ds<\epsilon \hspace{15pt}x\in K,\  p\in\mathcal{O}_1.
\end{equation}
Now, consider the dilatation semimetrics
$$\lambda^0_{K,f}(\Phi_1,\Phi_2)=\sup\{\text{dil}(f\circ\Phi_1-f\circ\Phi_2)(x)\ | \ x\in K\},\hspace{10pt} f\in C^\infty(M), \ K\subseteq \mathcal{U}\ \text{compact},$$
for $C^{\rm{lip}}(\mathcal{U};M)$. By Corollary \ref{cor:3.12}, we let $C\in\R_{>0}$ and let $\mathcal{O'}\subseteq\mathcal{O}$ be a neighbourhood of $p_0$ such that
$$d_{\mathbb{G}}(\Phi^X(t,t_0,x_1,p),\Phi^X(t,t_0,x_2,p))\leq Cd_{\mathbb{G}}(x_1,x_2), \hspace{10pt} t\in|t_0,t_1|,\ x_1,x_2\in K,\ p\in\mathcal{O'}.$$
Let $f\in C^\infty(M)$, let $K\subseteq\mathcal{U}$ be compact, and let $\epsilon>0$. We now estimate 
\begin{equation*}
   \int_{|t_0,t_1|}\left|\text{dil}\left(\frac{d}{d\tau}\big|_{\tau=s}f\circ\Phi^{p_0}_{\tau,t_0}\circ(\Phi^{p}_{s,t_0})^{-1}-\frac{d}{d\tau}\big|_{\tau=s}f\circ\Phi^{p_0}_{\tau,t_0}\circ(\Phi^{p_0}_{s,t_0})^{-1}\right)(x')\right|\d s. 
\end{equation*}
We shall use a strategy similar to the above for the $p^0_K$ seminorm, and borrow the notation from the computations there. Let $f\in C^\infty(M)$ and $p_0\in\mathcal{P}$ be fixed. Define a map 
\begin{eqnarray*}
  g:\mathbb{T}\times M &\rightarrow& \R\\
       (t,x)&\mapsto& \frac{d}{d\tau}\big|_{\tau=t}f\circ\Phi^{p_0}_{\tau,t_0}(x).
\end{eqnarray*}
Then for $\gamma\in C^0(\mathbb{T};M)$, the mapping
\begin{eqnarray*}
  \Psi_{|t_0,t_1|,M,g}:C^0(|t_0,t_1|;M) &\rightarrow & L^1_{\rm{loc}}(|t_0,t_1|;\R)\\
       \gamma&\mapsto& (t\mapsto \text{dil}\ g_t(\gamma(t))),
\end{eqnarray*} 
is well-defined and continuous in a manner similar to \ref{lem:4.15} merely changing absolute value for dilatation. Combining the observations of the continuous maps, the mapping
$$\Psi_{|t_0,t_1|,M,g}\circ \iota_{|t_0,t_1|}\circ \Phi_{\mathbb{T}, M,\mathcal{P}}: M\times \mathcal{P}\rightarrow L^1_{\text{loc}}(|t_0,t_1|;\R)$$
is continuous.

Let $K\subseteq M$ be compact. For $f\in C^\infty(M)$, $\epsilon>0$, and $x\in K$, there exists a relative neighbourhood $\mathcal{V}_x\subseteq K$ of $x$ and a neighbourhood $\mathcal{O}_x\subseteq\mathcal{O}$ of $p_0$ such that
\begin{eqnarray*}
&&\int_{|t_0,t_1|}\left|\text{dil}\left(\frac{d}{d\tau}\big|_{\tau=s}f\circ\Phi^{p_0}_{\tau,t_0}\circ(\Phi^{p}_{s,t_0})^{-1}-\frac{d}{d\tau}\big|_{\tau=s}f\circ\Phi^{p_0}_{\tau,t_0}\circ(\Phi^{p_0}_{s,t_0})^{-1}\right)(x')\right|\d s<\epsilon \\
&&\hspace{13cm} x'\in\mathcal{V}_x, \ p\in\mathcal{O}_x.
\end{eqnarray*}
Let $x_1,...,x_m\in K$ be such that $K= \cup_{j=1}^{m}\mathcal{V}_{x_j}$ and define a neighbourhood $\mathcal{O}_2=\cap_{j=1}^{k}\mathcal{O}_{x_j}$ of $p_0$. Then, for $x\in K$ and $p\in\mathcal{O}_2$, we have
\begin{equation}\label{eq:4.6}
   \int_{|t_0,t_1|}\left|\text{dil}\left(\frac{d}{d\tau}\big|_{\tau=s}f\circ\Phi^{p_0}_{\tau,t_0}\circ(\Phi^{p}_{s,t_0})^{-1}-\frac{d}{d\tau}\big|_{\tau=s}f\circ\Phi^{p_0}_{\tau,t_0}\circ(\Phi^{p_0}_{s,t_0})^{-1}\right)(x')\right|\d s<\epsilon. 
\end{equation}
Finally, combining the (\ref{eq:4.5}) and (\ref{eq:4.6}) for the two sorts of different seminorms for $C^{\rm{lip}}(\mathcal{U};M)$, we ascertain that, for every compact $K\subseteq\mathcal{U}$, every $f\in C^\infty(M)$, and every $\epsilon\in\R_{>0}$, if $p\in\mathcal{O}_1\cap\mathcal{O}_2$, then we ave 
$$\int_{|t_0,t_1|}p^0_K\left(\frac{d}{d\tau}\big|_{\tau=s}f\circ\Phi^{p_0}_{\tau,t_0}\circ(\Phi^{p}_{s,t_0})^{-1}-\frac{d}{d\tau}\big|_{\tau=s}f\circ\Phi^{p_0}_{\tau,t_0}\circ(\Phi^{p_0}_{s,t_0})^{-1}\right)\d s<\epsilon$$
which gives the desired result.

\vspace{10pt}\textbf{(The $C^{m}$-case)}. The topology for $C^m(\mathcal{U};M)$ is the uniform topology defined by the semimetrics 
$$d^m_{K,f}(\Phi_1,\Phi_2)=\sup\{\|j_m(f\circ\Phi_1)(x)-j_m(f\circ\Phi_2)(x)\|_{\mathbb{G}_{M,m}}\ | \ x\in K\},\  f\in C^\infty(M), \ K\subseteq \mathcal{U}\ \text{compact}.$$
Consider the mapping
\begin{eqnarray*}
  \Phi_{|t_0,t_1|,M,\mathcal{P}}:M\times\mathcal{P} &\rightarrow & C^0(|t_0,t_1|;J^m(\mathcal{U};M))\\
       (x,p)&\mapsto& (t\mapsto j_m\Phi^{X^p}_{t,t_0}(x)),
\end{eqnarray*}
which is well-defined and continuous. It is obvious that the mapping 
\begin{eqnarray*}
  \Phi^*_{|t_0,t_1|, M,\mathcal{P}}:M\times\mathcal{P} &\rightarrow & C^0(|t_0,t_1|;J^m(\mathcal{U};M))\\
       (x,p)&\mapsto& (t\mapsto j_m(\Phi^{p}_{t,t_0})^{-1}(x)),
\end{eqnarray*}
is also continuous. 

For $(x,p)\in M\times\mathcal{P}$ and for $t\in|t_0,t_1|$, we can think of $j_m(\Phi^{p}_{t,t_0})^{-1}(x)$ as a linear mapping 
\begin{eqnarray*}
  j_m(\Phi^{p}_{t,t_0})^{-1}(x):J^m(M;\R)_{(\Phi^{p}_{t,t_0})^{-1}(x)} &\rightarrow & J^m(M;\R)_x \\
       j_m g((\Phi^{p}_{t,t_0})^{-1}(x))&\mapsto& j_m (g\circ(\Phi^{p}_{t,t_0})^{-1})(x).
\end{eqnarray*}
For a fixed $f\in C^\infty(M)$ and some $\gamma\in C^0(\mathbb{T};M)$, define a map 
\begin{eqnarray*}
  g:\mathbb{T}\times M &\rightarrow& \R\\
       (s,x)&\mapsto& j_m(\frac{d}{d\tau}\big|_{\tau=s}f\circ\Phi^{p_0}_{\tau,t_0})(x).
\end{eqnarray*}
Denote $g_x: s\mapsto g(s,x)$ and $g_s: x\mapsto g(s,x)$. Then $g_x\in L^1_{\text{loc}}(\mathbb{T};\R)$ since $f\circ\Phi^{p_0}_{\tau,t_0}$ is locally absolutely continuous, and $g_s\in C^0(M,\R)$ since for a fixed $s\in\mathbb{T}$,
$$\frac{d}{d\tau}\big|_{\tau=s}f\circ\Phi^{p_0}_{\tau,t_0}\in C^m(M;\R).$$
Now, fixing $(x,p)\in M\times \mathcal{P}$ for the moment, recall the constructions of Section \ref{sec:2.3.2}, particularly those preceding the statement of Lemma \ref{lem:2.5}. We consider the notation from those constructions with
\begin{enumerate}
    \item $N=M$,
    
    \item $E=F=J^m(M;\R)$,
    
    \item $\Gamma(s)=j_m(\Phi^{p}_{s,t_0})^{-1}(x)\in \text{Hom}_{\R}(J^m(M,\R)_{(\Phi^{p}_{s,t_0})^{-1}(x)};J^m(M;\R)_x)$, and
    
    \item $\xi=j_m(\frac{d}{d\tau}\big|_{\tau=s}f\circ\Phi^{p_0}_{\tau,t_0})$.
\end{enumerate}
Thus, again in the notation from Section \ref{sec:2.3.2}, we have
$$\gamma_M(s)=(\Phi^{p}_{s,t_0})^{-1}(x),\hspace{10pt}\gamma_N(s)=x.$$
We then have the integrable section of $E=J^m(M;\R)$ given by 
\begin{eqnarray*}
  \xi_\Gamma:|t_0,t_1| &\rightarrow & E\\
       s&\mapsto& (s\mapsto j_m (\frac{d}{d\tau}\big|_{\tau=s}f\circ\Phi^{p_0}_{\tau,t_0}\circ(\Phi^{p}_{s,t_0})^{-1})(x))
\end{eqnarray*}
to obtain continuity of the mapping
\begin{eqnarray*}
  \Psi_{|t_0,t_1|,J^m(M;\R),j_m (\frac{d}{d\tau}\big|_{\tau=s}f\circ\Phi^{p_0}_{\tau,t_0})}:C^0(|t_0,t_1|;J^m(\mathcal{U};M)) \rightarrow  L^1_{\text{loc}}(|t_0,t_1|;J^m(M;\R))\\
       \Gamma\mapsto (s\mapsto \Gamma(s)(j_m(\frac{d}{d\tau}\big|_{\tau=s}f\circ\Phi^{p_0}_{\tau,t_0})(\gamma_M(s)))),
\end{eqnarray*}
and so of the composition
$$\Psi_{|t_0,t_1|,J^m(M;\R),j_m (\frac{d}{d\tau}\big|_{\tau=s}f\circ\Phi^{p_0}_{\tau,t_0})}\circ \Phi_{|t_0,t_1|, M,\mathcal{P}}: M\times\mathcal{P}\rightarrow L^1_{\text{loc}}(|t_0,t_1|;J^m(M;\R)).$$
Note that this is precisely the continuity of the mapping
$$M\times\mathcal{P}\ni (x,p)\mapsto (t\mapsto j_m(\frac{d}{d\tau}\big|_{\tau=s}f\circ\Phi^{p_0}_{\tau,t_0}\circ (\Phi^{p}_{s,t_0})^{-1}(x)))\in L^1_{\text{loc}}(|t_0,t_1|;J^m(M;\R)).$$
In order to convert this continuity into a continuity statement involving the fibre norm for $J^m(M;\R)$, we note that for $x\in K$,  there exists a neighbourhood $\mathcal{V}_x$ and affine functions $F^1_x,..., F^{n+k}_x\in\text{Aff}^\infty(J^m(M;\R))$ which are coordinates for $\rho^{-1}_m(\mathcal{V}_x)$. We can choose a Riemannian metric for $J^m(M;\R)$, whose restriction to fibres agrees with the fibre metric (\ref{eq:fibremetric}) \citep[\S 4.1]{lewis2020geometric}. It follows, therefore, from Lemma \ref{lem:5.2} that there exists $C_x\in\R_{>0}$ such that
$$\|j_m g_1(x')-j_m g_2(x')\|_{\mathbb{G}_{M,m}}\leq C_x|F^l_x\circ j_m g_1(x')-F^l_x\circ j_m g_2(x')|,$$
for $g_1,g_2\in C^\infty(M)$, $x'\in\mathcal{V}_x$, $l\in\{1,...,n+k\}$. By the continuity proved in the preceding paragraph, we can take a relative neighbourhood $\mathcal{V}_x\subseteq K$ of $x$ sufficiently small and a neighbourhood $\mathcal{O}_x\subseteq\mathcal{O}$ of $p_0$ such that 
$$\int_{|t_0,t_1|}\left|F^l_x\circ j_m(\frac{d}{d\tau}\big|_{\tau=s}f\circ\Phi^{p_0}_{\tau,t_0}\circ (\Phi^{p}_{s,t_0})^{-1})(x')-F^l_x\circ j_m(\frac{d}{d\tau}\big|_{\tau=s}f\circ\Phi^{p_0}_{\tau,t_0}\circ (\Phi^{p_0}_{s,t_0})^{-1})(x')\right|\d s<\frac{\epsilon}{2C_x},$$
for all $x'\in\mathcal{V}_x$, $p\in\mathcal{O}_x$, and $l\in\{1,...,n+k\}$. Therefore,
$$\int_{|t_0,t_1|}\left\|j_m(\frac{d}{d\tau}\big|_{\tau=s}f\circ\Phi^{p_0}_{\tau,t_0}\circ (\Phi^{p}_{s,t_0})^{-1})(x')-j_m(\frac{d}{d\tau}\big|_{\tau=s}f\circ\Phi^{p_0}_{\tau,t_0}\circ (\Phi^{p_0}_{s,t_0})^{-1})(x')\right\|_{\mathbb{G}_{M,m}} \d s <\frac{\epsilon}{2}$$
for all $x'\in\mathcal{V}_x$, $p\in\mathcal{O}_x$. Now let $x_1,...,x_s\in K$ be such that $K= \cup_{r=1}^{s}\mathcal{V}_{x_r}$ and define a neighbourhood $\mathcal{O'}=\cap_{r=1}^{s}\mathcal{O}_{x_r}$ of $p_0$. Then we have 
$$\int_{|t_0,t_1|}\left\|j_m(\frac{d}{d\tau}\big|_{\tau=s}f\circ\Phi^{p_0}_{\tau,t_0}\circ (\Phi^{p}_{s,t_0})^{-1})(x')-j_m(\frac{d}{d\tau}\big|_{\tau=s}f\circ\Phi^{p_0}_{\tau,t_0}\circ (\Phi^{p_0}_{s,t_0})^{-1})(x')\right\|_{\mathbb{G}_{M,m}} ds <\frac{\epsilon}{2}$$
for all $x'\in K$, $p\in\mathcal{O'}$, as desired.
\end{proof}

\textbf{(The $C^{m+\text{lip}}$-case).} We will note a few facts here.
\begin{enumerate}
    \item The $C^{m+\text{lip}}$ topology for $C^{m+\text{lip}}$ is the initial topology induced by the $C^{\text{lip}}$-topology for $\Gamma^{\text{lip}}(J^m(M;\R))$ under the mapping $f\mapsto j_m f$. 
    
    \item The mapping 
    $$\mathcal{O}\ni p\mapsto j_m(\frac{d}{d\tau}\big|_{\tau=s}f\circ\Phi^{p_0}_{\tau,t_0}\circ(\Phi^{p}_{s,t_0})^{-1})\in C^{\text{lip}}(\mathcal{U};J^m(M;\R))$$
    is well-defined and continuous (by the $C^m$-case proved in Section \ref{sec:3.3.2}). Thus in the diagram
    \begin{center}
    \begin{tikzcd}[
     column sep={6cm,between origins},
     row sep={3cm,between origins},]
    C^{m+\text{lip}}(\mathcal{U};\R)
    \arrow[r,"\Phi\mapsto j_m\Phi"] &C^{\text{lip}}(\rho^{-1}_m(\mathcal{U});J^m(M;\R)) 
     \\
    \mathcal{O} \arrow[u, "p\mapsto\frac{d}{d\tau}|_{\tau=s}f\circ\Phi^{p_0}_{\tau,t_0}\circ(\Phi^{p}_{s,t_0})^{-1}"]
    \arrow[ur, "p\mapsto j_m(\frac{d}{d\tau}|_{\tau=s}f\circ\Phi^{p_0}_{\tau,t_0}\circ(\Phi^{p}_{s,t_0})^{-1})"', sloped]  
    \end{tikzcd}
    \end{center}
    the continuity of the diagonal map gives the continuity of the vertical map, as desired.
\end{enumerate}

\textbf{(The $C^\infty$-case).} From the result in the $C^m$-case for $m\in\mathbb{Z}_{\geq 0}$, the mapping
\begin{eqnarray*}
  \mathcal{O}\ni p \mapsto \frac{d}{d\tau}\big|_{\tau=s}f\circ\Phi^{p_0}_{\tau,t_0}\circ(\Phi^{p}_{s,t_0})^{-1}(x)\in C^\nu(\mathcal{U};\R)
\end{eqnarray*}
is continuous for each $m\in\mathbb{Z}_{\geq 0}$. From the diagram
\begin{center}
    \begin{tikzcd}[
     column sep={6cm,between origins},
     row sep={3cm,between origins},]
    C^{\infty}(\mathcal{U};\R)
    \arrow[r,"\hookrightarrow"] &C^{m}(\mathcal{U};\R) 
     \\
    \mathcal{O} \arrow[u, "p\mapsto\frac{d}{d\tau}|_{\tau=s}f\circ\Phi^{p_0}_{\tau,t_0}\circ(\Phi^{p}_{s,t_0})^{-1}"]
    \arrow[ur, "p\mapsto j_m(\frac{d}{d\tau}|_{\tau=s}f\circ\Phi^{p_0}_{\tau,t_0}\circ(\Phi^{p}_{s,t_0})^{-1})"', sloped]  
    \end{tikzcd}
\end{center}
and noting that the diagonal mappings in the diagram are continuous, we obtain the continuity of the vertical mapping as a result of the fact that the $C^\infty$-topology is the initial topology induced by the $C^m$-topologies, $m\in\mathbb{Z}_{\geq 0}$.

\textbf{(The $C^\omega$-case).} It will suffice to show that, for $f\in C^\omega(M)$, $K\subseteq\mathcal{U}$ compact, $\boldsymbol{a}=(a_j)_{j\in\mathbb{Z}_{\geq 0}}\in c_0(\mathbb{Z}_{\geq 0};\R_{>0})$, and for $\epsilon\in\R_{>0}$, there exists a neighborhood $\mathcal{O}$ of $p_0$ such that 
\begin{eqnarray*}
   a_0a_1...a_n\left\|j_n(\frac{d}{d\tau}\big|_{\tau=s}f\circ\Phi^{p_0}_{\tau,t_0}\circ (\Phi^{p}_{s,t_0})^{-1})-\frac{d}{d\tau}\big|_{\tau=s}f\circ\Phi^{p_0}_{\tau,t_0}\circ (\Phi^{p_0}_{s,t_0})^{-1}))(x)\right\|_{G_{M,\pi,n}}<\epsilon,\\
   \hspace{80pt} x\in K, \ p\in\mathcal{O}, \ n\in\mathbb{Z}_{\geq 0}
\end{eqnarray*}
Let $f\in C^\omega(M)$, $K\subseteq\mathcal{U}$ compact, $\boldsymbol{a}=(a_j)_{j\in\mathbb{Z}_{\geq 0}}\in c_0(\mathbb{Z}_{\geq 0};\R_{>0})$, and let $\epsilon\in\R_{>0}$. As in all preceding cases, the estimate is broken into two parts.

For the first part, we let $\overbar{M}$ be a holomorphic extension of $M$ and, by Lemma \ref{lem:heotdvf}, let $\overbar{\mathcal{U}}\subseteq\overbar{M}$ be a neighbourhood of $M$ and let $\overbar{\frac{d}{d\tau}\big|_{\tau=s}\Phi^{p_0}_{\tau,t_0}}\in\Gamma^{\text{hol}}_{\text{LI}}(|t_0,t_1|;T\overbar{\mathcal{U}})$ be such that $\overbar{\frac{d}{d\tau}\big|_{\tau=s}\Phi^{p_0}_{\tau,t_0}}|M=\frac{d}{d\tau}\big|_{\tau=s}\Phi^{p_0}_{\tau,t_0}$. We also let $\overbar{f}$ be the extension of $f$, possibly after shrinking $\overbar{\mathcal{U}}$. Then the mapping
$$K\times\mathcal{O}\ni (z,p)\mapsto (t\mapsto \overbar{\frac{d}{d\tau}\big|_{\tau=s}f\circ\Phi^{p_0}_{\tau,t_0}}\circ (\Phi^{p}_{s,t_0})^{-1}(z))\in L^1_{\text{loc}}(|t_0,t_1|;\mathbb{C})$$
is continuous, just as in the $C^0$-case from Section \ref{sec:3.3.1}. Therefore, restricting to $M$,
\begin{equation}\label{eq:5.26}
    K\times\mathcal{O}\ni (x,p)\mapsto (t\mapsto \frac{d}{d\tau}\big|_{\tau=s}f\circ\Phi^{p_0}_{\tau,t_0}\circ (\Phi^{p}_{s,t_0})^{-1}(x))\in L^1_{\text{loc}}(|t_0,t_1|;\mathbb{R})
\end{equation}
is continuous. By Cauchy estimates for holomorphic sections \citep[Proposition 4.2]{Jafarpour2014}, there exists $C,r\in \R_{>0}$ such that \begin{eqnarray*}   &&\left\|j_m\left(\frac{d}{d\tau}\big|_{\tau=s}f\circ\Phi^{p_0}_{\tau,t_0}\circ (\Phi^{p}_{s,t_0})^{-1}\right)(x)-j_m\left(\frac{d}{d\tau}\big|_{\tau=s}f\circ\Phi^{p_0}_{\tau,t_0}\circ (\Phi^{p_0}_{s,t_0})^{-1}\right)(x)\right\|_{\mathbb{G}_{M,\pi,m}}\\
   &&\leq Cr^{-m}\left|\frac{d}{d\tau}\big|_{\tau=s}f\circ\Phi^{p_0}_{\tau,t_0}\circ (\Phi^{p}_{s,t_0})^{-1}(x)-\frac{d}{d\tau}\big|_{\tau=s}f\circ\Phi^{p_0}_{\tau,t_0}\circ (\Phi^{p_0}_{s,t_0})^{-1}(x)\right|, \hspace{20pt} x\in K, t\in|t_0,t_1|.
\end{eqnarray*} 
Without loss of generality, we can take $r\in (0,1)$. Let $N\in\mathbb{Z}_{\geq 0}$ be the smallest integer for which $a_j\leq r$ for $j\geq N$. Then for $m\in\mathbb{Z}_{\geq 0}$, we have 
$$Ca_0\frac{a_1}{r}\cdots\frac{a_m}{r}\leq \left\{\begin{array}{lcl} Ca_0\frac{a_1}{r}\cdots\frac{a_m}{r},\hspace{45pt}m\in\{0,1,...,N\},\\
Ca_0\frac{a_1}{r}\cdots\frac{a_N}{r},\hspace{45pt}m\geq N+1.\end{array}\right.$$
Denote 
$$M=\max\left\{Ca_0\frac{a_1}{r}\cdots\frac{a_m}{r}\ \big|\ m\in\{o,1,...,N\}\right\}$$
and for $x\in K$, let $\mathcal{V}_x\subseteq K$ be a relative neighborhood of $x$ and let $\mathcal{O}_x\subseteq\mathcal{O}$ be a
neighbourhood of $p_0$ such that
$$\int_{|t_0,t_1|}\left|\frac{d}{d\tau}\big|_{\tau=s}f\circ\Phi^{p_0}_{\tau,t_0}\circ (\Phi^{p}_{s,t_0})^{-1}(x')-\frac{d}{d\tau}\big|_{\tau=s}f\circ\Phi^{p_0}_{\tau,t_0}\circ (\Phi^{p_0}_{s,t_0})^{-1}(x')\right|\ \d s<\frac{\epsilon}{M}, \ x'\in\mathcal{V}_x, \ p\in\mathcal{O}_x,$$
this being possible by continuity of the mapping (\ref{eq:5.26}). We then have
\begin{eqnarray*}
   &&\int_{|t_0,t_1|}a_0a_1...a_m\left\|j_m\left(\frac{d}{d\tau}\big|_{\tau=s}f\circ\Phi^{p_0}_{\tau,t_0}\circ (\Phi^{p}_{s,t_0})^{-1}-\frac{d}{d\tau}\big|_{\tau=s}f\circ\Phi^{p_0}_{\tau,t_0}\circ (\Phi^{p_0}_{s,t_0})^{-1}\right)(x')\right\|_{\mathbb{G}_{M,m}}\d s\\
   &&\leq \int_{|t_0,t_1|}Ca_0\frac{a_1}{r}\cdots\frac{a_m}{r}\left|\frac{d}{d\tau}\big|_{\tau=s}f\circ\Phi^{p_0}_{\tau,t_0}\circ (\Phi^{p}_{s,t_0})^{-1}(x')-\frac{d}{d\tau}\big|_{\tau=s}f\circ\Phi^{p_0}_{\tau,t_0}\circ (\Phi^{p_0}_{s,t_0})^{-1}(x')\right|\d s<\epsilon
\end{eqnarray*}
for $x'\in\mathcal{V}_x, \ p\in\mathcal{O}_x,\ m\in\mathbb{Z}_{\geq 0}.$
Letting $x_1,...,x_s\in K$ be such that $K\subseteq \cup_{r=1}^{s}\mathcal{V}_{x_r}$ and defining $\mathcal{O'}=\cap_{r=1}^{k}\mathcal{O}_{x_r}$, we have  for $x\in K$ and $p\in\mathcal{O'}$, we have
$$\int_{|t_0,t_1|}a_0a_1...a_m\left\|j_m\left(\frac{d}{d\tau}\big|_{\tau=s}f\circ\Phi^{p_0}_{\tau,t_0}\circ (\Phi^{p}_{s,t_0})^{-1}-\frac{d}{d\tau}\big|_{\tau=s}f\circ\Phi^{p_0}_{\tau,t_0}\circ (\Phi^{p_0}_{s,t_0})^{-1}\right)(x')\right\|_{\mathbb{G}_{M,m}}\d s<\epsilon$$
for $x'\in K, \ p\in\mathcal{O'},\ m\in\mathbb{Z}_{\geq 0}$, as desired.
\begin{propo}
Let $m\in \mathbb{Z}_{\geq 0}$, let $m'\in\{0,\text{lip}\}$, let $\nu\in\{m+m',\infty,\omega,\text{hol}\}$, and let $r\in \{\infty,\omega,\text{hol}\}$, as required. Let $M$ be a $C^r$-manifold and let $\mathbb{T}\subseteq \R$ be an interval.Let $\mathbb{S}\subseteq \mathbb{S'}\subseteq \mathbb{T}$ and $\mathcal{U}\subseteq M$ be open. The map
\begin{eqnarray*}
  \text{exp}_{\mathbb{S'}\times\mathbb{S}\times\mathcal{U}}:\mathcal{N}\subseteq\mathcal{V}^\nu_{\mathbb{S'}\times\mathbb{S}\times\mathcal{U}}&\rightarrow & \text{LocFlow}^\nu(\mathbb{S'};\mathbb{S};\mathcal{U})\\
       X&\mapsto& \Phi^X
\end{eqnarray*}
is open.
\end{propo}
\begin{proof}
Since $\text{exp}_{\mathbb{S'}\times\mathbb{S}\times\mathcal{U}}$ is one-to-one and onto its image, it is enough to show the continuity of the inverse map, denoted by 
\begin{eqnarray*}
  \text{exp}^{-1}_{\mathbb{S'}\times\mathbb{S}\times\mathcal{U}}:\text{LocFlow}^\nu(\mathbb{S'};\mathbb{S};\mathcal{U})&\rightarrow &\mathcal{V}^\nu_{\mathbb{S'}\times\mathbb{S}\times\mathcal{U}}\\
       \Phi&\mapsto& X_\Phi
\end{eqnarray*}
where 
$$X_\Phi(t,x)=X_\Phi(t,\Phi(t,t_0,x_0))=\frac{d}{d\tau}\big|_{\tau=t}\Phi(\tau,t_0,x_0)=\frac{d}{d\tau}\big|_{\tau=t}\Phi_{\tau,t_0}(x_0)=\frac{d}{d\tau}\big|_{\tau=t}\Phi_{\tau,t_0}\circ\Phi^{-1}_{t,t_0}(x)$$
and $(t,t_0,x_0)\in\mathbb{S'}\times\mathbb{S}\times \mathcal{U}$. Hence $\text{exp}^{-1}_{\mathbb{S'}\times\mathbb{S}\times\mathcal{U}}\circ \text{exp}_{\mathbb{S'}\times\mathbb{S}\times\mathcal{U}}=\text{Id}$. 

The topology of $\mathcal{V}^\nu_{\mathbb{S'}\times\mathbb{S}\times\mathcal{U}}$ is generated by the family of seminorms 
$$p_{K',\mathbb{I'},f}^\nu(X)=\int_{\mathbb{I'}} p^\nu_{K'}(X_tf)\ dt,$$
where $p^\nu_{K'}$ is the appropriate seminorm from (\ref{eq:3.2}).
For a fixed $\Phi\in \text{LocFlow}^\nu(\mathbb{S'};\mathbb{S};\mathcal{U})$, $X_\Phi \in \mathcal{V}^\nu_{\mathbb{S'}\times\mathbb{S}\times\mathcal{U}}$.

Let $\{K'_i\}_{i\in\mathbb{Z}_{>0}}\subset M$ be compact neighborhoods of $x_0$ and such that $K'_{j}\subset\text{int}(K'_{j+1})$ and $x_0\in \text{int}(K'_1)$, and $M=\bigcup\limits_{i\in\mathbb{Z}_{>0}}K'_i$. Denote $K'=\bigcap\limits_{i\in \mathbb{Z}_{>0}}K'_i$. Similarly, let $\{\mathbb{I}'_i\}_{i\in\mathbb{Z}_{>0}}$ be compact neighborhoods of $t_0$ and such that $\mathbb{I}'_{j}\subset\text{int}(\mathbb{I}'_{j+1})$ and $t_0\in \text{int}(\mathbb{I}'_1)$, and $\mathbb{T}=\bigcup\limits_{i\in\mathbb{Z}_{>0}}\mathbb{I}'_i$.

For for a fixed $\Phi\in\text{LocFlow}^\nu(\mathbb{S'};\mathbb{S};\mathcal{U})$, let $\mathcal{R}$ be a neighborhood of $X_\Phi\in\mathcal{V}^\nu_{\mathbb{S'}\times\mathbb{S}\times\mathcal{U}}$. Then there exist increasing sequences $\{i_1,i_2,...,i_m\}\subset \mathbb{Z}_{>0}$, $\{j_1,j_2,...,j_n\}\subset \mathbb{Z}_{>0}$ and a finite collection of functions $f_1,f_2,...f_p\in C^\nu(M)$ such that
$$\bigcap\limits^p_{s=1}\bigcap\limits^m_{k=1}\bigcap\limits^n_{l=1}\left\{Y\in \mathcal{V}^\nu_{\mathbb{S'}\times\mathbb{S}\times\mathcal{U}}\;|\;\;p_{K'_{i_k},\mathbb{I}'_{j_l},f_s}^{\nu}(Y-X_\Phi)<r\right\}\subseteq\mathcal{R}$$
forms a neighborhood of $X_\Phi$.

Observe, $i_a<i_b$ implies $K'_{i_a}\subset K'_{i_b}$ which gives
$$p_{K'_{i_b},\mathbb{I}'_{j_l},f_s}^{\nu^{-1}}([0,r))\subseteq p_{K'_{i_a},\mathbb{I}'_{j_l},f_s}^{\nu^{-1}}([0,r)),$$
and $j_a<j_b$ implies $\mathbb{I}'_{j_a}\subset \mathbb{I}'_{j_b}$, which gives
$$p_{K'_{i_k},\mathbb{I}'_{j_b},f_s}^{\nu^{-1}}([0,r))\subseteq p_{K'_{i_k},\mathbb{I}'_{j_a},f_s}^{\nu^{-1}}([0,r)).$$
Observe that $\mathbb{I}'_{j_1} \subseteq \mathbb{I}'_{j_2} \subseteq  ...\subseteq \mathbb{I}'_{j_n}$ and $K'_{i_1} \subseteq K'_{i_2} \subseteq ...\subseteq K'_{i_m}$. Denote $K':=K'_{i_m}$ and $\mathbb{I}':=\mathbb{I}'_{j_n}$, and let
$$Q:=\bigcap\limits^p_{s=1}\left\{Y\in \mathcal{V}^\nu_{\mathbb{S'}\times\mathbb{S}\times\mathcal{U}}\;|\;\;p_{K',\mathbb{I}',f_s}^{\nu}(Y-X_\Phi)<r\right\}.$$

The topology on topology on $\text{LocFlow}^\nu(\mathbb{S'};\mathbb{S};\mathcal{U})$ by the semi-metrics
$$q_{K,\mathbb{I},\mathbb{I'},f}^{\nu}(\Phi_1-\Phi_2)=\max \left\{p_{K,\mathbb{I},\mathbb{I'},\infty}^{\nu}(f\circ(\Phi_1-\Phi_2)),\  \hat{p}^\nu_{K,\mathbb{I},\mathbb{I'},1}(f\circ(\Phi_1-\Phi_2)) \right\}.$$
By Proposition (\ref{thm:5.11}), for a fixed $f_j\in\{f_1,...,f_p\}$, there exist $K_j\subseteq M$ and $\mathbb{I}_j\subseteq \mathbb{I'}$ compact with $t_0$ as interior, and a neighborhood $\mathcal{N}_j$ of $\Phi$, such that $\Psi^{-1}_{t,t_0}(x)\in \text{int}(K_j)$ for all $\Psi\in\mathcal{N}_j$ and $(t,t_0,x)\in \mathbb{I'}\times\mathbb{I}_j\times K'$ and that
\begin{equation}
\int_{\mathbb{I'}}p^{\nu}_{K'}\left(\frac{d}{d\tau}\big|_{\tau=s}f_j\circ\Psi_{\tau,t_0}\circ\Psi^{-1}_{s,t_0}(x)-\frac{d}{d\tau}\big|_{\tau=s}f_j\circ\Psi_{\tau,t_0}\circ\Phi^{-1}_{s,t_0}(x)\right)\d s<\frac{r}{2}.
\end{equation}
Now denote $K:=\bigcap\limits^p_{j=1}K_j$, $\mathbb{I}:=\bigcap\limits^p_{j=1} \mathbb{I}_j$.
Let 
$$\mathcal{N'}:=\bigcap\limits^p_{j=1}\left\{\Psi\in\text{LocFlow}^\nu(\mathbb{S'};\mathbb{S};\mathcal{U})\;|\;\;q_{K,\mathbb{I},\mathbb{I'},f_j}^{\nu}(\Psi-\Phi)<\frac{r}{2}\right\},$$
and denote $\mathcal{O}:=(\bigcap\limits^p_{j=1}\mathcal{N}_j)\cap \mathcal{N'}$. We claim that $\text{exp}^{-1}_{\mathbb{S'}\times\mathbb{S}\times\mathcal{U}}(\mathcal{O})\subseteq Q$. Indeed, for any $\Psi\in \mathcal{O}$ and for fixed $(t,t_0)\in\mathbb{I'}\times\mathbb{I}$, we have that for each $s\in\{1,...,p\}$,
\begin{eqnarray*}
   &&p^\nu_{K'}\left(X_\Psi f_s(t,x)-X_\Phi f_s(t,x)\right)\\
   &&=p^\nu_{K'}\left(\left(\frac{d}{d\tau}\bigg|_{\tau=t}\Psi_{\tau,t_0}\circ\Psi^{-1}_{t,t_0}-\frac{d}{d\tau}\bigg|_{\tau=t}\Phi_{\tau,t_0}\circ\Phi^{-1}_{t,t_0}\right)f_s\right)\\
   &&\leq p^\nu_{K'}\left(\left(\frac{d}{d\tau}\bigg|_{\tau=t}\Psi_{\tau,t_0}(\Psi^{-1}_{t,t_0}(x))-\frac{d}{d\tau}\bigg|_{\tau=t}\Psi_{\tau,t_0}(\Phi^{-1}_{t,t_0}(x))\right)f_s\right)\\
   &&\hspace{10pt}
   +p^\nu_{K'}\left(\left(\frac{d}{d\tau}\bigg|_{\tau=t}\Psi_{\tau,t_0}(\Phi^{-1}_{t,t_0}(x))-\frac{d}{d\tau}\bigg|_{\tau=t}\Phi_{\tau,t_0}(\Phi^{-1}_{t,t_0}(x))\right)f_s\right)\\
   &&\leq p^\nu_{K'}\left(\frac{d}{d\tau}\bigg|_{\tau=t}f_s\circ\Psi_{\tau,t_0}(\Psi^{-1}_{t,t_0}(x))-\frac{d}{d\tau}\bigg|_{\tau=t}f_s\circ\Psi_{\tau,t_0}(\Phi^{-1}_{t,t_0}(x))\right)\\
   &&\hspace{10pt}
   +p^\nu_{K}\left(\frac{d}{d\tau}\bigg|_{\tau=t}f_s\circ\Psi_{\tau,t_0}(y)-\frac{d}{d\tau}\bigg|_{\tau=t}f_s\circ\Phi_{\tau,t_0}(y)\right),
\end{eqnarray*}
which implies
\begin{eqnarray*}
   &&p_{K',\mathbb{I'},f_s}^{\nu}(X_\Psi-X_\Phi)\\
   &&=\int_{\mathbb{I'}} p^\nu_{K'}\left(X_\Psi f_s(t,x)-X_\Phi f_s(t,x)\right)\d t\\
   &&\leq \int_{\mathbb{I'}}p^\nu_{K'}\left(\frac{d}{d\tau}\bigg|_{\tau=t}f_s\circ\Psi_{\tau,t_0}(\Psi^{-1}_{t,t_0}(x))-\frac{d}{d\tau}\bigg|_{\tau=t}f_s\circ\Psi_{\tau,t_0}(\Phi^{-1}_{t,t_0}(x))\right)\d t\\
   &&\hspace{10pt}+\int_{\mathbb{I'}}p^\nu_{K}\left(\frac{d}{d\tau}\bigg|_{\tau=t}f_s\circ\Psi_{\tau,t_0}(y)-\frac{d}{d\tau}\bigg|_{\tau=t}f_s\circ\Phi_{\tau,t_0}(y)\right)\d t\\
   &&\leq \frac{r}{2}+\int_{\mathbb{I'}}\sup\left\{p^\nu_{K}\left(\frac{d}{d\tau}\bigg|_{\tau=t}f_s\circ\Psi_{\tau,t_0}(y)-\frac{d}{d\tau}\bigg|_{\tau=t}f_s\circ\Phi_{\tau,t_0}(y)\right)\ \bigg| t_0\in \mathbb{I}\right\}dt\\
   &&=\frac{r}{2}+\hat{p}^\nu_{K,\mathbb{I},\mathbb{I'},1}(f_s\circ\Psi-f_s\circ\Phi)\\
   &&\leq \frac{r}{2}+ q_{K,\mathbb{I},\mathbb{I'},f_s}^{\nu}(\Psi-\Phi)\\
   &&\leq \frac{r}{2}+\frac{r}{2}=r,
\end{eqnarray*}
as desired.
\end{proof}

\appendix
\section{Riemannian metrics}
The following results are used for the proof of different theorems in this paper. 

\begin{lem}[Comparison of Riemannian distance for different Riemannian metrics)]\label{lem:5.2}    
If $\mathbb{G}_1$ and $\mathbb{G}_2$ are smooth Riemannian metrics on $M$ with metrics $d_1$ and $d_2$, respectively, and if $K \subseteq M$ is compact, then there exists $c \in \R_{>0}$ such that
$$c^{-1}d_1(x_1, x_2) \leq d_2(x_1, x_2) \leq cd_1(x_1, x_2)$$
for every $x_1, x_2 \in K$.
\end{lem}
\begin{proof}
We shall prove the result in increments. The first step is simple linear algebra.
\begin{sublem}
If $\mathbb{G}_1$ and $\mathbb{G}_2$ are inner products on a finite-dimensional $\R$-vector space $V$, then there exists $c \in \R_{>0}$ such that
$$c^{-1}\mathbb{G}_1(v, v) \leq \mathbb{G}_2(v, v) \leq c\mathbb{G}_1(v, v)$$
for all $v\in V$.
\end{sublem}
\renewcommand\qedsymbol{$\nabla$}
\begin{proof}
Let $\mathbb{G}^{\fl}_j \in \text{Hom}_\R(V; V^*)$ and $\mathbb{G}^{\sh}_j \in \text{Hom}_\R(V^*; V)$, $j \in \{1, 2\}$, be the induced linear maps. Note that
$$\mathbb{G}_1(\mathbb{G}^{\sh}_1\circ \mathbb{G}^{\fl}_2(v_1), v_2) = \mathbb{G}_2(v_1, v_2) = \mathbb{G}_2(v_2, v_1) = \mathbb{G}_1(\mathbb{G}^{\sh}_1\circ \mathbb{G}^{\fl}_2(v_2), v_1),$$
showing that $\mathbb{G}^{\sh}_1\circ \mathbb{G}^{\fl}_2$
is $\mathbb{G}_1$-symmetric. Let $(e_1, . . . , e_n)$ be a $\mathbb{G}_1$-orthonormal basis for $V$ that is also a basis of eigenvectors for $G^{\sh}_1\circ \mathbb{G}^{\fl}_2$. The matrix representatives of $\mathbb{G}_1$ and $\mathbb{G}_2$ are then
$$[\mathbb{G}_1]=
\begin{bmatrix}
1 & 0 & \cdots& 0 \\
0 & 1 & \cdots & 0\\
\vdots & \vdots &  \ddots & \vdots\\
0 & 0 & \cdots& 1
\end{bmatrix}, \hspace{20pt}[\mathbb{G}_2]=
\begin{bmatrix}
\lambda_1 & 0 & \cdots& 0 \\
0 & \lambda_2 & \cdots & 0\\
\vdots & \vdots &  \ddots & \vdots\\
0 & 0 & \cdots& \lambda_n
\end{bmatrix},$$
where $\lambda_1, . . . , \lambda_n \in \R_{>0}$. Let us assume without loss of generality that $\lambda_1< \cdots < \lambda_n$. Then taking $c = \max\{\lambda_n, \lambda^{-1}_1\}$ gives the result, as one can verify directly.
\end{proof}
Next let us give the local version of the result.
\begin{sublem}
Let $\mathbb{G}_1$ and $\mathbb{G}_2$ be smooth Riemannian metrics on a manifold $M$ with metrics $d_1$ and $d_2$, respectively. For each $x \in M$, there exists a neighbourhood $\mathcal{U}_x$ of $x$ and $c_x \in \R_{>0}$ such that
$$c^{-1}_xd_1(x_1, x_2) \leq d_2(x_1, x_2) \leq c_xd_1(x_1, x_2)$$
for every $x_1, x_2 \in \mathcal{U}_x$.
\end{sublem}
\begin{proof}
Let $x \in M$. Let $\mathcal{N}_1$ and $\mathcal{N}_2$ be geodesically convex neighbourhoods of $x$ with respect to the Riemannian metrics $\mathbb{G}_1$ and $\mathbb{G}_2$, respectively [Kobayashi and Nomizu 1963, Proposition IV.3.4]. Thus every pair of points in $\mathcal{N}_1$ can be connected by a unique distanceminimising geodesic for $\mathbb{G}_1$ that remains in $\mathcal{N}_1$, and similarly with $\mathcal{N}_2$ and $\mathbb{G}_2$. By Sublemma 1, let $c_x \in \R_{>0}$ be such that
$$c^{-2}_x\mathbb{G}_1(v_x, v_x) \leq \mathbb{G}_2(v_x, v_x) \leq c_x^2\mathbb{G}_1(v_x, v_x), \hspace{10pt} v_x\in T_xM. $$
By continuity of $\mathbb{G}_1$ and $\mathbb{G}_2$, we can choose $\mathcal{N}_1$ and $\mathcal{N}_2$ sufficiently small that
$$c^{-2}_x\mathbb{G}_1(v_y, v_y) \leq \mathbb{G}_2(v_y, v_y) \leq c_x^2\mathbb{G}_1(v_y, v_y), \hspace{10pt} v_y\in \mathcal{N}_1\cup\mathcal{N}_1.$$
Now define $\mathcal{U}_x = \mathcal{N}_1\cap\mathcal{N}_1$. Then every pair of points in $\mathcal{U}_x$ can be connected with a unique distance-minimising geodesic of both $\mathbb{G}_1$ and $\mathbb{G}_2$ that remains in $\mathcal{N}_1\cup\mathcal{N}_1$. Now let $x_1, x_2 \in \mathcal{U}_x$. Let $\gamma : [0, 1] \to M$ be the unique distance-minimising $\mathbb{G}_1$-geodesic connecting $x_1$ and $x_2$. Then
\begin{eqnarray*}
  d_2(x_1, x_2) &\leq& \ell_{\mathbb{G}_2}(\gamma) = \int_0^1 \sqrt{\mathbb{G}_2(\gamma'(t),\gamma'(t) )} \d t\\
  &\leq& c_x\int_0^1 \sqrt{\mathbb{G}_1(\gamma'(t),\gamma'(t) )} \d t\\
  &\leq& c_x \ell_{\mathbb{G}_1}(\gamma)= c_xd_1(x_1,x_2).
\end{eqnarray*}
One similarly shows that $d_1(x_1, x_2) \leq c_xd_2(x_1, x_2)$.
\end{proof}
\renewcommand\qedsymbol{$\blacksquare$}
Now let $K \subseteq M$ be compact and, for each $x \in K$, let $\mathcal{U}_x$ be a neighbourhood of $x$ and let $c_x \in \R_{>0}$ be as in the preceding sublemma. Then $(\mathcal{U}_x)_{x\in K}$ is an open cover of $K$. Let $x_1, . . . , x_k \in K$ be such that
$$K \subseteq \cup^k_{j=1}\mathcal{U}_{x_j}.$$
Let
$$D_a = \sup\{d_a(x, y)\ \suchthat\  x, y \in K\},\hspace{10pt} a \in \{1, 2\}.$$
By the Lebesgue Number Lemma \citep[Theorem 1.6.11]{MR1835418}, let $r_a \in \R_{>0}$ be such that, if $x \in K$, then there exists $j \in \{1, . . . , k\}$ for which $B_a(r, x) \in \mathcal{U}_{x_j}$ ($B_a(r, x)$ is the ball with respect to the metric $d_a$). Let us denote
$$c=\max\left\{c_{x_1},...,c_{x_k},\frac{D_1}{r_2},\frac{D_2}{r_1}\right\}.$$
Now let $x_1, x_2 \in K$. If $d_1(x_1, x_2) < r_1$, then let $j \in \{1, . . . , k\}$ be such that $x_1, x_2 \in \mathcal{U}_j$. Then
$$d_2(x_1, x_2) \leq cd_1(x_1, x_2).$$
If $d_1(x_1, x_2) \geq r_1$, then
$$\frac{d_2(x_1,x_2)r_1}{D_2}\leq \frac{d_2(x_1,x_2)r_1}{d_2(x_1,x_2)}\leq d_1(x_1,x_2).$$
This gives $d_2(x_1, x_2) \leq cd_1(x_1, x_2)$. Swapping the roles of $\mathbb{G}_1$ and $\mathbb{G}_2$ gives $d_1(x_1, x_2) \leq cd_2(x_1, x_2)$, giving the lemma.
\end{proof}
\section{Holomorphic extensions}
\begin{lem}[Holomorphic extensions of $C^\omega$ time-dependent vector fields]\label{lem:heotdvf}
Let $\pi : E \to M$ be a $C^\omega$-vector bundle with $C^{\text{hol}}$-extension $\overbar{\pi}: \overbar{E} \to \overbar{M}$, let $\mathbb{T} \subseteq \R$ be an interval, and let $\xi \in \Gamma^\omega_{\text{LI}}(\mathbb{T}; E)$. Then, for every compact subinterval $\mathbb{S} \subseteq \mathbb{T}$, there exists a neighbourhood $\overbar{\mathcal{U}} \subseteq \overbar{M}$ of $M$ and $\overbar{\xi} \in L^1(\mathbb{S}; \Gamma^{\text{hol},\R}(\overbar{E}|\overbar{\mathcal{U}}))$ such that $\overbar{\xi}(t, x) = \xi(t, x)$ for every $(t, x) \in \mathbb{S} \times M$. 
\end{lem}
\begin{proof}
First of all, following \citep[Proposition 1]{MR102094} there is an holomorphic extension $\overbar{\pi}: \overbar{E} \to \overbar{M}$ of $\pi : E \to M$. This means that $M$ is a totally real submanifold of maximal dimension of $\overbar{M}$ and that $E$ is a totally real subbundle of $\overbar{E}$ of maximal dimension over the submanifold $M$. Let $\mathcal{N}_M$ be the set of neighbourhoods in $\overbar{M}$ of $M$ and denote by $\mathscr{G}^{\text{hol},\R}_{M,\overbar{E}}$ the set of germs of holomorphic sections about $M$ whose restriction to $M$ takes values in $E$. If $\overbar{\xi}$ is such an holomorphic section defined on some neighbourhood, we denote its germ by $[\overbar{\xi}]_M$. We provide $\mathcal{N}_M$ with the partial order $\mathcal{U} \preceq \mathcal{V}$ if $\mathcal{V} \subseteq \mathcal{U}$. Given $\mathcal{U} \preceq \mathcal{V}$, we
have a natural restriction mapping
\begin{eqnarray*}
  r_{\overbar{\mathcal{U}},\overbar{\mathcal{V}}}: \overbar{E}|\overbar{\mathcal{U}}&\to&\overbar{E}|\overbar{\mathcal{V}}\\
  \overbar{\xi} &\mapsto& \overbar{\xi}|\overbar{\mathcal{V}}.
\end{eqnarray*}
In this way we arrive at a directed system
$$\left(\Gamma^{\text{hol},\R}(\overbar{E}|\overbar{\mathcal{U}})_{\overbar{\mathcal{U}}\in\mathcal{N}_M},(r_{\overbar{\mathcal{U}},\overbar{\mathcal{V}}})_{\mathcal{U} \preceq \mathcal{V}}\right)$$
of locally convex $\R$-topological vector spaces. Algebraically, the direct limit is isomorphic to $\mathscr{G}^{\text{hol},\R}_{M,\overbar{E}}$, and this then gives $\mathscr{G}^{\text{hol},\R}_{M,\overbar{E}}$ the locally convex direct limit. Since every section $\xi \in \Gamma^\omega(E)$ possesses an extension $\overbar{\xi}\in \Gamma^{\text{hol},\R}(\overbar{E}|\overbar{\mathcal{U}})$ to some neighbourhood $\overbar{\mathcal{U}}$, it is fairly easy to show that the mapping
\begin{eqnarray*}
  \iota_M: \Gamma^\omega(E)&\to&\mathscr{G}^{\text{hol},\R}_{M,\overbar{E}}\\
  \xi &\mapsto& [\overbar{\xi}]|M,
\end{eqnarray*}
where $\overbar{\xi}$ is some holomorphic extension of $\xi$, is an isomorphism of $\R$-vector spaces \citep[Lemma 5.2]{Jafarpour2014}. In this way, we give $\Gamma^\omega(E)$ the topology induced from that of
$\mathscr{G}^{\text{hol},\R}_{M,\overbar{E}}$. This topology is, in fact, the $C^\omega$-topology described in Section \ref{sec:sfras}, although this is by no means easy to show. We refer to \citep{Jafarpour2014} for details.

From our description of the $C^\omega$-topology above, we have
$$\Gamma^\omega(E) \simeq \lim_{\overset{\longrightarrow}{\overbar{\mathcal{U}}\in\mathcal{N}_M}} 
\Gamma^{\text{hol},\R}(\overbar{E}|\overbar{\mathcal{U}})$$

Let $\mathbb{S} \subseteq \mathbb{T}$ be compact.  As per \citep[Theorem 3.2]{lewis2021geometric}, we have
\begin{eqnarray*}
  L^1(\mathbb{S}; \Gamma^\omega(E)) &\simeq& L^1(\mathbb{S}; \R)\overbar{\otimes}_\pi \Gamma^\omega(E),\\
  L^1(\mathbb{S}; \Gamma^{\text{hol},\R}(\overbar{E}|\overbar{\mathcal{U}})) &\simeq& L^1(\mathbb{S}; \R)\overbar{\otimes}_\pi \Gamma^{\text{hol},\R}(\overbar{E}|\overbar{\mathcal{U}}).
\end{eqnarray*}
Here $\otimes_\pi$ is the projective tensor product and $\overbar{\otimes}_\pi$ denotes the completion of the projective tensor product \citep[Chapter 15]{MR632257}. It then follows by\citep[Corollary 15.5.4(a)]{MR632257} that
$$L^1(\mathbb{S}; \R)\otimes_\pi \Gamma^\omega(E)\simeq \lim_{\overset{\longrightarrow}{\overbar{\mathcal{U}}}}L^1(\mathbb{S}; \R)\otimes_\pi \Gamma^{\text{hol},\R}(\overbar{E}|\overbar{\mathcal{U}})$$
We wish to assert the same conclusion for the completion of the projective tensor product.
For this, we use a sublemma.
\begin{sublem}
Let $((V_i)_{i\in I} ,(\phi_{ij} )_{i \preceq j} )$ be a directed system of locally convex topological vector spaces with locally convex direct limit $V$ and let $B$ be a Banach space. Then the locally convex direct limit of the directed system
$$((B\overbar{\otimes}_\pi V_i)_{i\in I} ,(\text{id}_B \overbar{\otimes}_\pi \phi_{ij} )_{i \preceq j} )$$
is isomorphic to $B\overbar{\otimes}_\pi V$.
\end{sublem}
\renewcommand\qedsymbol{$\nabla$}
\begin{proof}
By \citep[Corollary 15.5.4]{MR632257}, the direct limit of $(B\otimes_\pi V_i)_{i\in I}$ is isomorphic to $B\otimes_\pi V$, giving mappings $\text{id}B \otimes_\pi V_i \in L(B \otimes_\pi V_i; B \otimes_\pi V)$. By unique extension, we have a continuous linear mapping $\text{id}_B \overbar{\otimes}_\pi \phi_i \in L(B \overbar{\otimes}_\pi V_i; B \overbar{\otimes}_\pi V)$. We then have
$$B \otimes_\pi V_i\subseteq (\text{id}_B \overbar{\otimes}_\pi \phi_i)^{-1}(B \otimes_\pi V).$$
Immediately, upon taking closures,
$$B \overbar{\otimes}_\pi V_i\subseteq \text{cl}((\text{id}_B \overbar{\otimes}_\pi \phi_i)^{-1}(B \otimes_\pi V)).$$
We also have
$$(\text{id}_B \overbar{\otimes}_\pi \phi_i)^{-1}(B \otimes_\pi V)\subseteq (\text{id}_B \overbar{\otimes}_\pi \phi_i)^{-1}(\text{cl}( B \otimes_\pi V)).$$
Since the set on the right is closed, we have
$$\text{cl}((\text{id}_B \overbar{\otimes}_\pi \phi_i)^{-1}(B \otimes_\pi V))\subseteq (\text{id}_B \overbar{\otimes}_\pi \phi_i)^{-1}(\text{cl}( B \otimes_\pi V))=B \overbar{\otimes}_\pi V_i.$$
That is to say, $(\text{id}_B \overbar{\otimes}_\pi \phi_i)^{-1}(B \otimes_\pi V)$ is dense in $B \overbar{\otimes}_\pi V_i$. By \citep[Corollary 1]{Hustad1963} we conclude that $B\otimes_\pi V$ is dense in the direct limit of the directed system $(B\overbar{\otimes}_\pi V_i)_{i\in I}$. This
then gives the lemma since the closure of $B\otimes_\pi V$ is $B\overbar{\otimes}_\pi V$.
\end{proof}
\renewcommand\qedsymbol{$\blacksquare$}
By the sublemma, we have
$$L^1(\mathbb{S}; \Gamma^\omega(E))\simeq \lim_{\overset{\longrightarrow}{\overbar{\mathcal{U}}}}L^1(\mathbb{S}; \Gamma^{\text{hol},\R}(\overbar{E}|\overbar{\mathcal{U}})).$$
Therefore, there exists $\overbar{\mathcal{U}}\in\mathcal{N}_M$ and $\overbar{\xi}\in L^1(\mathbb{S}; \Gamma^{\text{hol},\R}(\overbar{E}|\overbar{\mathcal{U}}))$ such that $\xi(t,x)=\overbar{\xi}(t,x)$ for $(t,x)\in \mathbb{S}\times M$.
\end{proof}

\newpage

\printbibliography[heading=bibintoc]

\end{document}